%% file: GanMejSch24.tex
\documentclass[10pt]{article}
\input{header}

\title{Vertex characterization via second-order topological derivatives}
\author{Peter Gangl$^1$\\{\footnotesize\href{mailto:email}{peter.gangl@ricam.oeaw.ac.at}}
\and Bochra Mejri$^{1}$\\{\footnotesize\href{mailto:email}{bochra.mejri@ricam.oeaw.ac.at}}
\and Otmar Scherzer$^{1,2,3}$\\{\footnotesize\href{mailto:email}{otmar.scherzer@univie.ac.at}}}
\date{\today}
\usepackage{float,bm,xcolor,colortbl} 
\usepackage{caption,subcaption,tikz}
\usetikzlibrary{angles,quotes}
\usetikzlibrary{fit,calc,positioning}
\usetikzlibrary{positioning}
\usetikzlibrary{decorations.pathreplacing}
\usetikzlibrary{shapes.geometric, arrows}
\usepackage{subcaption}

\usepackage{hyperref} 
\usepackage{algorithm}
\usepackage{algpseudocode}

\usetikzlibrary{calc,shapes}
\newcommand\sqw{1}
\tikzset{
      pics/square/.default={\sqw},
      pics/square/.style = {
        code = {
        \draw[pic actions,draw=none] (0,0) rectangle (#1,#1);
    }}}   
\usepackage{esint}
\usetikzlibrary{fadings}
\tikzfading[name=fade out,
inner color=transparent!0,
outer color=transparent!100]
\newcommand{\dx}{\,\mbox{d}x}

\newcommand{\ue}{u_\varepsilon}
\definecolor{ao(english)}{rgb}{0.0, 0.5, 0.0}
\definecolor{gray(x11gray)}{rgb}{0.75, 0.75, 0.75}
\definecolor{slategray}{rgb}{0.44, 0.5, 0.56}
\definecolor{antiquefuchsia}{rgb}{0.57, 0.36, 0.51}
\definecolor{deepfuchsia}{rgb}{0.76, 0.33, 0.76}
\usepackage{algpseudocode}
\newif\ifpics
\picstrue   %put \picstrue for including pictures
\newif\ifpicsold
\begin{document}

%'
%' titlepage
%'
\maketitle
\thispagestyle{empty}
% ADAPT
\begin{center}
\hspace*{5em}
\parbox[t]{17em}{\footnotesize
\hspace*{-1ex}$^1$Johann Radon Institute for Computational\\
and Applied Mathematics (RICAM)\\
Altenbergerstraße 69\\
A-4040 Linz, Austria}
\hfil
\parbox[t]{12em}{\footnotesize
\hspace*{-1ex}$^2$Faculty of Mathematics\\
University of Vienna\\
Oskar-Morgenstern-Platz 1\\
A-1090 Vienna, Austria}
\end{center}

% if publication is from CD-LAB MaMSi
\begin{center}
\parbox[t]{19em}{\footnotesize
\hspace*{-1ex}$^3$Christian Doppler Laboratory\\
for Mathematical Modeling and Simulation\\
of Next Generations of Ultrasound Devices (MaMSi)\\
Oskar-Morgenstern-Platz 1\\
A-1090 Vienna, Austria}
\end{center}

\begin{abstract}
This paper focuses on identifying vertex characteristics in 2D images using topological asymptotic analysis. Vertex characteristics include both the location and the type of the vertex, with the latter defined by the number of lines forming it and the corresponding angles. This problem is crucial for computer vision tasks, such as distinguishing between fore- and background objects in 3D scenes. We compute the second-order topological derivative of a Mumford-Shah type functional with respect to inclusion shapes representing various vertex types. This derivative assigns a likelihood to each pixel that a particular vertex type appears there. Numerical tests demonstrate the effectiveness of the proposed approach.
\end{abstract}

%%%%%%%%%%%%%%%%%%%%%%%%%%%%%%%%
\section{Introduction}
\label{introduction}
%%%%%%%%%%%%%%%%%%%%%%%%%%%%%%%%
Vertex characterization plays a crucial role in computer vision for interpreting 3D scenes from 2D images, as discussed by Guzm\'{a}n~\cite{Guz68}. This problem has been approached in the psychology literature such as Clowes~\cite{Clo71}, Kanizsa~\cite{Kan76}, and Cavanagh~\cite{Cav87} and also in the vision literature, see e.g. Nitzberg {\it{et al.}}~\cite{NitMumShi93b}, among others. An important conclusion from this research is that different classes of vertices in 2D images can provide detailed information on the 3D scene. For example, an 'L-corner' indicates an object's corner~\cite{MehNicRan90}, a 'T-junction' represents overlapping objects~\cite{NitMumShi93b}, and an 'X-junction' suggests the occurrence of transparencies~\cite{WatCav93}. Generally, vertices are classified into eight classes, as shown in Fig.~\ref{fig_classification_vertices}; for more details, see~\cite{Guz68, Clo71}.
\ifpics
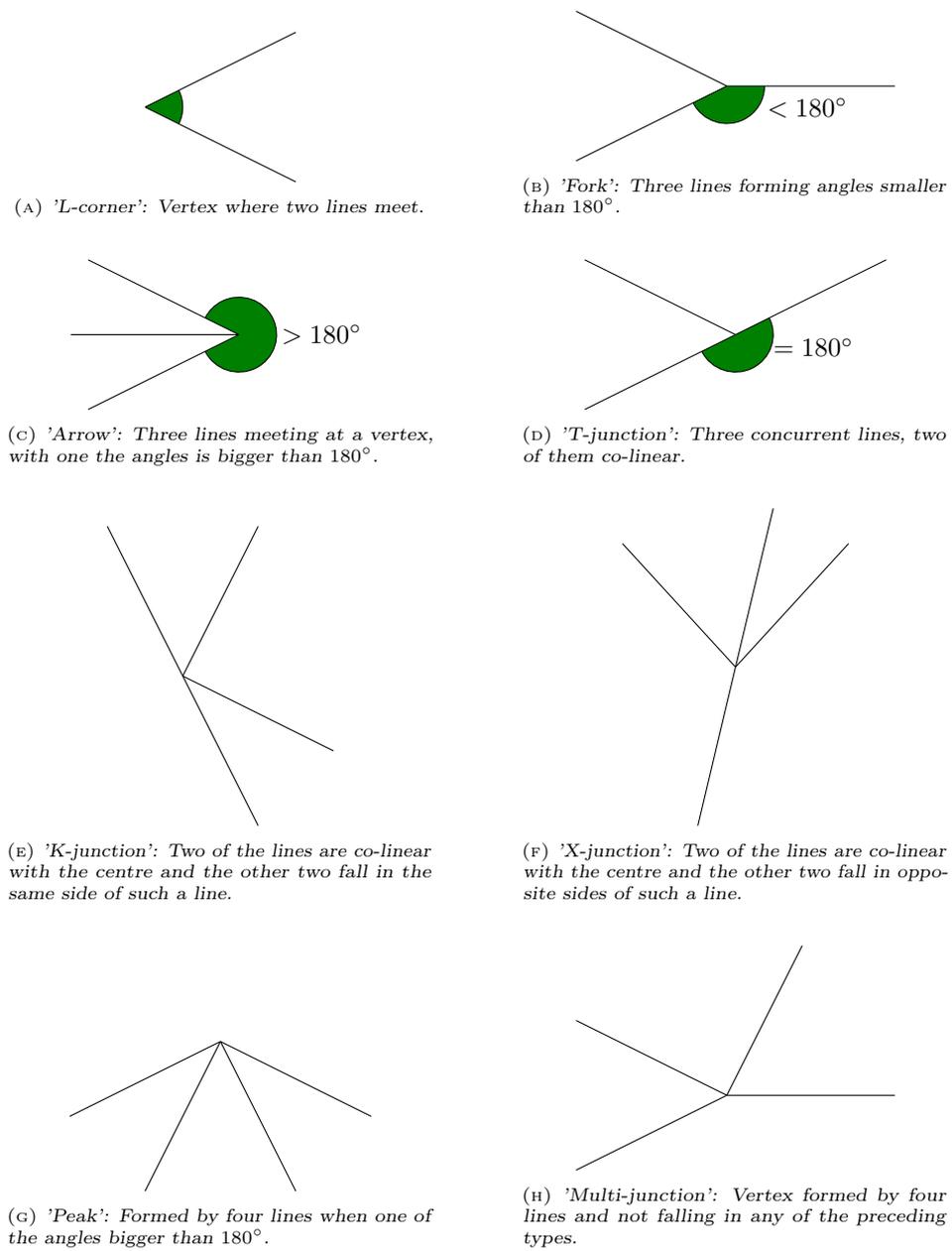
\begin{figure}%[h]
\centering
	\begin{subfigure}{0.35\textwidth}
	\centering
		\begin{tikzpicture}%[scale=.8]
			\coordinate (a) at (0,0);
			\coordinate (b) at (2,1);
			\coordinate (c) at (2,-1);
			\draw pic[draw,fill=ao(english),angle radius=.5cm] {angle=c--a--b};
			\draw (0,0) -- (2,1);
			\draw (0,0) -- (2,-1);
		\end{tikzpicture}
	\caption{'L-corner': Vertex where two lines meet.}
        \end{subfigure}
	\hspace{1cm}
	\begin{subfigure}{0.35\textwidth}
	\centering
		\begin{tikzpicture}%[scale=.8]
			\coordinate (a) at (0,0);
			\coordinate (b) at (2.23,0);
			\coordinate (c) at (-2,-1);
			\draw pic[draw,fill=ao(english),angle radius=.5cm,"$<180^{\circ}$" shift={(10mm,0mm)}] {angle=c--a--b};
			\draw (0,0) -- (-2,1);
			\draw (0,0) -- (2.23,0);
			\draw (0,0) -- (-2,-1);
		\end{tikzpicture}
	\caption{'Fork': Three lines forming angles smaller than $180^{\circ}$.}
        \end{subfigure}
       \\
       \vspace{.5cm}
	\begin{subfigure}{0.35\textwidth}
	\centering
		\begin{tikzpicture}%[scale=.8]
			\coordinate (a) at (0,0);
			\coordinate (b) at (-2,1);
			\coordinate (c) at (-2,-1);
			\draw pic[draw,fill=ao(english),angle radius=.5cm,"$>180^{\circ}$" shift={(8mm,0mm)}] {angle=c--a--b};
			\draw (0,0) -- (-2,1);
			\draw (0,0) -- (-2.23,0);
			\draw (0,0) -- (-2,-1);
		\end{tikzpicture}
	\caption{'Arrow': Three lines meeting at a vertex, with one the angles is bigger than $180^{\circ}$.}
        \end{subfigure}
	\hspace{1cm}
	\begin{subfigure}{0.35\textwidth}
	\centering
	\begin{tikzpicture}%[scale=.8]
		\coordinate (a) at (0,0);
		\coordinate (b) at (2,1);
		\coordinate (c) at (-2,-1);
		\draw pic[draw,fill=ao(english),angle radius=.5cm,"$=180^{\circ}$" shift={(9mm,1mm)}] {angle=c--a--b};
		\draw (0,0) -- (-2,1);
		\draw (0,0) -- (2,1);
		\draw (0,0) -- (-2,-1);
	\end{tikzpicture}
	\caption{'T-junction': Three concurrent lines, two of them co-linear.}
        \end{subfigure}
        \\
       \vspace{.5cm}
	\begin{subfigure}{0.35\textwidth}
	\centering
		\begin{tikzpicture}%[scale=.8]
			\draw (0,0) -- (-1,2);
			\draw (0,0) -- (1,2);
			\draw (0,0) -- (2,-1);
			\draw (0,0) -- (1,-2);
		\end{tikzpicture}
	\caption{'K-junction': Two of the lines are co-linear with the centre and the other two fall in the same side of such a line.}
        \end{subfigure}
	\hspace{1cm}
	\begin{subfigure}{0.35\textwidth}
	\centering
	\begin{tikzpicture}%[scale=.8]
		\draw (0,0) -- (.5,2.12);
		\draw (0,0) -- (-.5,-2.12);
		\draw (0,0) -- (-1.5,1.65);
		\draw (0,0) -- (1.5,1.65);
	\end{tikzpicture}
	\caption{'X-junction': Two of the lines are co-linear with the centre and the other two fall in opposite sides of such a line.}
        \end{subfigure}
        \\
       \vspace{.5cm}
	\begin{subfigure}{0.35\textwidth}
	\centering
		\begin{tikzpicture}%[scale=.8]
			\draw (0,0) -- (-1,-2);
			\draw (0,0) -- (1,-2);
			\draw (0,0) -- (-2,-1);
			\draw (0,0) -- (2,-1);
		\end{tikzpicture}
	\caption{'Peak': Formed by four lines when one of the angles bigger than $180^{\circ}$.}
        \end{subfigure}
	\hspace{1cm}
	\begin{subfigure}{0.35\textwidth}
	\centering
	\begin{tikzpicture}%[scale=.8]
		\draw (0,0) -- (1,2);
		\draw (0,0) -- (-2,1);
		\draw (0,0) -- (-2,-1);
		\draw (0,0) -- (2.23,0);
	\end{tikzpicture}
	\caption{'Multi-junction': Vertex formed by four lines and not falling in any of the preceding types.}
        \end{subfigure}
\caption{Classification of vertices~\cite{Guz68}.}
\label{fig_classification_vertices}
\end{figure}
\fi
To highlight the appearance of the particular vertices represented in Fig.~\ref{fig_classification_vertices}, we consider two 3D scenes: Fig.~\ref{fig_cube_example_vertices_classification}(a) depicts a cube, and Fig.~\ref{fig_cube_example_vertices_classification} (b) illustrates a cube partially occluding a bar. In the first case (a), the cube features vertices such as 'L-corner', 'Fork', and 'Arrow' (see Fig.~\ref{fig_classification_vertices}(a)-(c)). In contrast, in the occlusion scenario (b), we also observe 'T-junction' vertices (see Fig.~\ref{fig_classification_vertices}(d)).
\renewcommand\sqw{.325}
\ifpics
\begin{figure}%[h]
\centering
\begin{tikzpicture}
\node at (3.5,-.25) {(a)};
\draw[fill=gray(x11gray)] (4,7,4) -- (7,7,4) -- (7,4,4) -- (7,4,7) -- (4,4,7) -- (4,7,7) -- cycle;
\draw[fill=gray(x11gray)] (7,7,7) -- (4,7,7) (7,7,7) -- (7,4,7) (7,7,7) -- (7,7,4);
\draw[step=\sqw, style={thin,gray,opacity=.5}] (\sqw,\sqw) grid +(20*\sqw,20*\sqw);
\node at (5.5,5.5,7) {$R_1$};
\node at (5.5,7,5.5) {$R_2$};
\node at (7,5.5,5.5) {$R_3$};
\node at (7.5,3.25,5.5) {$R_4$};
%%% A
\node at (3.9,4.1,7.25) {A};
\draw[red,line cap=round,line width=1pt] (4,4,7) -- (4,4.5,7);
\draw[red,line cap=round,line width=1pt] (4,4,7) -- (4.5,4,7);
%%% B
\node at (3.9,7.1,7.25) {B};
\draw[red,line cap=round,line width=1pt] (4,7,7) -- (4,7,6.25);
\draw[red,line cap=round,line width=1pt] (4,7,7) -- (4.5,7,7);
\draw[red,line cap=round,line width=1pt] (4,6.5,7) -- (4,7,7);
%%% C
\node at (4,7.25,4.1) {C};
\draw[red,line cap=round,line width=1pt] (4,7,4) -- (4,7,4.75);
\draw[red,line cap=round,line width=1pt] (4,7,4) -- (4.5,7,4);
%%% D
\node at (7.1,7.25,4.1) {D};
\draw[red,line cap=round,line width=1pt] (7,7,4) -- (6.5,7,4);
\draw[red,line cap=round,line width=1pt] (7,7,4) -- (7,6.5,4);
\draw[red,line cap=round,line width=1pt] (7,7,4) -- (7,7,4.75);
%%% E
\node at (6.9,7.25,7) {E};
\draw[red,line cap=round,line width=1pt] (7,7,7) -- (6.5,7,7);
\draw[red,line cap=round,line width=1pt] (7,7,7) -- (7,6.5,7);
\draw[red,line cap=round,line width=1pt] (7,7,7) -- (7,7,6.25);
%%% F
\node at (7.2,4,4) {F};
\draw[red,line cap=round,line width=1pt] (7,4,4) -- (7,4.5,4);
\draw[red,line cap=round,line width=1pt] (7,4,4) -- (7,4,4.75);
%%% G
\node at (7.2,3.9,7) {G};
\draw[red,line cap=round,line width=1pt] (7,4,7) -- (7,4.5,7);
\draw[red,line cap=round,line width=1pt] (6.5,4,7) -- (7,4,7);
\draw[red,line cap=round,line width=1pt] (7,4,7) -- (7,4,6.25);
\end{tikzpicture}
\hspace{1cm}
\begin{tikzpicture}
\node at (3.5,-.25) {(b)};
\draw[fill=ao(english)] (6,8,6) -- (7,8,6) -- (7,3.75,6) -- (7,3.75,7) -- (6,3.75,7) -- (6,8,7) -- cycle;
\draw[fill=ao(english)] (7,8,7) -- (6,8,7) (7,8,7) -- (7,3.75,7) (7,8,7) -- (7,8,6);
\draw[fill=gray(x11gray)] (4,6,4) -- (6,6,4) -- (6,4,4) -- (6,4,6) -- (4,4,6) -- (4,6,6) -- cycle;
\draw[fill=gray(x11gray)] (6,6,6) -- (4,6,6) (6,6,6) -- (6,4,6) (6,6,6) -- (6,6,4);
\draw[step=\sqw, style={thin,gray,opacity=.5}] (\sqw,\sqw) grid +(20*\sqw,20*\sqw);
\node at (5.5,5.5,7) {$R_1$};
\node at (5,6,5) {$R_2$};
\node at (6,5,5) {$R_3$};
\node at (6.5,4.125,7) {$R_4$};
\node at (6.5,8,6.5) {$R_5$};
\node at (7,4.125,6.5) {\small{$R_6$}};
\node at (8,5,6.5) {$R_7$};
\node at (3.125,4.75) {$T_1$};
\draw[red,line cap=round,line width=1pt] (4.725,6,4) -- (4.975,6,4);
\draw[red,line cap=round,line width=1pt] (6,7.17,7) -- (6,7.42,7);
\node at (4.125,4.75) {$T_2$};
\draw[red,line cap=round,line width=1pt] (5.72,6,4) -- (5.97,6,4);
\draw[red,line cap=round,line width=1pt] (7,7.17,7) -- (7,7.42,7);
\node at (3.125,1.5) {$T_3$};
\draw[red,line cap=round,line width=1pt] (5.5,4,6) -- (5.75,4,6);
\draw[red,line cap=round,line width=1pt] (6,4.125,7) -- (6,4.375,7);
\node at (4.125,2) {$T_4$};
\draw[red,line cap=round,line width=1pt] (6,4,4.275) -- (6,4,4.525);
\draw[red,line cap=round,line width=1pt] (7,4.75,7) -- (7,5,7);
\end{tikzpicture}
\caption{Examples of classification of vertices. (a) - Cube: 'L-corner', 'Fork' and 'Arrow' vertices. (b) - Overlapping cubes: 'T-junction' vertex.}
\label{fig_cube_example_vertices_classification}
\end{figure}
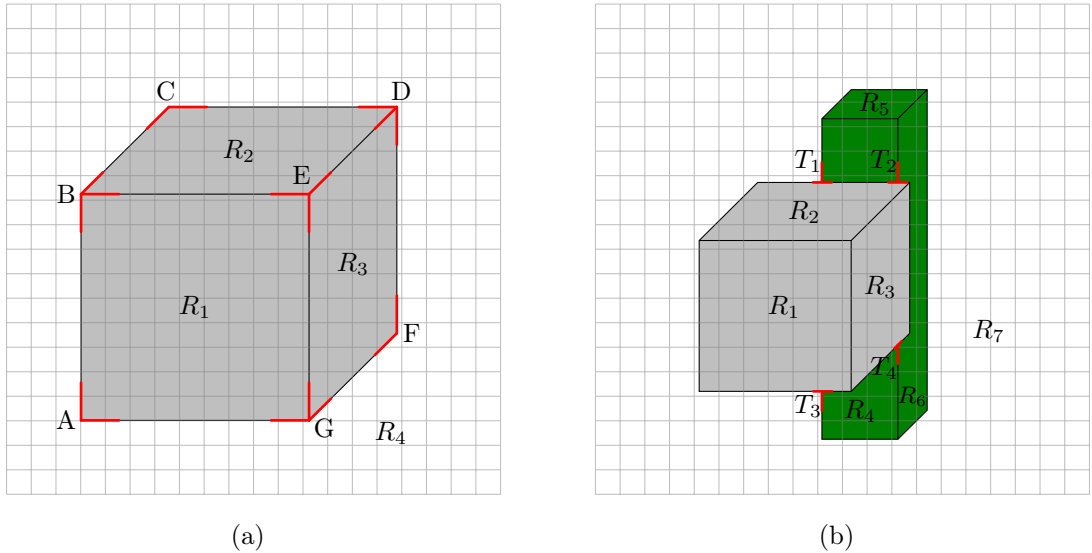
\fi

Topological asymptotic analysis has been incorporated in a broad spectrum of image processing tasks~\cite{BelJaoMasSia06,LarFeiNovTar08,Mus09,AurCohMas11,AmsFeh15}. This approach involves evaluating the sensitivity of a misfit function with respect to an infinitesimal variation of a trial defect (e.g., hole, crack, inclusion, etc.)~\cite{GarGuiMas01b,AmsHorMas05,Bon09}, among others. The most negative value of the topological derivative map is likely to be located at the zone where the defect exists. Belaid {\it{et al.}}~\cite{BelJaoMasSia06} adapted this concept for edge detection, originally used for {\it{crack detection}}~\cite{AmsHorMas05}, where edges can be viewed as a {\it{set of singularities}} in the image intensity function. Subsequently, a first-order topological derivative of a Mumford-Shah-type functional~\cite{MumSha89c} with respect to {\it{circular}} perturbations~\cite{Mus09,GraMusSch13} was developed. Additionally, the detection of edges at various scales (i.e. different magnitude of discontinuities) using topological derivatives was treated in the paper~\cite{DonGraKanSch13}. In Beretta {\it{et al.}}~\cite{BerGraMusSch14}, the authors extended this approach to the case where the image's edges are covered by a set of {\it{line segments}} rather than a set of {\it{points}}. However, the first-order topological derivative might not provide sufficient information or accuracy for detecting complex geometrical defects. To overcome these limitations, the second-order topological derivative was introduced for detecting finite-size holes~\cite{RocNovFeiTarPad07,RocNov09}, inclusions~\cite{Bon09,BonCor17}; see also~\cite{NovSokZoc18} for an overview of the application of second-order topological derivatives in inverse problems. Second-order topological derivatives have not been applied to computer vision for vertex characterization.

In this paper, we consider vertex detection using topological asymptotic analysis with respect to inclusion shapes, which represent the vertices illustrated in Fig.~\ref{fig_classification_vertices}. The novelty of this approach lies in the calculation of the second-order topological derivative associated with a Mumford-Shah-like functional with respect to an arbitrary inclusion shape, which we apply to the various classes of vertices from Fig.~\ref{fig_classification_vertices}. This second-order topological derivative is used as an indicator function, estimating the likelihood that a certain pixel in the 2D image belongs to a specific vertex type. Mathematically, these indicators provide not only the vertex's location but also its type, i.e., the number of lines forming it and the corresponding angles. 

This paper is structured as follows. Section~\ref{model_problem} reviews the concept of topological asymptotic analysis based on a Mumford-Shah type functional. The topological sensitivity method is defined and presented for the modified Mumford-Shah type functional in Section~\ref{preliminaries}. Then, we derive the asymptotic expansion of the perturbed state solution in Section~\ref{analysis_perturbed_state_equation}. The first and second-order topological derivatives are computed explicitly with respect to arbitrary inclusions shape, in particular, vertex types from Fig.~\ref{fig_classification_vertices}, in Section~\ref{topological_derivatives_cost_function}. Section~\ref{numerical_algorithm} presents the developed algorithm for detecting inclusions, covering known and unknown inclusion shapes. The numerical results are discussed, in Section~\ref{numerical_experiments}, for the two 3D scenes given in Fig.~\ref{fig_cube_example_vertices_classification}.     

%%%%%%%%%%%%%%%%%%%%%%%%%%%%%%%%
\section{Model problem}
\label{model_problem}
%%%%%%%%%%%%%%%%%%%%%%%%%%%%%%%%
Consider $D\subset\R^2$ as an open bounded Lipschitz domain and $f:D\rightarrow\R$ a given function in $L^\infty(D)$ representing an image. The Mumford-Shah functional
\[
F(u,K):=\frac12\int_{D}(u-f)^2\dx+\frac{\alpha}{2}\int_{D\setminus K}|\nabla u|^2\dx+\beta\mathcal{H}^1(K),
\] 
aims at segmenting an image $f$ by minimizing $F(u, K)$ over all smooth functions $u$ defined in $D$ and all curves $K\subset D$. Here, $\mathcal{H}^1(K)$ denotes the one-dimensional Hausdorff measure of the set of edges $K$ (i.e. if $K$ is a regular one-dimensional set, it will be equal to its length), and $\alpha$, $\beta$ are positive weights. This functional measures the matching between a given image data $f$ and an approximation $u$, where the first term imposes $u$ to approximate $f$ and the second one controls the variation of $u$ away from the edge set $K$, for more details see~\cite{MumSha89c}.

In this work, we follow the approach of Grasmair {\it{et al.}}~\cite{GraMusSch13} and Beretta {\it{et al.}}~\cite{BerGraMusSch14} where the cost function
\begin{align} 
\label{eq_Jeps}
J_\varepsilon(u,v_K):=\frac12\int_{D}(u-f)^2\dx+\frac{\alpha}{2}\int_{D}v_K|\nabla u|^2\dx+2\beta\varepsilon m_\varepsilon(v_K),
\end{align}
was considered as an approximation of the Mumford-Shah functional $F(u,K)$ in the sense of $\Gamma-$convergence. Here, the considered functional is minimized over all the functions $u\in H^1(D)$ and $v_K\in L^\infty(D)$, where $v_K:D\rightarrow\R$ is a piecewise constant edge indicator defined by
\[  
v_K(x):=
    \begin{cases}
      \kappa,&x\in K, \\
    1,&x\in D\setminus\overline{K},
    \end{cases}
\]
where $K\subset D$ and $0<\kappa<1$. In~\eqref{eq_Jeps}, $m_\varepsilon(v_K)$ denotes the number of inclusions of size $\varepsilon>0$ that are required to cover the edge represented by $v_K$, for $\varepsilon$ small enough and $K$ sufficiently regular. In both these works, the first-order topological derivative of the cost function~\eqref{eq_Jeps} was computed and used in an iterative algorithm in order to detect the edge sets of an image. While in Grasmair {\it{et al.}}~\cite{GraMusSch13} the edge set is covered by ball-shaped inclusions, in Beretta {\it{et al.}}~\cite{BerGraMusSch14} the edges are covered with a finite number of thin strips rather than accumulations of points, resulting in a faster algorithm.

In this work, in contrast, we are not interested in an iterative algorithm but rather employ a one-shot method based on second-order topological derivatives in order to locate vertices and extract their characteristics.
Therefore, we define the following cost functional inspired by the Mumford-Shah model,
\begin{equation}
\label{unper_cost_func}
J(u,\Omega):=\frac12\int_D(u-f)^2+\alpha\lambda_\Omega|\nabla u|^2\dx,
\end{equation}
where $u\in H^1(D)$ and $\lambda_\Omega\in L^\infty(D)$. For $\Omega$ a fixed open subset of $D$, we define the function $\lambda_\Omega:D\rightarrow\R$ by
\[  
\lambda_{\Omega}(x)=
    \begin{cases}
      \lambda^{\text{in}},&x\in\Omega,\\
      \lambda^{\text{out}},&x\in D\setminus\overline{\Omega},
    \end{cases}
\]
where $0<\lambda^{\text{in}} < \lambda^{\text{out}}$ are positive constants representing the material distributions.

For a fixed admissible shape $\Omega\subset D$, the unique minimizer of the convex functional
\begin{align*}
    J(\cdot, \Omega): H^1(D) \rightarrow &\mathbb R \\
    u \mapsto& J(u, \Omega)
\end{align*}
is the unique solution $u \in H^1(D)$ to the boundary value problem
\begin{equation}
\int_D \alpha\lambda_\Omega\nabla u\cdot\nabla v + uv\dx= \int_D fv \dx,\quad\quad\forall v \in H^1(D),
\label{unper_state_vari_form}
\end{equation}
which, in its strong form, reads
\begin{align} \label{eq_state_strongform}
    \left\{\begin{aligned}
     -\alpha\mbox{ div}(\lambda_\Omega\nabla u) + u=&f&&\mbox{in }D,                  \\
     \frac{\partial u}{\partial n}=&0&&\mbox{on }\partial D.
    \end{aligned} \right.
\end{align}
For a given $\Omega \subset D$, defining $u(\Omega)$ as the unique solution to~\eqref{unper_state_vari_form}, we can introduce the reduced cost functional as
\begin{align} 
\label{eq_reducedCost}
    \mathcal{J}(\Omega):=J(u(\Omega),\Omega).
\end{align}
\begin{remark}
Note that for a given $\Omega \subset D$ and the solution $u(\Omega)$ of~\eqref{unper_state_vari_form}, it holds for all $v \in H^1(D)$ that
\begin{align}
\label{eq_Jprimezero}
    J'(u(\Omega), \Omega)(v) = 0.
\end{align}
Also note that when considering~\eqref{unper_cost_func} together with~\eqref{unper_state_vari_form} as a PDE-constrained optimization problem, the right-hand side of the corresponding adjoint equation would be given by~\eqref{eq_Jprimezero} and, thus, the adjoint state would vanish.
\end{remark}

%%%%%%%%%%%%%%%%%%%%%%%%%%%%%%%%
\section{Asymptotic expansion of $\mathcal{J}$}
\label{asymptotic_expansion_J}
%%%%%%%%%%%%%%%%%%%%%%%%%%%%%%%%
In this section, we consider perturbed domains $\Omega_\varepsilon$ obtained from $\Omega$ by adding or removing small inclusions $\omega_\varepsilon(z)$ of size $\varepsilon$ centered around a spatial point $z\in \Omega \cup (D\setminus \overline \Omega)$.
Our goal is to establish an asymptotic expansion for $\mathcal{J}(\Omega_\varepsilon)-\mathcal{J}(\Omega)$ in powers of $\varepsilon$ as $\varepsilon\rightarrow 0$ using first and second-order topological derivatives.

%%%%%%%%%%%%%%%%%%%%%%%%%%%%%%%%
\subsection{Preliminaries}
\label{preliminaries}
%%%%%%%%%%%%%%%%%%%%%%%%%%%%%%%%
Let $\omega\subset\R^2$ with $\bm{0}\in\omega$ represent the shape of the considered perturbation and let $z\in\Omega\cup(D\setminus\overline{\Omega})$. For $\varepsilon>0$, we define the perturbation of shape $\omega$ and size $\varepsilon$ as $\omega_\varepsilon(z):=z+\varepsilon\omega$.
\begin{definition}\label{definition_topological_derivative} (Topological derivative). The first-order topological derivative of a shape function $\mathcal{J}$ at the point $z\in\Omega\cup(D\setminus\overline{\Omega})$ with respect to the inclusion shape $\omega$ is defined by
\begin{equation}
\label{definition_TD1}
d\mathcal{J}(\Omega)(z,\omega):=
    \left\{\begin{aligned}
    &\lim_{\varepsilon\rightarrow 0}\frac{1}{|\omega_\varepsilon(z)|}(\mathcal{J}(\Omega\setminus\overline{\omega}_\varepsilon(z))-\mathcal{J}(\Omega)),&z\in\Omega,\\
    &\lim_{\varepsilon\rightarrow 0}\frac{1}{|\omega_\varepsilon(z)|}(\mathcal{J}(\Omega\cup\omega_\varepsilon(z))-\mathcal{J}(\Omega)),&z\in D\setminus\overline{\Omega}.
    \end{aligned} \right.
\end{equation}
The second-order topological derivative of a shape function $\mathcal{J}$ at the point $z\in\Omega\cup(D\setminus\overline{\Omega})$ with respect to the inclusion shape $\omega$ is defined by
\begin{equation}
\label{definition_TD2}
d^2\mathcal{J}(\Omega)(z,\omega):=
    \left\{\begin{aligned}
    &\lim_{\varepsilon\rightarrow 0}\frac{1}{\varepsilon|\omega_\varepsilon(z)|}(\mathcal{J}(\Omega\setminus\overline{\omega}_\varepsilon(z))-\mathcal{J}(\Omega)-|\omega_\varepsilon(z)|d\mathcal{J}(\Omega)(z,\omega)),&z\in\Omega,\\
    &\lim_{\varepsilon\rightarrow 0}\frac{1}{\varepsilon|\omega_\varepsilon(z)|}(\mathcal{J}(\Omega\cup\omega_\varepsilon(z))-\mathcal{J}(\Omega)-|\omega_\varepsilon(z)|d\mathcal{J}(\Omega)(z,\omega)),&z\in D\setminus\overline{\Omega}.
    \end{aligned} \right.
\end{equation}
\end{definition}
For the rest of this section, we will fix a spatial point $z \in D \setminus \overline \Omega$ and use the shorthand notation $\omega_\varepsilon$ instead of $\omega_\varepsilon(z)$. We also define the perturbed domain $\Omega_\varepsilon:= \Omega \cup \omega_\varepsilon$. Moreover, we assume that the solution to the unperturbed problem~\eqref{unper_state_vari_form} is smooth enough near the point of perturbation $z$, i.e., $u \in C^2(B_\delta(z))$ for some $\delta >0$ where $B_\delta(z)$ denotes the ball of radius $\delta$ centered at the point $z$.
\begin{remark}
By Definition~\ref{definition_topological_derivative}, the first and second-order topological derivatives satisfy the topological asymptotic expansion
 \[
\mathcal{J}(\Omega_\varepsilon)=\mathcal{J}(\Omega)+\ell_1(\varepsilon) d\mathcal{J}(\Omega)(z,\omega)+ \ell_2(\varepsilon) d^2\mathcal{J}(\Omega)(z,\omega)+o(\ell_2(\varepsilon)). 
 \]
 with $\ell_1(\varepsilon) = |\omega_\varepsilon|$ and $\ell_2(\varepsilon) = \varepsilon|\omega_\varepsilon|$. We remark that, depending on the problem setting at hand (e.g. on the choice of boundary conditions on the boundary of the perturbation $\partial \omega_\varepsilon$), first and second-order topological derivatives are defined with different choices of $\ell_1(\varepsilon)$, $\ell_2(\varepsilon)$ satisfying $\ell_2(\varepsilon)/\ell_1(\varepsilon) = o(1)$ as $\varepsilon \rightarrow 0$.
\end{remark}

First-order topological derivatives of shape optimization problems with linear PDE constraints have been studied for a long time, e.g. the early works~\cite{SokZoc99,GarGuiMas01b} or related work on asymptotic analysis of PDEs~\cite{KozMazMov99}. When the principal part of a PDE is perturbed, as is the case in~\eqref{unper_state_vari_form}, the corresponding topological derivative typically involves the solution to an exterior problem. If the inclusion shape $\omega$ is a disk or ellipse in 2D or a ball or ellipsoid in 3D, this exterior problem can be solved explicitly and closed-form formulas for the corresponding topological derivative can be obtained~\cite{Ams06}. While for other inclusion shapes or quasilinear PDE constraints, the exterior problem can no longer be solved explicitly, the topological derivative can still be obtained in terms of the solution to the exterior problem and numerical approximations can be computed~\cite{AmsGan19,GanStu21,GanNeeSti24}. Second and higher-order topological derivatives have been studied by several authors~\cite{CanNovRoc14,BonCor17,NovSokZoc18,BauStu22}. 

In this work, we consider first and second-order topological derivatives with respect to certain polygonal inclusion shapes in order to detect vertices in images, see Fig.~\ref{fig_conf_pixels}(a) and (b) for an exemplary situation of an unperturbed and perturbed setting, respectively.
\ifpics
\begin{figure}%[h]
\centering
	\begin{tikzpicture}
		\draw[fill=gray(x11gray)] (0.2,0.2) -- (0.2,4.2) -- (4.2,4.2) -- (4.2,0.2) -- cycle;
		\draw[fill=red] (2.2,2.2) -- (3.2,2.2) -- (3.2,3.2) -- (2.7,3.2) -- (2.7,2.7) -- (2.2,2.7) -- cycle;
		\draw (1.5,4) node[anchor=north east] {$D\setminus\Omega$};
		\draw (3,2.7) node[anchor=north east] {$\Omega$};
		\node at (1.75,.75) {.};
		\node at (1.9,.75) {$z$};
		\draw[->,dashed,thick] (.2,0) -- (.2,4.4);
		\draw[->,dashed,thick] (0,.2) -- (4.4,.2);
		\node at (2,-.25) {(a)};
	\end{tikzpicture}
	\hspace{2cm}
	\begin{tikzpicture}
		\draw[fill=gray(x11gray)] (0.2,0.2) -- (0.2,4.2) -- (4.2,4.2) -- (4.2,0.2) -- cycle;
		\draw[fill=red] (2.2,2.2) -- (3.2,2.2) -- (3.2,3.2) -- (2.7,3.2) -- (2.7,2.7) -- (2.2,2.7) -- cycle;
		\draw[fill=red] (1.4375,.6875) -- (1.8125,.6875) -- (1.8125,1.0625) -- (1.6875,1.0625) -- (1.6875,.8125) -- (1.4375,.8125) -- cycle;
		\draw (1.5,4) node[anchor=north east] {$D\setminus\Omega$};
		\draw (3,2.7) node[anchor=north east] {$\Omega$};
		\node at (1.75,.75) {.};
		\node at (1.95,.75) {$z$};
		\node at (1.85,1.2) {$\omega_\varepsilon$};
		\draw[->,dashed,thick] (.2,0) -- (.2,4.4);
		\draw[->,dashed,thick] (0,.2) -- (4.4,.2);
		\node at (2,-.25) {(b)};
	\end{tikzpicture}
	\caption{(a) Unperturbed configuration. (b) Perturbed configuration where the domain is perturbed in an inclusion shape $\omega_\varepsilon$ whose center is the point $z$.}
\label{fig_conf_pixels}
\end{figure}
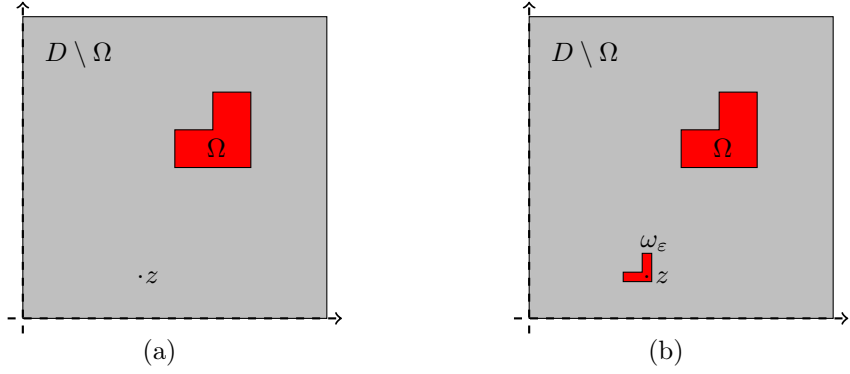
\fi

With the notation introduced above, the perturbed cost functional is given by
\begin{equation}
\label{per_cost_func}
\mathcal{J}(\Omega_\varepsilon)=J(u(\Omega_\varepsilon),\Omega_\varepsilon)=\frac12\int_D(\ue-f)^2+\alpha\lambda_{\Omega_\varepsilon}|\nabla \ue|^2\dx,
\end{equation}
where $\lambda_{\Omega_\varepsilon}:D\rightarrow\R$ is given by
\[
\lambda_{\Omega_\varepsilon}(x)=
    \left\{\begin{aligned}
      &\lambda^{\text{in}},&x\in\Omega_\varepsilon,
    \\&\lambda^{\text{out}},&x\in D\setminus\overline{\Omega}_\varepsilon,
    \end{aligned} \right.
\]
and $\ue:= u(\Omega_\varepsilon) \in H^1(D)$ satisfies the variational formulation
\begin{equation}
\label{per_state_vari_form}
\int_D\alpha\lambda_{\Omega_\varepsilon}\nabla\ue\cdot\nabla v + \ue v\dx= \int_D f v \dx,\quad\quad\forall v \in H^1(D),
\end{equation}
which can be written in its strong form as
\[
    \left\{\begin{aligned}
    -\alpha\mbox{ div}(\lambda_{\Omega_\varepsilon}\nabla \ue)+ \ue=&f&&\mbox{in }D,                   \\
    \frac{\partial \ue}{\partial n}=&0&&\mbox{on }\partial D.
    \end{aligned} \right.
\]

%%%%%%%%%%%%%%%%%%%%%%%%%%%%%%%%
\subsection{Analysis of the perturbed state equation}
\label{analysis_perturbed_state_equation}
%%%%%%%%%%%%%%%%%%%%%%%%%%%%%%%%
The key ingredient to computing topological derivatives in the presence of PDEs is the asymptotic analysis of the solution to the perturbed problem. Here, we will follow the publication~\cite[Sec. 3]{BauStu22} where the authors perform the rigorous analysis up to second-order for a linear elasticity problem. The results on the asymptotic can be transferred from the vector-valued setting considered there to the case of the scalar elliptic PDE~\eqref{per_state_vari_form} in our setting in a straightforward way. Thus, for the sake of a compact presentation, we will not repeat the proofs of Lemma~\ref{lem_Keps_to_K} and Lemma~\ref{lem_K2eps_to_K2} and just state the results here.

We define the affine transformation $\phi_{\varepsilon}:\omega\rightarrow\omega_\varepsilon,x\mapsto z+\varepsilon x$ satisfying $\phi_{\varepsilon}(\omega)=\omega_\varepsilon$ and introduce the notation $D_\varepsilon := \phi_\varepsilon^{-1}(D)$ for the rescaled domain, see also Fig.~\ref{fig_rescaled_perturbed_domain}.
\begin{definition} 
\label{def_Keps}
The first variation of the state $\ue$ is defined by
\begin{equation}
\label{first_vari_state}
K^{(1)}_{\varepsilon,\omega} :=\left(\frac{\ue-u}{\varepsilon}\right)\circ\phi_\varepsilon \in H^1(D_\varepsilon),\quad\quad\text{for }\varepsilon>0.
\end{equation}
\end{definition}
\begin{lemma}
    The function $K_{\varepsilon, \omega}^{(1)} \in H^1(D_\varepsilon)$ defined in Definition~\ref{def_Keps} is the unique solution to
    \begin{equation}
        \label{firstvari__vari_form}
        \int_{D_\varepsilon}\alpha\lambda_{\omega \cup \phi_\varepsilon^{-1}(\Omega)}\nabla K^{(1)}_{\varepsilon,\omega}\cdot\nabla v + \varepsilon K^{(1)}_{\varepsilon,\omega}v\dx=-\alpha(\lambda^{\text{in}}-\lambda^{\text{out}})\int_{\omega}\nabla u\circ\phi_{\varepsilon}\cdot\nabla v\dx,
    \end{equation}
    for all $v \in H^1(D_\varepsilon)$. Moreover, $\nabla K_{\varepsilon, \omega}^{(1)}$ is bounded, i.e., there exists $C>0$ such that
    \begin{align} \label{eq_Keps_bounded}
        \|\nabla K_{\varepsilon, \omega}^{(1)} \|_{L^2(D_\varepsilon)} \leq C.
    \end{align}
\end{lemma}
\begin{proof}
By subtracting the perturbed equation~\eqref{per_state_vari_form} from the unperturbed state equation~\eqref{unper_state_vari_form}, we have
\begin{equation}
\label{state_difference}
\int_D\alpha\lambda_{\Omega_\varepsilon}\nabla(\ue-u)\cdot\nabla v+(\ue-u)v\dx=-\alpha(\lambda^{\text{in}}-\lambda^{\text{out}})\int_{\omega_\varepsilon}\nabla u\cdot\nabla v\dx,
\end{equation}
for $v\in H^1(D)$. Using the definition~\eqref{first_vari_state} and the fact that $\nabla(w\circ\phi_{\varepsilon}^{-1})=\varepsilon^{-1}\nabla w\circ\phi_{\varepsilon}^{-1}$, one gets
\[
\int_{D}\alpha\lambda_{\Omega_\varepsilon}(\nabla K^{(1)}_{\varepsilon,\omega}\circ\phi^{-1}_{\varepsilon})\cdot\nabla v+\varepsilon(K^{(1)}_{\varepsilon,\omega}\circ\phi^{-1}_{\varepsilon})v\dx=-\alpha(\lambda^{\text{in}}-\lambda^{\text{out}})\int_{\omega_\varepsilon}\nabla u\cdot\nabla v\dx,
\]
for $v\in H^1(D)$. Transforming the last equation to the perturbed domain $D_\varepsilon=\phi_{\varepsilon}^{-1}(D)$, as depicted in Fig.~\ref{fig_rescaled_perturbed_domain}, and using the fact that $\text{det}(\partial\phi_{\varepsilon})=\varepsilon^2$ is independent of $x$, it follows that $K_{\varepsilon, \omega}^{(1)}$ satisfies~\eqref{firstvari__vari_form}. The uniqueness follows from the Lax-Milgram lemma by standard arguments.

The proof of the boundedness of $\nabla K_{\varepsilon,\omega}^{(1)}$ can be found in~\cite[Lem. 4.1]{GanStu20} or~\cite[Lem. 3.6]{BauStu22}.
\end{proof}
\ifpics
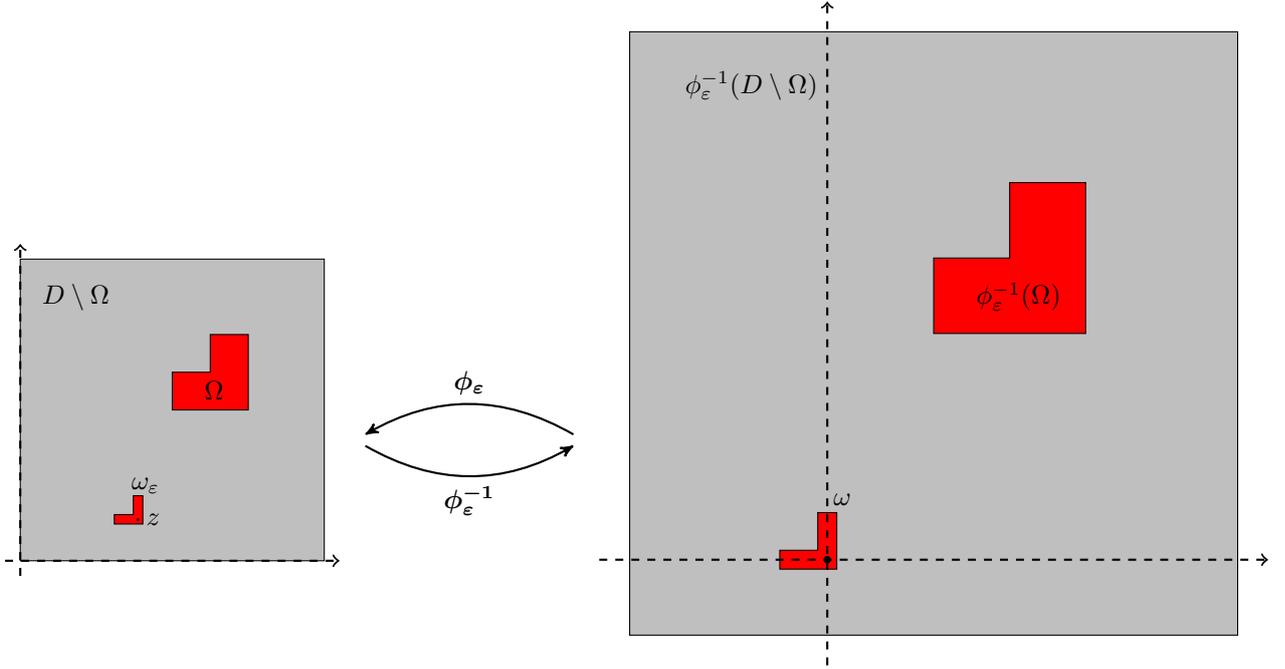
\begin{figure}%[h]
	\begin{tikzpicture}[baseline={(0,-1.2)}]
		\draw[fill=gray(x11gray)] (0.2,0.2) -- (0.2,4.2) -- (4.2,4.2) -- (4.2,0.2) -- cycle;
		\draw[fill=red] (2.2,2.2) -- (3.2,2.2) -- (3.2,3.2) -- (2.7,3.2) -- (2.7,2.7) -- (2.2,2.7) -- cycle;
		\draw[fill=red] (1.4375,.6875) -- (1.8125,.6875) -- (1.8125,1.0625) -- (1.6875,1.0625) -- (1.6875,.8125) -- (1.4375,.8125) -- cycle;
		\draw (1.5,4) node[anchor=north east] {$D\setminus\Omega$};
		\draw (3,2.7) node[anchor=north east] {$\Omega$};
		\node at (1.75,.75) {.};
		\node at (1.95,.75) {$z$};
		\node at (1.85,1.2) {$\omega_\varepsilon$};
		\draw[->,dashed,thick] (.2,0) -- (.2,4.4);
		\draw[->,dashed,thick] (0,.2) -- (4.4,.2);
	\end{tikzpicture}
	\begin{tikzpicture}[->,>=stealth',auto,node distance=3cm,thick,baseline={(0,-3)}]
		\node (1) {};
		\node (2) [right of=1] {};
		\path (2) edge[bend right] node [above] {$\bm{\phi_{\varepsilon}}$} (1);
		\path (1) edge[bend right] node [below] {$\bm{\phi_{\varepsilon}^{-1}}$} (2);
	\end{tikzpicture}
	\begin{tikzpicture}
		\draw[fill=gray(x11gray)] (0.4,0.4) -- (0.4,8.4) -- (8.4,8.4) -- (8.4,0.4) -- cycle;
		\draw[fill=red] (4.4,4.4) -- (6.4,4.4) -- (6.4,6.4) -- (5.4,6.4) -- (5.4,5.4) -- (4.4,5.4) -- cycle;
		\draw[fill=red] (2.375,1.275) -- (3.125,1.275) -- (3.125,2.025) -- (2.875,2.025) -- (2.875,1.525) -- (2.375,1.525) -- cycle;
		\draw (3,8) node[anchor=north east] {$\phi_{\varepsilon}^{-1}(D\setminus\Omega)$};
		\draw (6.2,5.2) node[anchor=north east] {$\phi_{\varepsilon}^{-1}(\Omega)$};
		\draw[fill=black] (3,1.4) circle (1.2pt);
		\node at (3.2,2.2) {$\omega$};
		\draw[->,dashed,thick] (3,0) -- (3,8.8);
		\draw[->,dashed,thick] (0,1.4) -- (8.8,1.4);
	\end{tikzpicture}
\caption{Original perturbed domain $D$ with subdomains $\Omega$ and $\omega_\varepsilon$ and rescaled perturbed domain $D_\varepsilon = \phi_\varepsilon^{-1}(D)$ with subdomains $\phi_\varepsilon^{-1}(\Omega)$ and $\omega = \phi_\varepsilon^{-1}(\omega_\varepsilon)$.}
\label{fig_rescaled_perturbed_domain}
\end{figure}
\fi

It can be shown that the first variation of the state, i.e. $K_{\varepsilon,\omega}^{(1)}$, converges to the so-called first corrector function $K^{(1)}_{\omega}\in \dot{BL}(\R^2)$. Here, $BL(\R^2):=\{v\in H^1_{\text{loc}}(\R^2):\nabla v\in L^2(\R^2)^2\}$ is the Beppo-Levi space of locally integrable functions whose derivative is square integrable over the whole unbounded domain and $\dot{BL}(\R^2):=BL(\R^2)/\R$ is the space of equivalence classes. Next, we present the exterior problem defining this limit.
\begin{definition}
\label{definition_first_corrector_function}
Let $\omega\subset\R^2$ open with $\bm{0}\in\omega$.
For a vector $\eta \in \R^2$, we define $K^{(1)}_\omega[\eta]\in \dot{BL}(\R^2)$ as the unique solution to
\begin{equation}
\label{corrector_vari_form}
\int_{\R^2}\alpha\lambda_{\omega}\nabla K^{(1)}_{\omega}[\eta]\cdot\nabla v\dx=-\alpha(\lambda^{\text{in}}-\lambda^{\text{out}})\int_{\omega} \eta \cdot\nabla v\dx,\quad\forall v\in BL(\R^2),
\end{equation}
where
\[
\lambda_{\omega}(x)=
    \left\{\begin{aligned}
      &\lambda^{\text{in}},&x\in\omega,
    \\&\lambda^{\text{out}},&x\in\R^2\setminus\overline{\omega}.
    \end{aligned} \right.
\]

Moreover, for $z\in D\setminus\overline{\Omega}$ and $\nabla u(z)$ the point evaluation of the gradient of the solution $u$ to~\eqref{unper_state_vari_form} at $z$, we call $K^{(1)}_\omega[\nabla u(z)]\in \dot{BL}(\R^2)$ the first corrector function of the perturbed problem~\eqref{per_state_vari_form}.
\end{definition}
\begin{remark}
The fact that~\eqref{definition_first_corrector_function} has a unique solution in $\dot{BL}(\R^2)$ follows from an application of the Lax-Milgram lemma. Note that $\dot{BL}(\R^2)$ is a space of equivalence classes of functions and that elements of $\dot{BL}(\R^2)$ are only unique up to additive constants.
Moreover, it can be seen from Definition~\ref{definition_first_corrector_function} that the mapping $\zeta\mapsto K^{(1)}_\omega[\zeta]$ is linear and we have
\begin{equation}
\label{linear_corrector_function}
K^{(1)}_\omega[\nabla u(z)]=\frac{\partial u}{\partial x_1}(z)K^{(1)}_\omega[\bm{e}^{(1)}]+\frac{\partial u}{\partial x_2}(z)K^{(1)}_\omega[\bm{e}^{(2)}],
\end{equation}
where $(\bm{e}^{(1)},\,\bm{e}^{(2)})$ is the canonical orthonormal basis of $\R^2$.
\end{remark}
\begin{lemma}\label{lem_Keps_to_K}(\cite[Thm. 3.9]{BauStu22},~\cite[Thm. 4.3]{GanStu20})
The first variation of the state $K_{\varepsilon, \omega}^{(1)}$ converges to the first corrector $K_\omega^{(1)}$ in the sense
    \begin{align}
        \nabla K^{(1)}_{\varepsilon,\omega} \rightarrow& \nabla K_\omega^{(1)}[\nabla u(z)] \mbox{ in } L^2(\omega).\label{eq_Keps_to_K} 
    \end{align}
\end{lemma}
\begin{definition}
The second variation of the state $\ue$ is defined as follows 
\begin{equation}
\label{second_vari_state}
K^{(2)}_{\varepsilon,\omega}:=\frac{K^{(1)}_{\varepsilon,\omega}-K^{(1)}_{\omega}}{\varepsilon} \in H^1(D_\varepsilon),\quad\quad\text{for }\varepsilon>0
\end{equation}
\end{definition}
\begin{remark}
    By extending the perturbed and unperturbed states $u_\varepsilon$ and $u$ to $\R^d$ using a continuous extension operator, we can view $K_{\varepsilon,\omega}^{(1)}$ and $K_{\varepsilon,\omega}^{(2)}$ as elements of $BL(\R^2)$.
\end{remark}

The second variation of the state can be shown to converge to the so-called second corrector function $K^{(2)}_{\omega}[\nabla^2  u(z)]\in \dot{BL}(\R^2)$, which is again defined as the unique solution to an exterior problem.
\begin{definition}
\label{definition_second_corrector_function}
Let $\omega\subset\R^2$ open with $\bm{0}\in\omega$. 
For a spatial point $z\in D\setminus\overline{\Omega}$, let $\nabla^2 u(z)$ denote the point evaluation of the Hessian of the solution $u$ to~\eqref{unper_state_vari_form} at $z$. The second corrector function $K^{(2)}_\omega[\nabla^2 u(z)]\in \dot{BL}(\R^2)$ is the unique solution to
\begin{equation}
\label{second_corrector_vari_form}
\int_{\R^2}\alpha\lambda_{\omega}\nabla K^{(2)}_{\omega}[\nabla^2 u(z)]\cdot\nabla v\dx=-\alpha(\lambda^{\text{in}}-\lambda^{\text{out}})\int_{\omega}\nabla^2 u(z)x\cdot\nabla v\dx,\quad\forall v\in BL(\R^2),
\end{equation}
where
\[
\lambda_{\omega}(x)=
    \begin{cases}
      \lambda^{\text{in}},&x\in\omega, \\
    \lambda^{\text{out}},&x\in\R^2\setminus\overline{\omega}.
    \end{cases}
\]
\end{definition}
\begin{lemma} \label{lem_K2eps_to_K2}(\cite[Thm. 3.16(2)]{BauStu22})
The second variation of the state $K_{\varepsilon, \omega}^{(2)}$ converges to the second corrector $K_\omega^{(2)}$ in the sense
    \begin{align} \label{eq_K2eps_to_K2}
        \nabla K_{\varepsilon, \omega}^{(2)} \rightarrow \nabla K_\omega^{(2)}[\nabla^2 u(z)] \; \mbox{ in } L^2(\omega).
    \end{align}
\end{lemma}

%%%%%%%%%%%%%%%%%%%%%%%%%%%%%%%%
\subsection{Topological derivatives of the cost function}
\label{topological_derivatives_cost_function}
%%%%%%%%%%%%%%%%%%%%%%%%%%%%%%%%
Here, we state the first and second-order topological derivatives at a spatial point $z\in D$ with respect to an arbitrary inclusion shape $\omega\subset\R^2$ with $\bm{0}\in\omega$ for the cost function defined in~\eqref{eq_reducedCost}.
\begin{theorem}
\label{THTD1}
Let a spatial point $z\in D\setminus\overline\Omega$ and an inclusion shape $\omega\subset\R^2$ with $\bm{0} \in \omega$ be given.
The first-order topological derivative at $z$ with respect to $\omega$ is given by
 \begin{equation}
 \label{TD1}
d\mathcal{J}(\Omega)(z,\omega)=\frac{\alpha}{2}(\lambda^{\text{in}}-\lambda^{\text{out}})\nabla u(z)^T\left[\mathcal{I}_2+\mathcal{P}^{(1)}_{\omega}\right]\nabla u(z).
\end{equation}
Here, $v^T$ denotes the transpose of a vector $v$, $\mathcal{I}_2$ denotes the second-order identity tensor and $\mathcal{P}^{(1)}_\omega$ denotes the matrix defined as
\begin{equation}
\label{first_order_weak_polarisation_matrix}
\mathcal{P}^{(1)}_\omega=\left[\frac{1}{|\omega|}\int_\omega\nabla K^{(1)}_\omega[\bm{e}^{(1)}]\dx\quad\quad\frac{1}{|\omega|}\int_\omega\nabla K^{(1)}_\omega[\bm{e}^{(2)}]\dx\right]\in\R^{2\times 2},
\end{equation}
where $K^{(1)}_\omega[\bm{e}^{(k)}]$ is defined in Definition~\ref{definition_first_corrector_function}, for $k=1,2$.
\end{theorem}
\begin{remark}
The matrix $\mathcal{P}^{(1)}_\omega$ is also refered to as weak polarization matrix~\cite{Stu20}. The more commonly known concept, refered to as polarization matrix $P_\omega$~\cite{AmmKan07}, is related to $\mathcal P_\omega^{(1)}$ by $P_\omega = (\lambda^{\text{in}}-\lambda^{\text{out}})/\lambda^{\text{out}} (\mathcal I_2 + \mathcal P_\omega^{(1)})$ such that the topological derivative reads $d \mathcal J(\Omega)(z, \omega) = \frac{\alpha}{2} \lambda^{\text{out}} \nabla u(z)^T P_\omega^{(1)} \nabla u(z)$.
\end{remark}
\begin{proof}
For the first-order topological derivative, we use the definition~\eqref{definition_TD1} and Taylor expansion, as follows
\begin{equation*}
\begin{split}
\mathcal{J}(\Omega_\varepsilon)-\mathcal{J}(\Omega)=J(\ue,\Omega_\varepsilon)-J(u,\Omega)
  &=J(\ue,\Omega_\varepsilon)-J(\ue,\Omega)+J(\ue,\Omega)-J(u,\Omega)
\\&=J(\ue,\Omega_\varepsilon)-J(\ue,\Omega)+J'(u,\Omega)(\ue-u)+\frac12J''(u)(\ue - u)^2,
\end{split}
\end{equation*}
where we used the fact that $J$ is quadratic in $u$. Due to~\eqref{eq_Jprimezero}, we have that $J'(u, \Omega)(w)=0$ for all $w \in H^1(D)$. Thus, plugging in~\eqref{unper_cost_func}, we get
\[
\mathcal{J}(\Omega_\varepsilon)-\mathcal{J}(\Omega)=\frac{\alpha}{2}(\lambda^{\text{in}}-\lambda^{\text{out}})\int_{\omega_\varepsilon}|\nabla \ue|^2\dx+\frac12\int_D (\ue-u)^2+\alpha\lambda_\Omega|\nabla\ue-\nabla u|^2\dx.
\]
Some rearrangements to the last equation lead to
\begin{equation*}
\begin{split}
\mathcal{J}(\Omega_\varepsilon)-\mathcal{J}(\Omega)
  =&\alpha(\lambda^{\text{in}}-\lambda^{\text{out}})\int_{\omega_\varepsilon}\nabla(\ue-u)\cdot\nabla u\dx+\frac{\alpha}{2}(\lambda^{\text{in}}-\lambda^{\text{out}})\int_{\omega_\varepsilon}|\nabla u|^2\dx
  \\&+\frac12\int_D (\ue-u)^2+\alpha\lambda_{\Omega_\varepsilon}|\nabla\ue-\nabla u|^2\dx.
\end{split}
\end{equation*}
By considering $v=\ue-u$ in Equation~\eqref{state_difference}, one gets 
\begin{equation}
\label{TD1_Delfour}
\mathcal{J}(\Omega_\varepsilon)-\mathcal{J}(\Omega)=\frac{\alpha}{2}(\lambda^{\text{in}}-\lambda^{\text{out}})\int_{\omega_\varepsilon}\nabla(\ue-u)\cdot\nabla u\dx+\frac{\alpha}{2}(\lambda^{\text{in}}-\lambda^{\text{out}})\int_{\omega_\varepsilon}|\nabla u|^2\dx.
\end{equation}
First, we compute the limit of the first term in~\eqref{TD1_Delfour} when divided by the volume $|\omega_\varepsilon|$,
\[
\mathcal{R}_1^{(1)}(u):=\lim_{\varepsilon\rightarrow 0}\frac{1}{|\omega_\varepsilon|}\frac{\alpha}{2}(\lambda^{\text{in}}-\lambda^{\text{out}})\int_{\omega_\varepsilon}\nabla(\ue-u)\cdot\nabla u\dx.
\]
Using the definition of the quantity $K^{(1)}_{\varepsilon,\omega}$~\eqref{first_vari_state} and making a change of variable, we have
\begin{equation}
\label{eq_reste121}
\begin{split}
\mathcal{R}_1^{(1)}(u)
  &=\lim_{\varepsilon\rightarrow 0}\frac{1}{|\omega_\varepsilon|}\frac{\alpha}{2}(\lambda^{\text{in}}-\lambda^{\text{out}})\int_{\omega_\varepsilon}(\nabla K^{(1)}_{\varepsilon,\omega})\circ\phi_{\varepsilon}^{-1}\cdot\nabla u\dx
\\&=\lim_{\varepsilon\rightarrow 0}\frac{1}{|\omega|}\frac{\alpha}{2}(\lambda^{\text{in}}-\lambda^{\text{out}})\int_{\omega}\nabla K^{(1)}_{\varepsilon,\omega}\cdot\nabla u\circ\phi_{\varepsilon}\dx,
\end{split}
\end{equation}
where we again used $\nabla(w\circ\phi_{\varepsilon}^{-1})=\varepsilon^{-1}\nabla w\circ\phi_{\varepsilon}^{-1}$ as well as $\text{det}(\partial\phi_{\varepsilon})=\varepsilon^2$ and $|\omega_\varepsilon|=\varepsilon^2|\omega|$.
For passing to the limit $\varepsilon\rightarrow 0$, using~\eqref{eq_Keps_to_K}, considering the smoothness of $u$ in a neighbourhood of $z$ and exploiting the boundedness of $\nabla K_{\varepsilon,\omega}^{(1)}$~\eqref{eq_Keps_bounded}, Lebesgue's dominated convergence theorem yields
\begin{equation}
\label{eq_reste12}
\mathcal{R}_1^{(1)}(u)=\frac{1}{|\omega|}\frac{\alpha}{2}(\lambda^{\text{in}}-\lambda^{\text{out}})\int_{\omega}\nabla K^{(1)}_{\omega}[\nabla u(z)]\cdot\nabla u(z)\dx.
\end{equation}
Similarly, we get for the last term in~\eqref{TD1_Delfour}, again after division by $|\omega_\varepsilon|$,
\begin{equation}
\label{eq_reste13}
\begin{split}
\mathcal{R}_2^{(1)}(u)
  :=&\lim_{\varepsilon\rightarrow 0}\frac{1}{|\omega_\varepsilon|}\frac{\alpha}{2}(\lambda^{\text{in}}-\lambda^{\text{out}})\int_{\omega_\varepsilon}|\nabla u|^2\dx
\\=&\frac{\alpha}{2}(\lambda^{\text{in}}-\lambda^{\text{out}})|\nabla u|^2(z).
\end{split}
\end{equation}
Combining these limits,~\eqref{eq_reste12} and~\eqref{eq_reste13}, we find the first-order topological derivative of the cost functional $\mathcal{J}$ as follows
\begin{equation} 
\label{eq_dJinproof}
d\mathcal{J}(\Omega)(z,\omega)=\frac{\alpha}{2}(\lambda^{\text{in}}-\lambda^{\text{out}})\left\{|\nabla u|^2(z)+\fint_{\omega}\nabla K^{(1)}_{\omega}[\nabla u(z)]\cdot\nabla u(z)\dx\right\},
\end{equation}
where $\fint_{\omega}v\dx=\frac{1}{|\omega|}\int_{\omega}v\dx$. 

Finally, considering the linearity of $\zeta \mapsto K_\omega^{(1)}[\zeta]$ according to~\eqref{linear_corrector_function} and the definition of the weak polarisation matrix~\eqref{first_order_weak_polarisation_matrix}, we write the first-order topological derivative at the point $z\in D\setminus\overline{\Omega}$ with respect to an inclusion shape $\omega$ as
 \[
d\mathcal{J}(\Omega)(z,\omega)=\frac{\alpha}{2}(\lambda^{\text{in}}-\lambda^{\text{out}})\nabla u^T(z)\left[\mathcal{I}_2+\mathcal{P}^{(1)}_{\omega}\right]\nabla u(z).
\]
which concludes the proof.
\end{proof}
\begin{theorem}
\label{THTD2}
Let a spatial point $z \in D \setminus \overline \Omega$ and an inclusion shape $\omega \subset \R^2$ with $\bm{0} \in \omega$ be given.
The second-order topological derivative at $z$ with respect to $\omega$ is given by
\begin{equation}
\label{TD2}
d^2\mathcal{J}(\Omega)(z,\omega)=\alpha(\lambda^{\text{in}}-\lambda^{\text{out}}) \text{vec}(\nabla^2 u(z))^\top \left[\mathcal{X}+\mathcal{P}^{(2)}_{\omega}\right] \nabla u(z),
\end{equation}
where $\mathcal{X}=\mathcal{I}_2\otimes\fint_{\omega} x\dx \in \R^{4 \times 2}$ and $\mathcal{P}^{(2)}_\omega$ denotes the matrix defined as follows
\begin{equation}
\label{second_order_weak_polarisation_matrix}
\mathcal{P}^{(2)}_\omega=\left[\frac{1}{|\omega|}\int_\omega\nabla K^{(1)}_\omega[\bm{e}^{(1)}]\otimes x\dx\quad\quad\frac{1}{|\omega|}\int_\omega\nabla K^{(1)}_\omega[\bm{e}^{(2)}]\otimes x\dx\right]\in\R^{4\times 2},
\end{equation}
where $K^{(1)}_\omega[\bm{e}^{(k)}]$ is defined in Definition~\ref{definition_first_corrector_function}, for $k=1,2$.
\end{theorem}
\begin{proof}
For the second-order term of the asymptotic expansion of $\mathcal{J}$, we use the definition of the topological derivative, Equation~\eqref{definition_TD2}, as follows
\begin{equation}
\label{TD2_delfour}
\begin{split}
\mathcal{J}(\Omega_\varepsilon)-\mathcal{J}(\Omega)-|\omega_\varepsilon|d\mathcal{J}(\Omega)(z,\omega)
&=\frac{|\omega_\varepsilon|}{|\omega|} \frac{\alpha}{2}(\lambda^{\text{in}}-\lambda^{\text{out}})\left\{\int_{\omega}\nabla K^{(1)}_{\varepsilon,\omega}\cdot\nabla u\circ\phi_{\varepsilon}+|\nabla u\circ\phi_{\varepsilon}|^2\dx\right\}
\\&-\frac{|\omega_\varepsilon|}{|\omega|}\frac{\alpha}{2}(\lambda^{\text{in}}-\lambda^{\text{out}})\left\{\int_{\omega}\nabla K^{(1)}_{\omega}\cdot\nabla u(z)+|\nabla u|^2(z) \dx\right\},
\end{split}
\end{equation}
where~\eqref{eq_reste121},~\eqref{eq_reste13} and~\eqref{eq_dJinproof} are also exploited. Then, we start our analysis with the first term
\[
\mathcal{R}_1^{(2)}(u):=\lim_{\varepsilon\rightarrow 0}\frac{1}{\varepsilon|\omega|}\frac{\alpha}{2}(\lambda^{\text{in}}-\lambda^{\text{out}})\int_{\omega}\nabla K^{(1)}_{\varepsilon,\omega}\cdot\nabla u\circ\phi_{\varepsilon}-\nabla K^{(1)}_{\omega}\cdot\nabla u(z)\dx.
\]
Rearranging the last equation gives
\[
\mathcal{R}_1^{(2)}(u)=\lim_{\varepsilon\rightarrow 0}\frac{\alpha}{2}(\lambda^{\text{in}}-\lambda^{\text{out}})\fint_{\omega}\varepsilon^{-1}\nabla\left(K^{(1)}_{\varepsilon,\omega}-K^{(1)}_{\omega}\right)\cdot\nabla u\circ\phi_{\varepsilon}+\nabla K^{(1)}_{\omega}\cdot\varepsilon^{-1}\left(\nabla u\circ\phi_{\varepsilon}-\nabla u(z)\right)\dx.
\]
Passing to the limit $\varepsilon\rightarrow 0$ and exploiting~\eqref{eq_Keps_to_K} and~\eqref{eq_K2eps_to_K2} as well as  the smoothness of $u$ near $z$ yields
\begin{equation}
\label{eq_reste22}
\mathcal{R}_1^{(2)}(u)=\frac{\alpha}{2}(\lambda^{\text{in}}-\lambda^{\text{out}})\fint_{\omega}\nabla K^{(2)}_{\omega}[\nabla u(z)]\cdot\nabla u(z)+\nabla K^{(1)}_{\omega}[\nabla u(z)]\cdot\nabla^2u(z)x\dx.
\end{equation}
Similarly, by Taylor expansion and assuming smoothness of $u$ near $z$, we get for the remaining terms in~\eqref{TD2_delfour}
\begin{equation}
\label{eq_reste23}
\begin{split}
\mathcal{R}_2^{(2)}(u)
  :=&\lim_{\varepsilon\rightarrow 0}\frac{1}{\varepsilon|\omega|}\frac{\alpha}{2}(\lambda^{\text{in}}-\lambda^{\text{out}})\int_{\omega}|\nabla u\circ\phi_{\varepsilon}|^2-|\nabla u|^2(z)\dx
\\=&\alpha(\lambda^{\text{in}}-\lambda^{\text{out}})\fint_{\omega}\nabla^2u(z)x\cdot\nabla u(z)\dx.
\end{split}
\end{equation}
Combining these limits,~\eqref{eq_reste22} and~\eqref{eq_reste23}, we find the second-order topological derivative of the cost functional $\mathcal{J}$,
\begin{equation*}
\begin{split}
d^2\mathcal{J}(\Omega)(z,\omega)
  &=\frac{\alpha}{2}(\lambda^{\text{in}}-\lambda^{\text{out}})\fint_{\omega}\nabla K^{(2)}_{\omega}[\nabla u(z)]\cdot\nabla u(z)+\nabla K^{(1)}_{\omega}[\nabla u(z)]\cdot\nabla^2u(z)x\dx
  \\&+\alpha(\lambda^{\text{in}}-\lambda^{\text{out}})\fint_{\omega}\nabla^2u(z)x\cdot\nabla u(z)\dx.
\end{split}
\end{equation*}
Considering first $v=K^{(2)}_{\omega}[\nabla u(z)]$ in Equation~\eqref{corrector_vari_form} and then $v=K^{(1)}_{\omega}[\nabla u(z)]$ in Equation~\eqref{second_corrector_vari_form}, we get  
\[
d^2\mathcal{J}(\Omega)(z,\omega)=\alpha(\lambda^{\text{in}}-\lambda^{\text{out}})\fint_{\omega}\nabla K^{(1)}_{\omega}[\nabla u(z)]\cdot\nabla^2u(z)x+\nabla^2u(z)x\cdot\nabla u(z)\dx.
\]
Finally, we write the second-order topological derivative at the point $z \in D\setminus\overline{\Omega}$ with respect to an inclusion shape $\omega$, as follows
\[
d^2\mathcal{J}(\Omega)(z,\omega)=\alpha(\lambda^{\text{in}}-\lambda^{\text{out}}) \text{vec}(\nabla^2 u(z))^\top \left[\mathcal{X}+\mathcal{P}^{(2)}_{\omega}\right] \nabla u(z),
\]
with $\mathcal{X}=\mathcal{I}_2\otimes\fint_{\omega} x\dx \in \R^{4\times 2}$ and the matrix $\mathcal P_\omega^{(2)}$ defined in~\eqref{second_order_weak_polarisation_matrix}, where we again used the linearity of the corrector function~\eqref{linear_corrector_function}.
\end{proof}
\begin{remark}
For a specific shape of perturbation (e.g. circular or elliptic-shaped inclusion), we derive an exact formula of the topological derivative $d^2\mathcal{J}(\Omega)(z,\omega)$. In this case, we have to compute a numerical approximation of the quantity $K^{(1)}_{\omega}[\nabla u(z)]\in\dot{BL}(\R^2)$, for more details see Section~\ref{choice_pixel_configurations}.
\end{remark}

%%%%%%%%%%%%%%%%%%%%%%%%%%%%%%%%
\section{A numerical vertex detection method}
\label{numerical_vertex_detection_method}
In this section, we first present numerical techniques based on the second-order topological derivative obtained in Theorem~\ref{THTD2} before assessing the developed approach's effectiveness in a set of numerical experiments. 

%%%%%%%%%%%%%%%%%%%%%%%%%%%%%%%%
\subsection{Numerical algorithm}
\label{numerical_algorithm}
The proposed algorithm is based on the numerical computation of second-order topological derivatives of the reduced cost function~\eqref{eq_reducedCost} with respect to particular inclusion shapes in order to detect junctions of lines of different configurations, cf. Fig.~\ref{fig_classification_vertices}. Recall that the analysis of Section~\ref{asymptotic_expansion_J} is valid for arbitrary inclusion shapes $\omega$ with $0 \in \omega$. On the one hand, we present a method for detecting the location of a particular given inclusion shape. On the other hand, we present how this procedure can be used to detect vertex classes and locations within an image simultaneously.

%%%%%%%%%%%%%%%%%%%%%%%%%%%%%%%%
\subsubsection{Choice of inclusion shapes}
\label{choice_pixel_configurations}
One needs to choose thoughtfully the inclusion shape $\omega$ covering all the classes of vertices, see Fig.~\ref{fig_classification_vertices}. To represent all different classes of vertices, we consider four angles (i.e. four lines) since four is the maximum number of lines to represent 'K-junction', 'X-junction', 'Peak', 'Multi-junction', see Fig.~\ref{fig_classification_vertices}(e)--(h). Each angle can attain values in the interval $[0^\circ,360^\circ]$, see Fig.~\ref{fig_exterior_pb} for some examples of the numerically computed first order corrector $K_\omega^{(1)}$ for different inclusion shapes $\omega$ resembling the vertex classes of Fig.~\ref{fig_classification_vertices}.

In order to illustrate these numerical computations, we consider an example of the configuration 'L-corner' (see Fig.~\ref{fig_classification_vertices}(a)) where two lines (represented by two points $P, Q \in \R^2$ with $|P|=|Q|=1$) meet at the origin $O=(0,0)^T$ with angles deg(P)=$0^\circ$ and deg(Q)=$45^\circ$, as depicted in Fig.~\ref{fig_enlargement}(a). Then, the volume $\omega$ is an enlargement of the union of the two line segments $\vec{OP}$ and $\vec{OQ}$ as it is depicted in Fig.~\ref{fig_enlargement}. Here, the line segments $\vec{P_1 P_2}$ and $\vec{Q_1 Q_2}$ are orthogonal to the line segments $\vec{OP}$ and $\vec{OQ}$ and pass through $P$ and $Q$ half way, respectively. Their length is given by the variable $w_{PQ}$ which is set to $w_{PQ}=0.05$ throughout this study. The point $PQ_1$ in Fig.~\ref{fig_enlargement} is the intersection point of that line parallel to $\vec{OP}$ passing through $P_1$ and that line parallel to $\vec{OQ}$ passing through $Q_2$. The point $PQ_2$ is defined analogously. In the same way, the other vertex classes defined by three or four lines in Fig.~\ref{fig_classification_vertices}(b)--(h), e.g. 'Fork' or 'Peak' junctions are treated. The enlargement of the configuration is considered as the inclusion shape $\omega$. We will reference particular shapes $\omega$ by the angles between the positive $x$-axis and the points defining them, i.e., by deg(P) and deg(Q) in the case of two lines, and additionally by deg(R) in the case of three and deg(S) in the case of four lines, see Fig.~\ref{fig_classification_vertices}(b)-(d) and (e)-(h), respectively. We will refer to the inclusion shapes obtained in this way by $\omega[\text{deg}(P), \text{deg}(Q)]$, $\omega[\text{deg}(P), \text{deg}(Q), \text{deg}(R)]$ and $\omega[\text{deg}(P), \text{deg}(Q), \text{deg}(R), \text{deg}(S)]$ in the case of two, three or four lines, respectively.
\ifpics
\begin{figure}%[h]
\centering
	\begin{tikzpicture}[scale=1.5,baseline={(0,0)}]
		\node at (1.25,-1) {(a)};
		\draw[thick] (0,0) -- (3,0);
		\draw[thick] (0,0) -- (2.12,2.12);
  
		\draw[ao(english)] (3,.225) -- (0,.225);
		\draw[ao(english)] (0,-.225) -- (3,-.225);
		\draw[ao(english)] (3,-.225) -- (3,.225);
		\draw[dashed,ao(english)] (-.56,-.225) -- (0,-.225);
        \draw[dashed,ao(english)] (-.56,.225) -- (0,.225);
  
		\draw[slategray] (1.96,2.28) -- (-.324,0);
		\draw[slategray] (0,-.318) -- (2.28,1.96);
		\draw[slategray] (2.28,1.96) -- (1.96,2.28);
		\draw[dashed,slategray] (-.324,0) -- (-.56,-.225);
        \draw[dashed,slategray] (-.324,0) -- (-.56,-.225);
        
		\draw (3.2,0) node {P};
		\draw (3.125,.27)  node[ao(english)] {\scriptsize P$_1$};
		\draw  (3.127,-.27)  node[ao(english)] {\scriptsize P$_2$};
  
		\draw [decorate,decoration={brace,amplitude=5pt,mirror,raise=4ex}] (-.225,.225) -- (-.225,-.225) node[midway,xshift=-4.25em]{$w_{PQ}$};

		\draw (2.32,2.32)  node{Q};
		\draw (2,2.375)  node[slategray]{\scriptsize Q$_1$};
		\draw (2.4,2)  node[slategray]{\scriptsize Q$_2$};
  
		\draw[fill=black] (.543,.225) circle (0.5pt);
		\draw (.75,.275) node {\scriptsize PQ$_1$};
		\draw[fill=black] (-.56,-.225) circle (0.5pt);
		\draw (-.56,-.35) node {\scriptsize PQ$_2$};
	\end{tikzpicture}
	\hspace*{.75cm}
	\begin{tikzpicture}[scale=1.5,baseline={(0,0)}]
		\node at (1.25,-1) {(b)};
		\draw[thick] (0,0) -- (3,0);
		\draw[thick] (0,0) -- (2.12,2.12);
		\draw[red] (3,.225) -- (.543,.225);
		\draw[red] (.543,.225) -- (2.28,1.96);
		\draw[red] (2.28,1.96) -- (1.96,2.28);
		\draw[red] (1.96,2.28) -- (-.56,-.225) ;
		\draw[red] (-.56,-.225) -- (3,-.225);
		\draw[red] (3,-.225) -- (3,.225);
		\draw (3.125,.27)  node {P$_1$};
		\draw (.8,.3) node {PQ$_1$};
		\draw (2.4,2)  node {Q$_2$};
		\draw (2,2.375)  node {Q$_1$};
		\draw (-.56,-.35) node {PQ$_2$};
		\draw  (3.127,-.27)  node {P$_2$};
	\end{tikzpicture}
\caption{Delineation of a vertices configuration and the corresponding enlargement set.}
\label{fig_enlargement}
\end{figure}
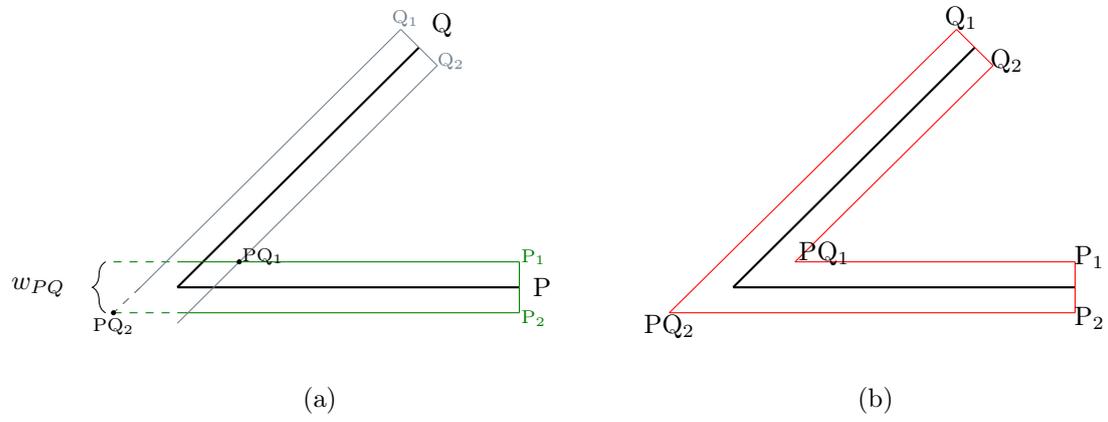
\fi

As explained in Section~\ref{topological_derivatives_cost_function}, one cannot calculate an analytic solution to the exterior problem~\eqref{corrector_vari_form}, defined at the unbounded domain $\R^2$, see Fig.~\ref{fig_trunc_domain}(a). Thus, we adapt the technique developed in the works~\cite{AmsGan19, GanNeeSti24, GanStu22}, where the exterior problem is approximated numerically by truncating the domain at a large radius $R$ (e.g. $R=30$) and using a finite element discretization with homogeneous Dirichlet boundary conditions on the truncated domain boundary, see Fig.~\ref{fig_trunc_domain}(b). 
\ifpics
\begin{figure}%[h]
\centering
	\begin{tikzpicture}[scale=.25]
		\node at (0,-13) {(a)};
		\draw (12,5) node[anchor=north east] {$\omega$};
		\fill[gray,path fading=fade out] (0,0) circle (10cm);
		\draw[->,dashed,thick] (0,-11) -- (0,11); 
		\draw[->,dashed,thick] (-11,0) -- (11,0); 
		\draw[fill=black] (0,0) circle (4pt);
		\draw[red] (3,.225) -- (.543,.225);
		\draw[red] (.543,.225) -- (2.28,1.96);
		\draw[red] (2.28,1.96) -- (1.96,2.28);
		\draw[red] (1.96,2.28) -- (-.56,-.225) ;
		\draw[red] (-.56,-.225) -- (3,-.225);
		\draw[red] (3,-.225) -- (3,.225);
		\draw (7,7) node[anchor=north east] {$\R^2$};
	\end{tikzpicture}
	\hspace{1.5cm}
	\begin{tikzpicture}[scale=.25]
		\node at (0,-13) {(b)};
		\draw (12,5) node[anchor=north east] {$\omega$};
		\draw[->,dashed,thick] (0,-11) -- (0,11); 
		\draw[->,dashed,thick] (-11,0) -- (11,0); 
		\draw [fill=gray(x11gray)] circle (10cm);
		\clip[draw] circle (10cm);
		\draw[->,dashed,thick] (0,-11) -- (0,11); 
		\draw[->,dashed,thick] (-11,0) -- (11,0); 
		\draw[fill=black] (0,0) circle (4pt);
		\draw[red] (3,.225) -- (.543,.225);
		\draw[red] (.543,.225) -- (2.28,1.96);
		\draw[red] (2.28,1.96) -- (1.96,2.28);
		\draw[red] (1.96,2.28) -- (-.56,-.225) ;
		\draw[red] (-.56,-.225) -- (3,-.225);
		\draw[red] (3,-.225) -- (3,.225);
		\draw (7,7) node[anchor=north east] {$B_R$};
	\end{tikzpicture}
\caption{(a) Rescaled perturbed domain after passing to limit. (b) Truncated domain $B_R$.}
\label{fig_trunc_domain}
\end{figure}
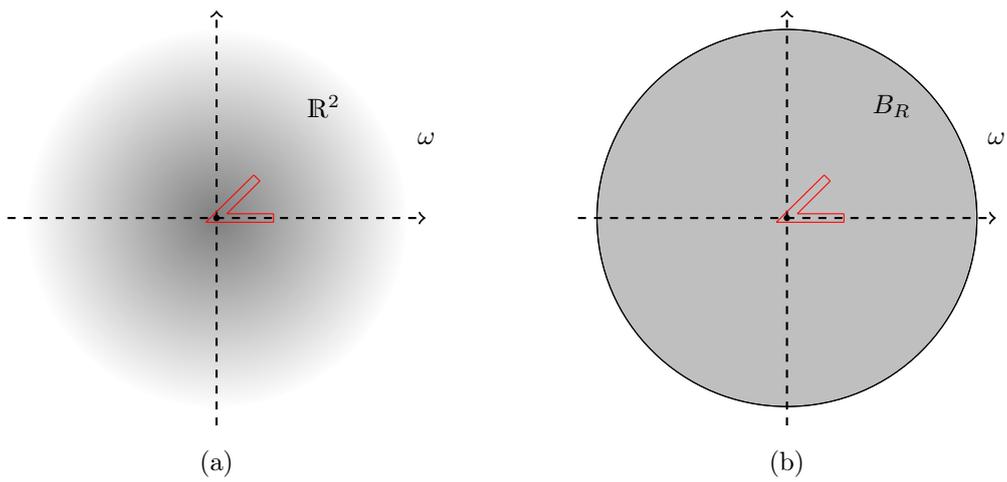
\fi

%%%%%%%%%%%%%%%%%%%%%%%%%%%%%%%%
\subsubsection{One-shot detection process for given inclusion shape} 
\label{one_shot_detection_process}
Here, we present the different steps for our one-shot detection process for a given inclusion shape $\omega$ representing the vertex configuration to be detected, cf. Fig.~\ref{fig_classification_vertices}.
\begin{enumerate}
\item Solve the boundary value problem~\eqref{eq_state_strongform},
\[
    \left\{\begin{aligned}
    -\alpha\mbox{ div}(\lambda_\Omega\nabla u)+u=&f&&\mbox{in }D,                  \\
    \frac{\partial u}{\partial n}=&0&&\mbox{on }\partial D,         
    \end{aligned} \right.
\]   
where $f$ represents the intensity values of the considered image (i.e. a given data). Here, $\Omega = \emptyset$ and thus $\lambda_\Omega(x) =\lambda^{\text{out}}$.
\item For the given inclusion shape $\omega$, compute an approximation of the exterior problem~\eqref{corrector_vari_form} in the truncated domain $B_R(0)$, see Fig.~\ref{fig_trunc_domain}(b), i.e., find $K_\omega^{(1)} \in H^1_0(B_R(0))$ such that
\begin{align} \label{eq_K_trunc}
\int_{B_R(0)}\alpha\lambda_{\omega}\nabla K^{(1)}_{\omega}[\bm{e}^{(k)}]\cdot\nabla v\dx=-\alpha(\lambda^{\text{in}}-\lambda^{\text{out}})\int_{\omega}\bm{e}^{(k)}\cdot\nabla v\dx,\quad\text{for } k=1,2,
\end{align}
for all $v \in H^1_0(B_R(0))$ where 
\[  
\lambda_{\omega}(x)=
    \begin{cases}
      \lambda^{\text{in}},&x\in\omega,\\
    \lambda^{\text{out}},&x\in B_R(0)\setminus\overline{\omega}. 
    \end{cases} 
\]
In our experiments, we choose the radius as $R=30$ as it is also suggested in~\cite{GanStu22}.
See, e.g., Fig.~\ref{fig_exterior_pb} for numerical approximations of $K_\omega^{(1)}$ for different inclusion shapes $\omega$.
\item Calculate an approximation to the matrix $\mathcal P_\omega^{(2)}$~\eqref{second_order_weak_polarisation_matrix} for the given inclusion shape $\omega$,
\[
\mathcal{P}^{(2)}_\omega=\left[\frac{1}{|\omega|}\int_\omega\nabla K^{(1)}_\omega[\bm{e}^{(1)}]\otimes x\dx\quad\quad\frac{1}{|\omega|}\int_\omega\nabla K^{(1)}_\omega[\bm{e}^{(2)}]\otimes x\dx\right]\in\R^{4\times2}.
\]
\ifpics
\begin{figure}%[h]
\centering
\begin{tabular}{cc}
	\includegraphics[scale=0.19]{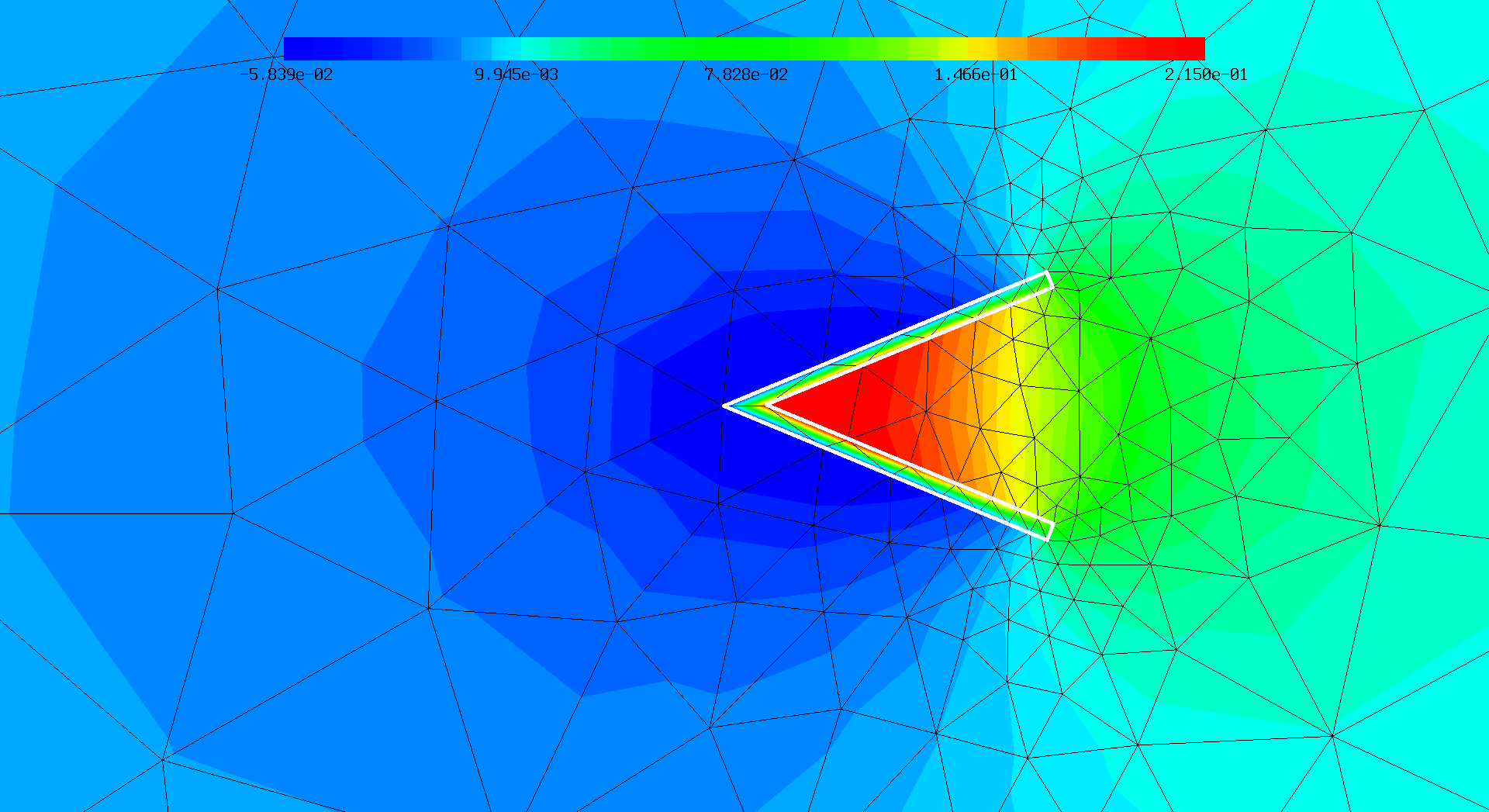}
	&
	\includegraphics[scale=0.19]{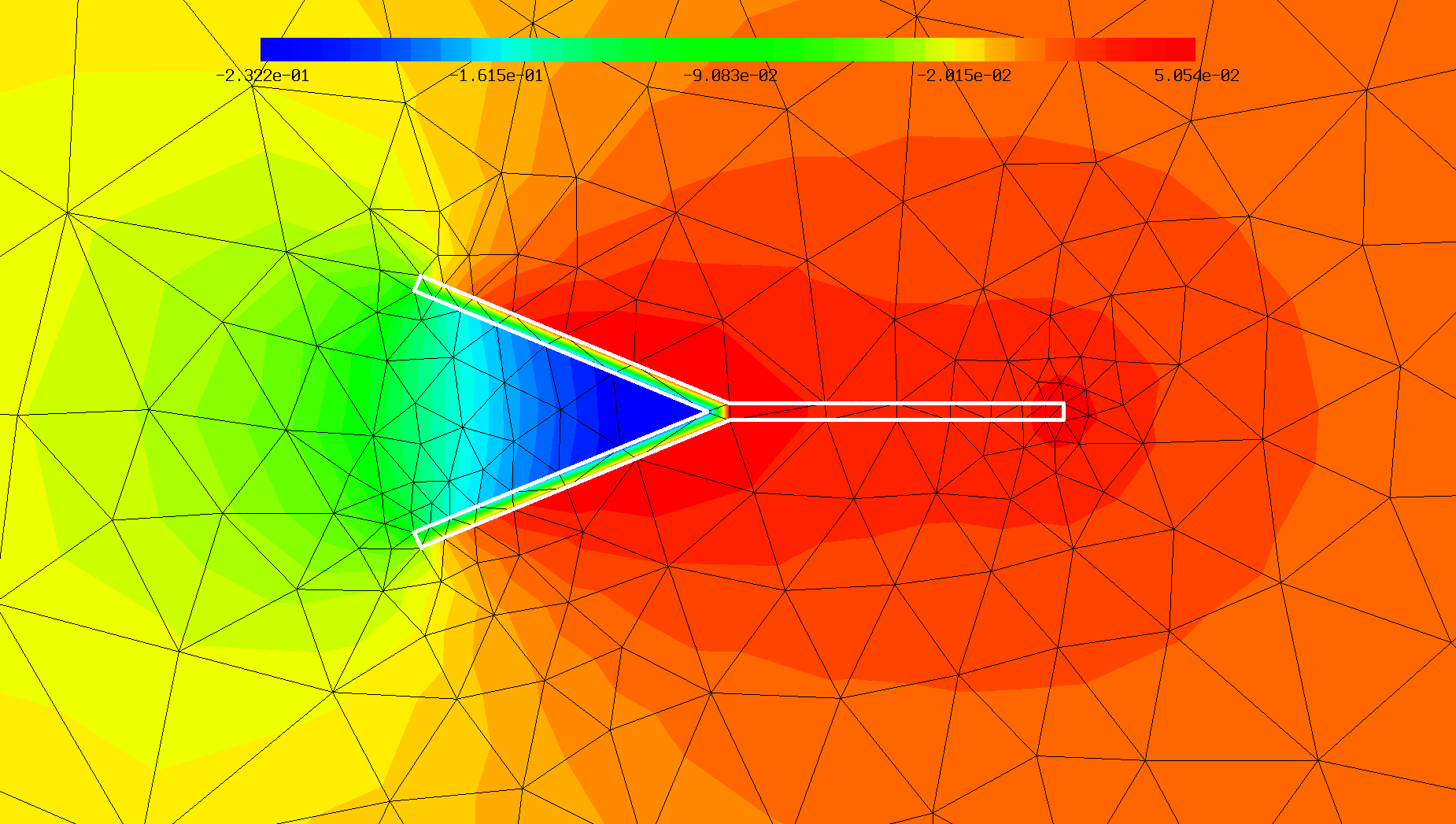} \\
 (a) & (b)\\
	\includegraphics[scale=0.19]{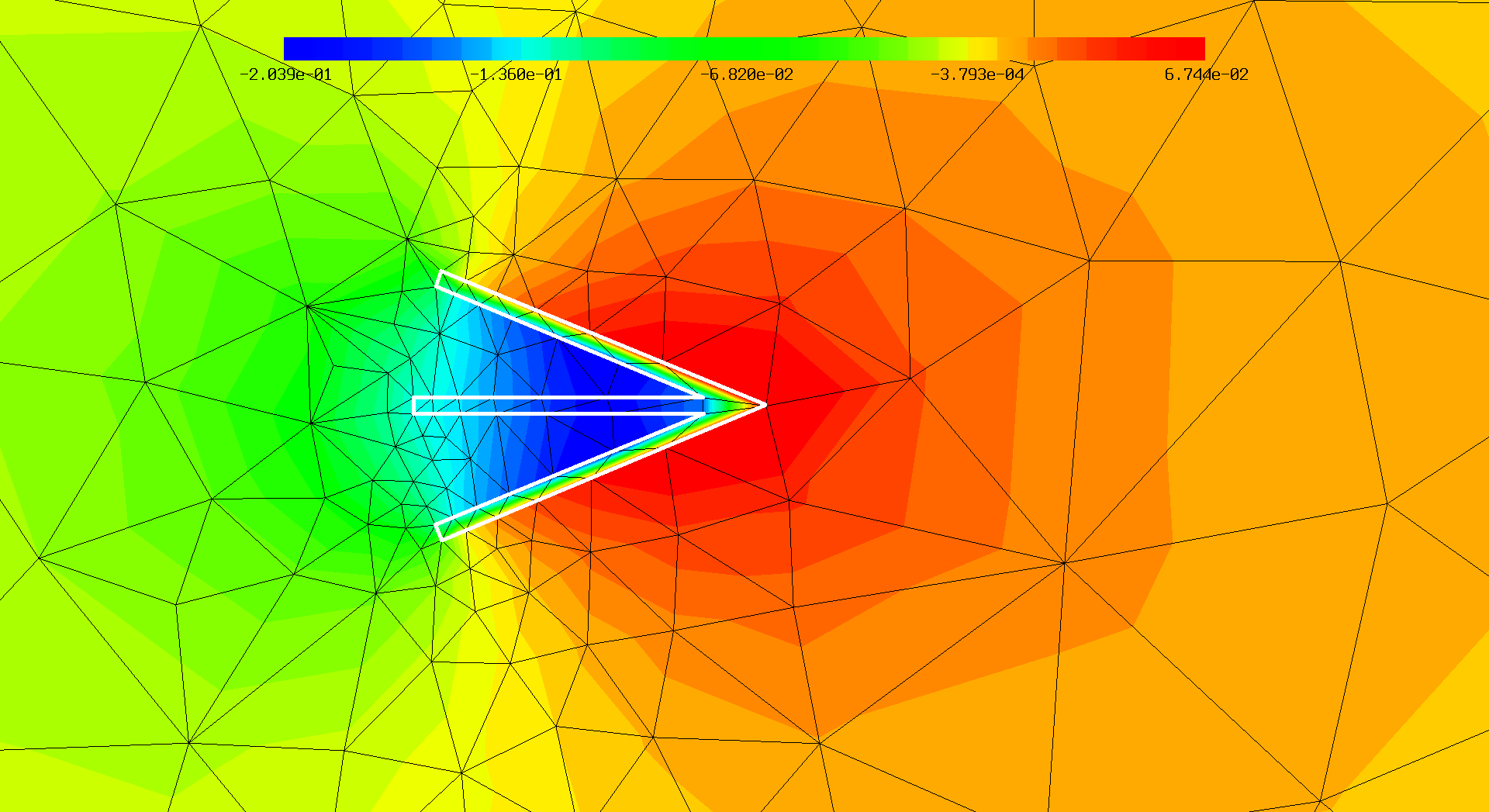}
	&
	\includegraphics[scale=0.19]{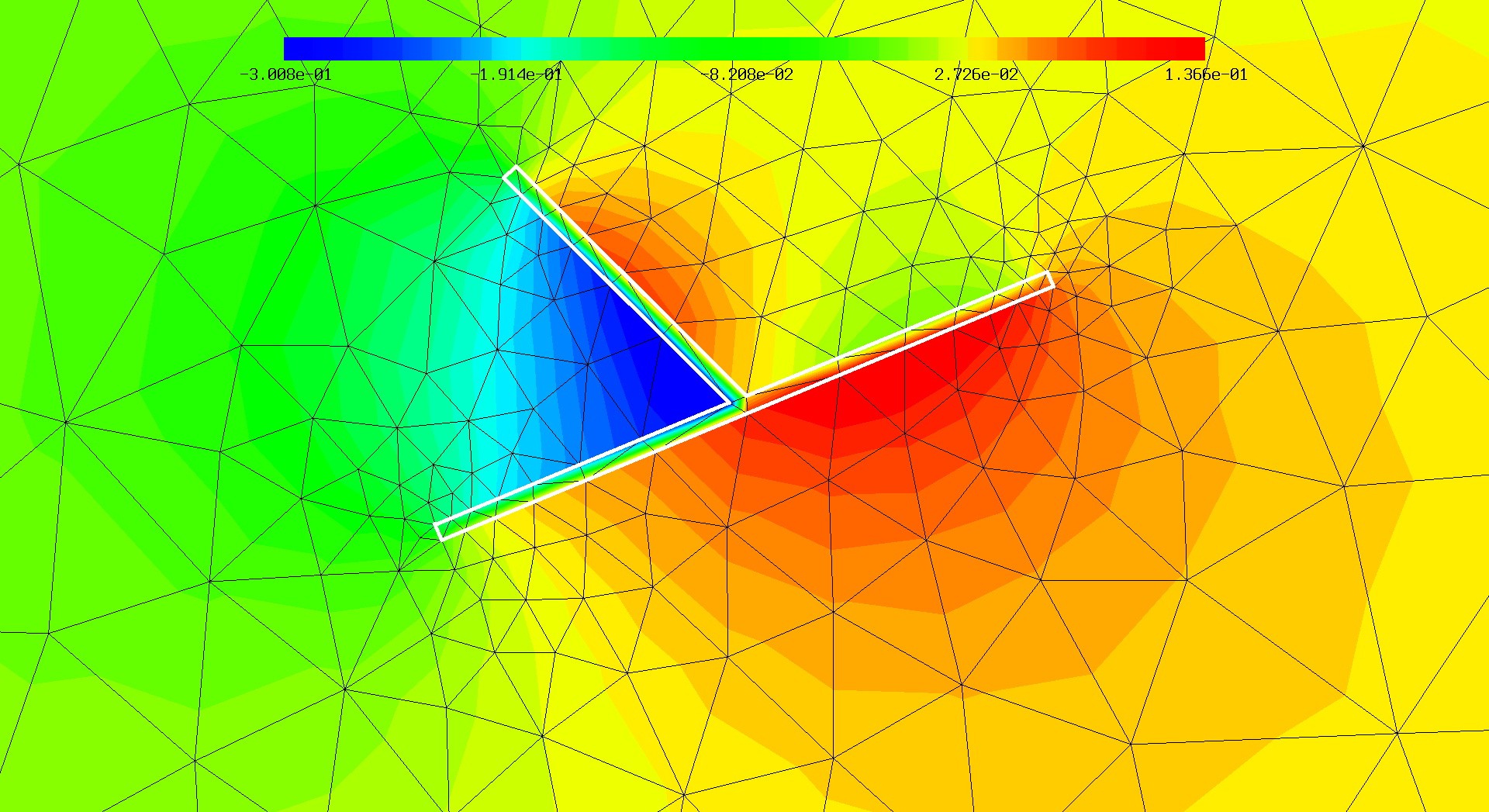}\\
 (c) & (d)\\
	\includegraphics[scale=0.19]{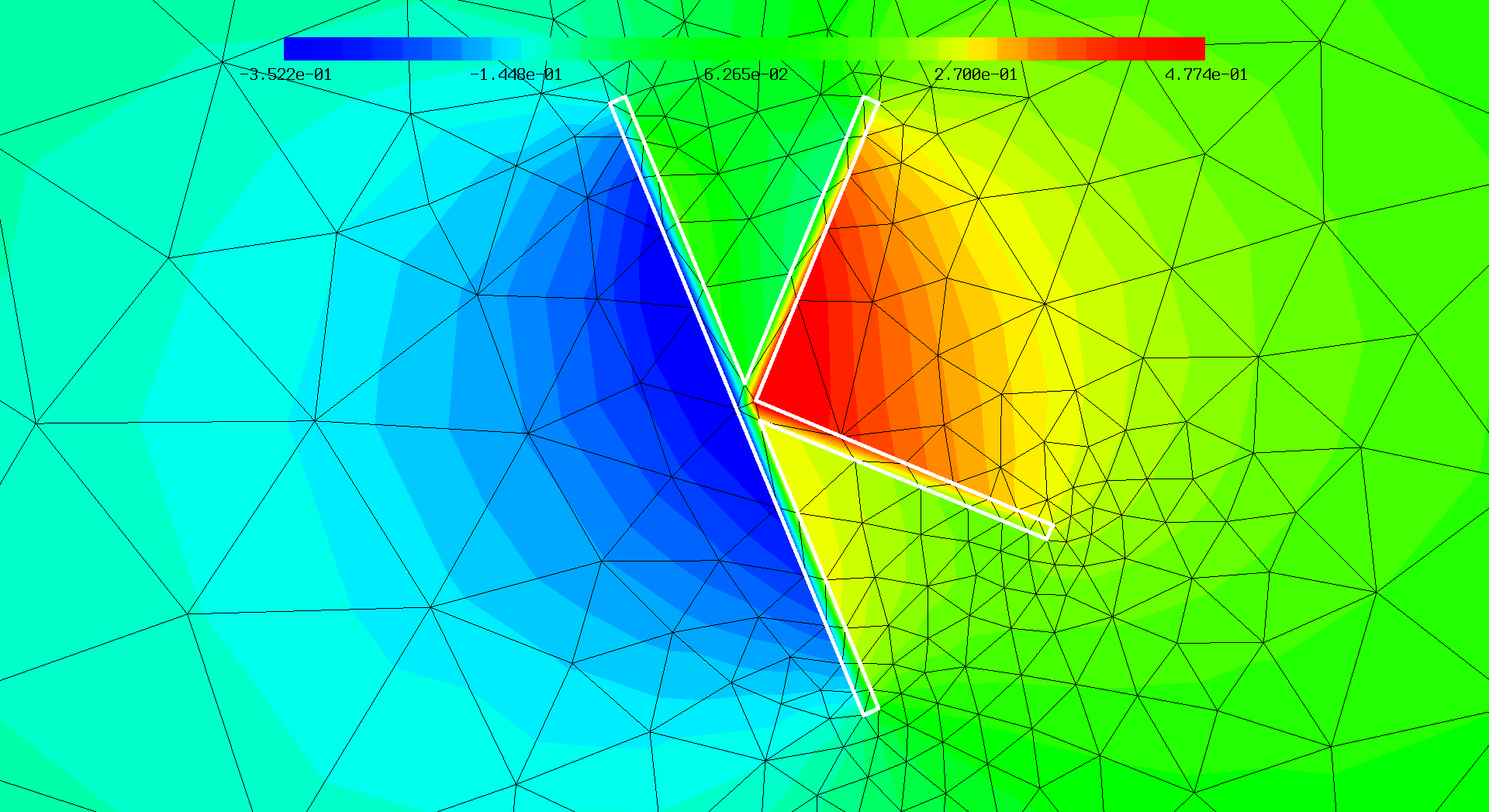}
	&
	\includegraphics[scale=0.19]{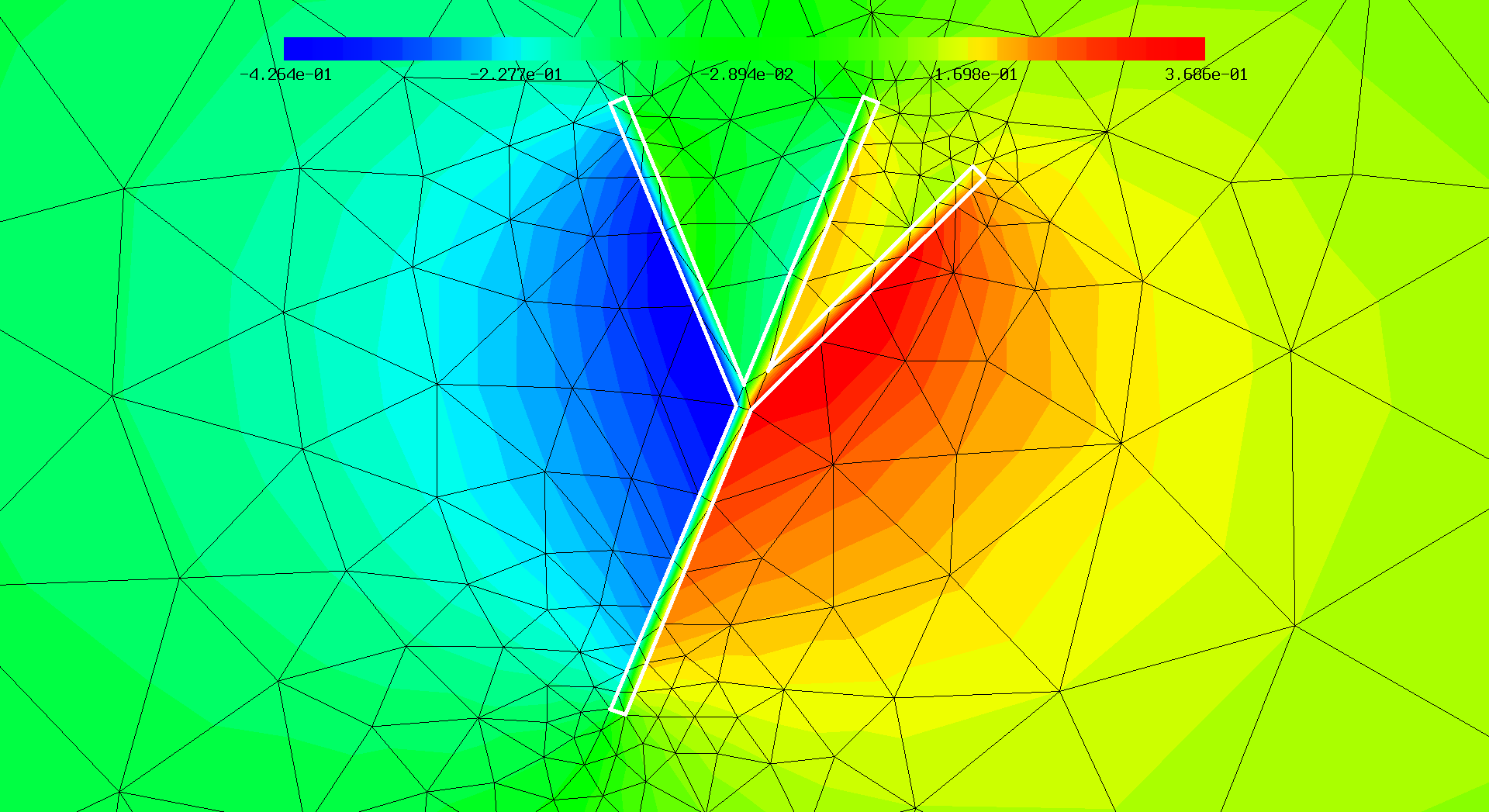}\\
 (e) & (f)\\
	\includegraphics[scale=0.19]{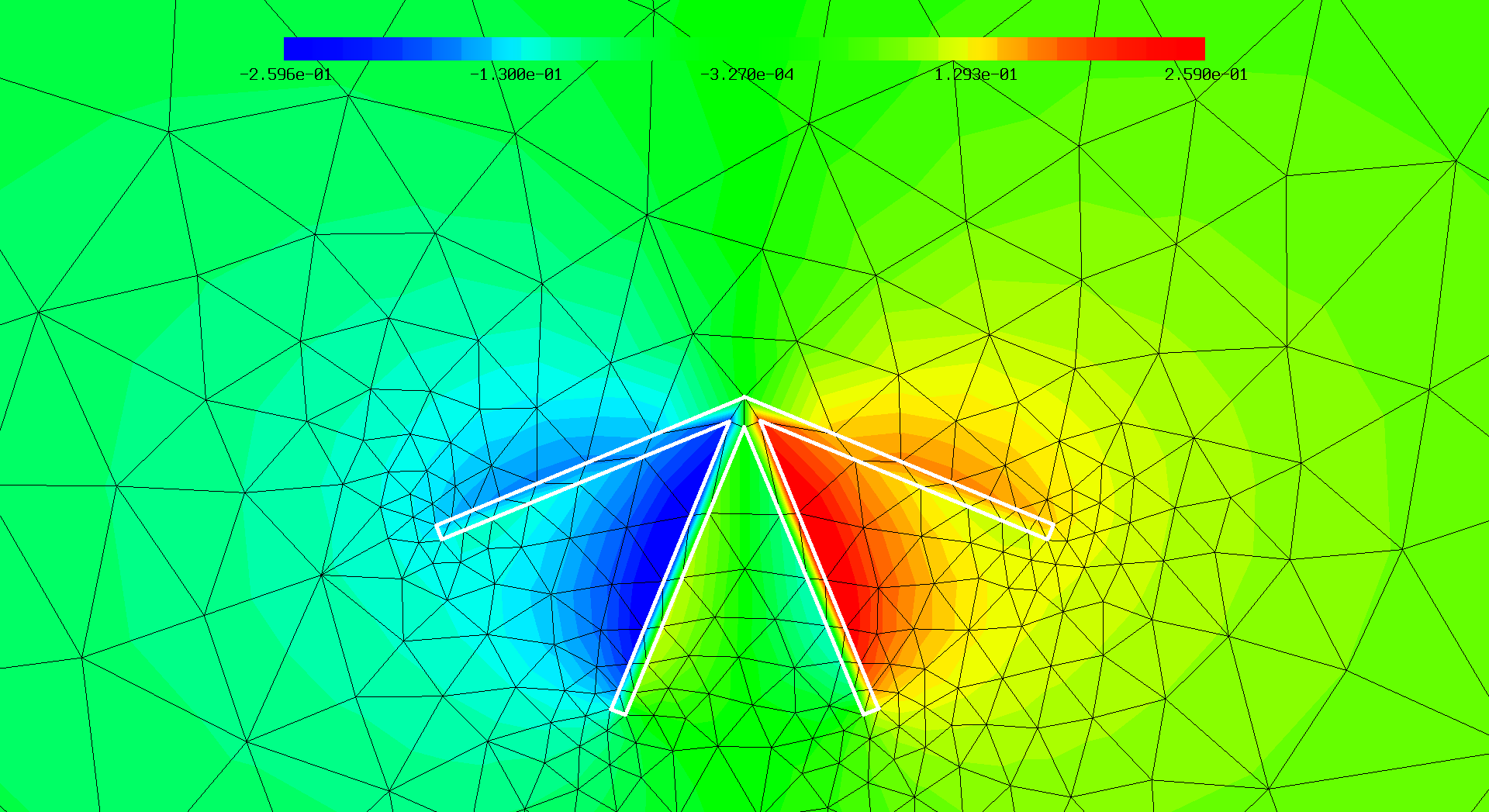}
	&
	\includegraphics[scale=0.19]{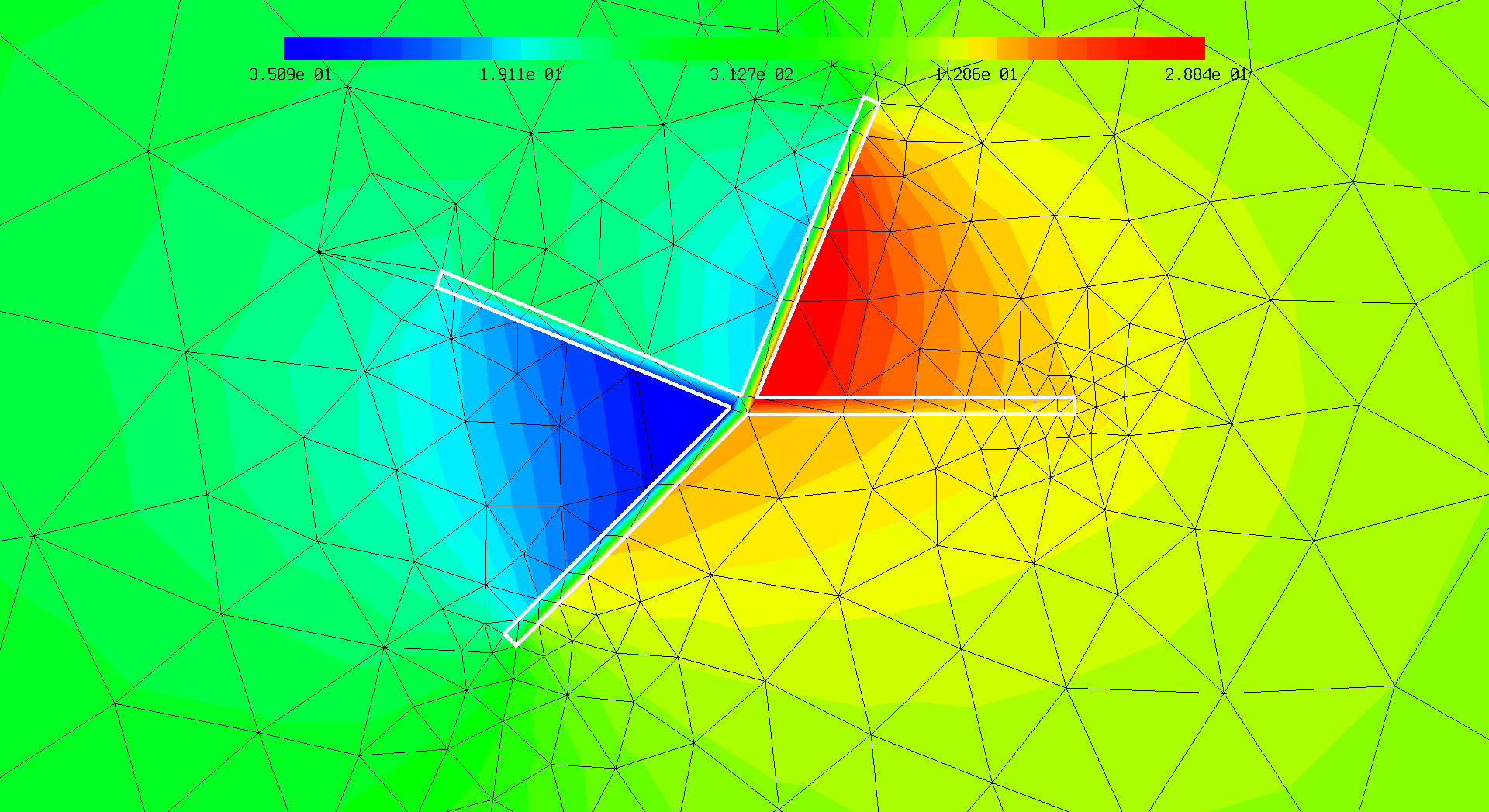}\\
 (g) & (h)
\end{tabular}
\caption{Numerical solution $K^{(1)}_{\omega}[\bm{e}^{(1)}]$ of the exterior problem~\eqref{corrector_vari_form} for $\lambda^{\text{in}}=0.05$, $\lambda^\text{out}=1$ and different examples of the inclusion shape $\omega$: (a) 'L-corner', (b) 'Fork', (c) 'Arrow', (d) 'T-junction', (e) 'K-junction', (f) 'X-junction', (g) 'Peak', (h) 'Multi-junction'.}
\label{fig_exterior_pb}
\end{figure}
\fi
\item Evaluate the second-order topological derivative~\eqref{TD2} with respect to the given inclusion shape $\omega$ at all spatial points $z\in D$,
\[
d^2\mathcal{J}(\Omega)(z,\omega)=\alpha(\lambda^{\text{in}}-\lambda^{\text{out}}) \text{vec}(\nabla^2 u(z))^\top \left[\mathcal{X}+\mathcal{P}^{(2)}_{\omega}\right] \nabla u(z).
\] 
\end{enumerate}
The vertex whose characteristics (the number of lines that form it and its angles) is represented by the inclusion shape $\omega$ is likely to be located at the point $z$ where the second-order topological derivative attains its most negative values. 

%%%%%%%%%%%%%%%%%%%%%%%%%%%%%%%%
\subsubsection{Detection of unknown inclusion shape} 
\label{sec_detectUnknown}
In realistic situations, the particular configurations to be recovered are not known. Due to the relatively low computational effort of the procedure described in Section~\ref{one_shot_detection_process} (only three boundary value problems have to be solved), our approach is still feasible to get information about both the location and the characteristics of the vertices appearing in the image simultaneously. For that purpose, we need to perform step (i) of Section~\ref{one_shot_detection_process}, i.e., the solution of the boundary value problem~\eqref{eq_state_strongform}, only once and compute the matrix $\mathcal P_\omega^{(2)}$ for a (possibly large) set of potential inclusions shapes $\Theta := \{ \omega^{(i)} \}_{i=1}^N$, (i.e., we need to solve the truncated exterior problem~\eqref{eq_K_trunc} $2N$ times). We then propose to rank all $N$ inclusion shapes with respect to the minimal value attained by the corresponding second-order topological derivative $d^2 \mathcal J(\Omega)(\cdot, \omega^{(i)})$. The resulting list gives information about which vertex configuration is most likely to be situated at which position. The procedure is summarized in Algorithm~\ref{algo_detectShape}. Note that the first for-loop, i.e., the calculation of the matrices $\mathcal P^{(2)}_{\omega^{(i)}}$, is independent of the particular image data. In fact, it only depends on the parameters $\alpha$, $\lambda^{\text{in}}$ and $\lambda^{\text{out}}$ and the set of inclusion shapes $\Theta$. Thus, when these parameters are kept fixed, the first for-loop has to be done only once and can be considered as a pre-computation phase.
\begin{algorithm}
\begin{algorithmic}
    \State{\textbf{Input:} intensity function $f:D \rightarrow \R$ representing image, set of inclusion shapes $\Theta := \{ \omega^{(i)} \}_{i=1}^N$}
    \For{$\omega^{(i)} \in \Theta$}
        \State{Solve approximated exterior problem~\eqref{eq_K_trunc} for $\omega=\omega^{(i)}$ for $k=1,2$ }
        \State{Calculate and store matrix $\mathcal P_\omega^{(2)}$ for $\omega =\omega^{(i)}$}
    \EndFor
    \State{Solve boundary value problem~\eqref{eq_state_strongform}}
    \For{$\omega^{(i)} \in \Theta$}
        \State{Read matrix $\mathcal P_\omega^{(2)}$ for $\omega =\omega^{(i)}$}
        \State{Evaluate second-order topological derivative~\eqref{TD2} for $\omega =\omega^{(i)}$ for all points $z \in D$}
    \EndFor
    \State{Rank inclusion shapes in ascending order according to $\underset{z \in D}{\mbox{min }} d^2 \mathcal J(\Omega)(z, \omega)$}
    \end{algorithmic}
    \caption{Detection algorithm}
    \label{algo_detectShape}
\end{algorithm}

In practice, the set of inclusion shapes $\Theta$ is chosen as follows: We make a subdivision of the interval $[0^\circ, 360^\circ]$ into $m$ angles $\alpha_j = j \Delta \alpha$, $j=0,\dots, m-1$ with precision $\Delta \alpha = 360/m$. When interested in vertices of type 'L-corner', we consider all inclusion shapes obtained by the procedure described in Section~\ref{choice_pixel_configurations} with $\text{deg}(P)=\alpha_j$ and $\text{deg}(Q)=\alpha_k$, $0 \leq j<k \leq m-1$, i.e., $\Theta = \Theta^{(2)} := \{ \omega[\alpha_j, \alpha_k]: 0 \leq j<k \leq m-1\}$. Similarly, when searching for vertex configurations consisting of three lines like 'Fork', 'Arrow' or 'T-junction', we define $\Theta = \Theta^{(3)} :=\{ \omega[\alpha_j, \alpha_k, \alpha_\ell]: 0 \leq j < k  < \ell \leq m-1\}$, and analogously in the case of four lines. 

%%%%%%%%%%%%%%%%%%%%%%%%%%%%%%%%
\subsection{Numerical experiments}
\label{numerical_experiments}
We consider the examples given in Fig.~\ref{fig_cube_example_vertices_classification} to illustrate the effectiveness of the proposed one-shot detection process. The image size is chosen as $100\times100$ pixels and we use piecewise bilinear, globally continuous finite elements on this grid. The parameters $\alpha=8$, $\lambda^{\text{in}}=0.05$ (representing {\it{edge}}) and $\lambda^{\text{out}}=1$ (representing {\it{non-edge}}) are chosen in the same way as in~\cite{BerGraMusSch14}. 
The image intensity is piecewise constant across the $n_R$ subregions $R_1, \dots, R_{n_R}$ of the image, see Fig.~\ref{fig_cube_example_vertices_classification}, with values $f_1, \dots, f_{n_R}$, i.e., $f(x) = \sum_{j=1}^{n_R} \chi_{R_j}(x) f_j$ with $\chi_S(x)$ the characteristic function of a set $S\subset D$. We collect the intensity values $f_i$ in a vector $\underline f = (f_1, \dots, f_{n_R})^T \in \R^{n_R}$. All numerical experiments are conducted using the finite element software package NGSolve \cite{Sch14}. For reproducibility reasons, the used code is available at \cite{GanMejSch24a}.

%%%%%%%%%%%%%%%%%%%%%%%%%%%%%%%%
\subsubsection{Cube: 'L-corner', 'Fork' and 'Arrow' detection}
\label{corner_detection}
The first example is a 3D unit cube in a 2D image with three sides of different colors and a background, i.e., $n_R=4$, see Fig.~\ref{fig:Cube2D}. The three intensity values $f_{1}$, $f_{2}$, $f_{3}$, are distributed for the cube's side regions $R_1$, $R_2$, $R_3$ and $f_{4}=0$ for the image background region $R_4$. We consider this example because the cube's corners cover three classes of vertices: vertices A, C, and F are classified as 'L-corners', vertices B, D, and G as 'Arrows', and vertex E as 'Fork', see Fig.~\ref{fig_cube_example_vertices_classification}(a). 

%%%%%%%%%%%%%%%%%%%%%%%%%%%%%%%%%
\paragraph{Localization of given inclusion shapes.}
We first present results for finding the location of a given inclusion shape as discussed in Section~\ref{one_shot_detection_process}.
\ifpics
\begin{figure}%[h]
\centering
\begin{tikzpicture}
    \node[anchor=south west,inner sep=0] at (-6.5,-0.5) {\includegraphics[scale=.35]{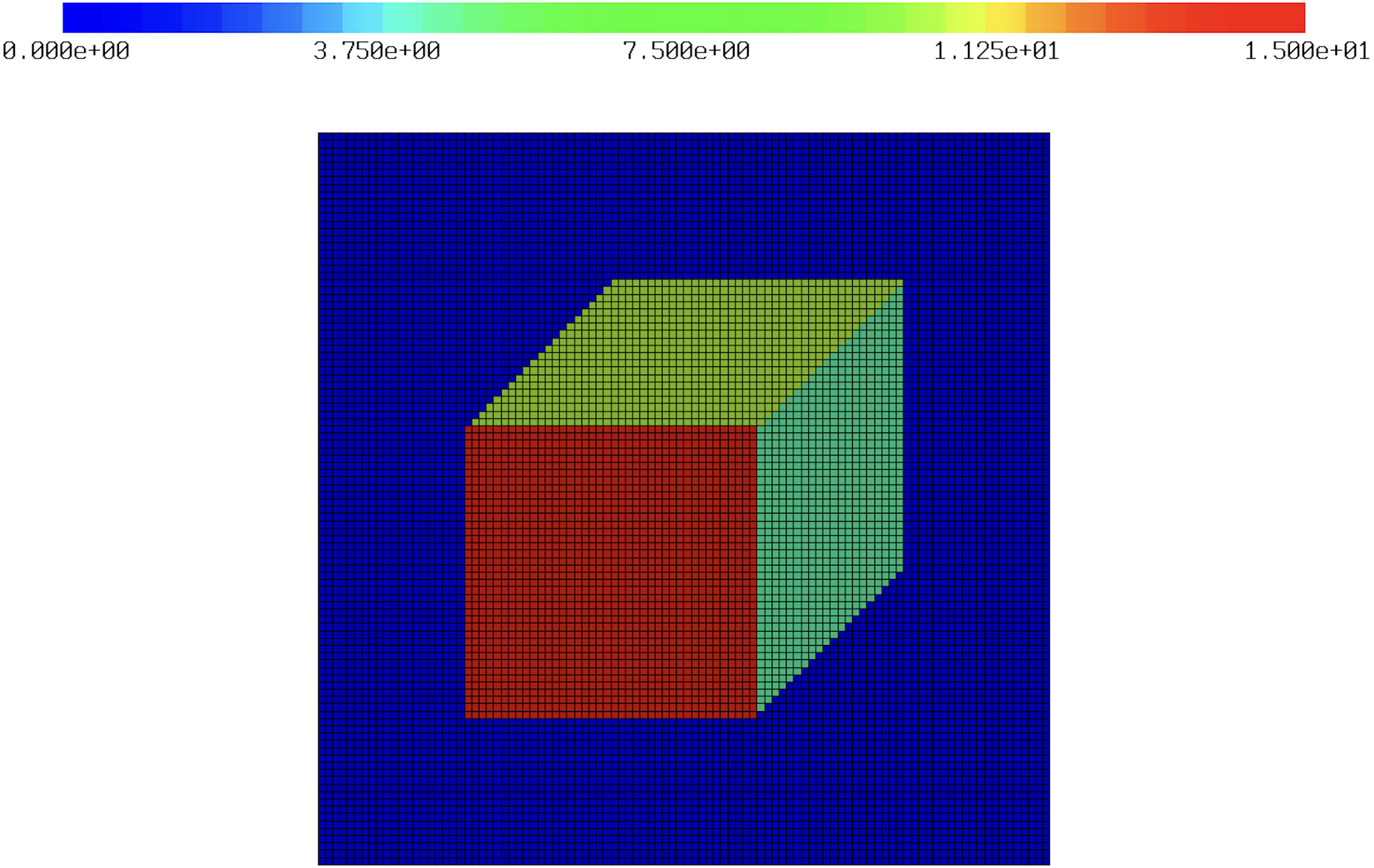}};
    \node[text=gray(x11gray)] at (-2.75,.75) {A};
    \node[text=gray(x11gray)] at (-2.75,3.25) {B};
    \node[text=gray(x11gray)] at (-1.5,4.5) {C};
    \node[text=gray(x11gray)] at (1.25,4.5) {D};
    \node[text=gray(x11gray)] at (0,3.25) {E};
    \node[text=gray(x11gray)] at (0,.75) {G};
    \node[text=gray(x11gray)] at (1.25,2) {F};
    \end{tikzpicture}
\caption{Cube 2D.}
\label{fig:Cube2D}
\end{figure}
\fi
Here we choose $f_{R_1}=15$, $f_{R_2}=10$, $f_{R_3}=5$, i.e., $\underline f= \underline f^{(1)} = (15, 10, 5, 0)^T$ as in Fig.~\ref{fig:Cube2D}. Figure~\ref{fig:Cube2D:vertsAB}(a) shows the map of the second-order topological derivative for the right-angled inclusion shape with angles deg(P)=$0^\circ$, deg(Q)=$90^\circ$, as indicated in the figure. The second-order topological derivative attains values close to its global minimum only in the vicinity of the vertex A. Thus, the indicator function gives the exact location of the vertex, represented by the topological derivative map for the given inclusion shape. The same holds true also for vertex B by means of an inclusion shape with three lines and deg(P)=$0^\circ$, deg(Q)=$45^\circ$, deg(R)=$270^\circ$, see Fig.~\ref{fig:Cube2D:vertsAB}(b), and vertex G, see Fig.~\ref{fig:Cube2D:vertsGE}(a) where deg(P)=$45^\circ$, deg(Q)=$90^\circ$, deg(R)=$180^\circ$. When attempting to identify vertex E of class 'Fork' in Fig.~\ref{fig:Cube2D:vertsGE}(b), we can see that, while the desired vertex is visible in the second-order topological derivative, it is dominated by two other vertices (B and G). Here, we observe that the contrast between intensities of subregions adjacent to the vertices plays a crucial role. Vertices whose adjacent subregions have high contrast in intensity are more likely to be detected, even if the considered inclusion shape is not or only partially suitable. For vertices C and D, Fig.~\ref{fig:Cube2D:vertsCD:f1} shows similar behavior with sometimes even full edges dominating the sought-for vertex due to the high contrast between the red and blue material in Fig.~\ref{fig:Cube2D}. 
\ifpics
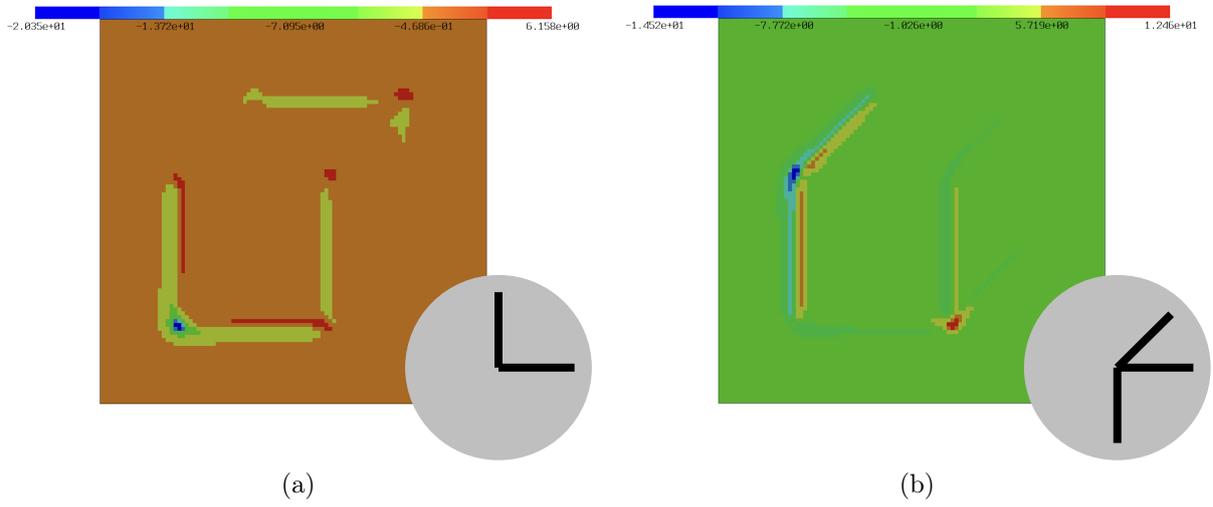
\begin{figure}
\begin{tabular}{cc}
    \begin{tikzpicture}
    \node[anchor=south west,inner sep=0] at (-6.5,-0.5) {\includegraphics[scale=.23]{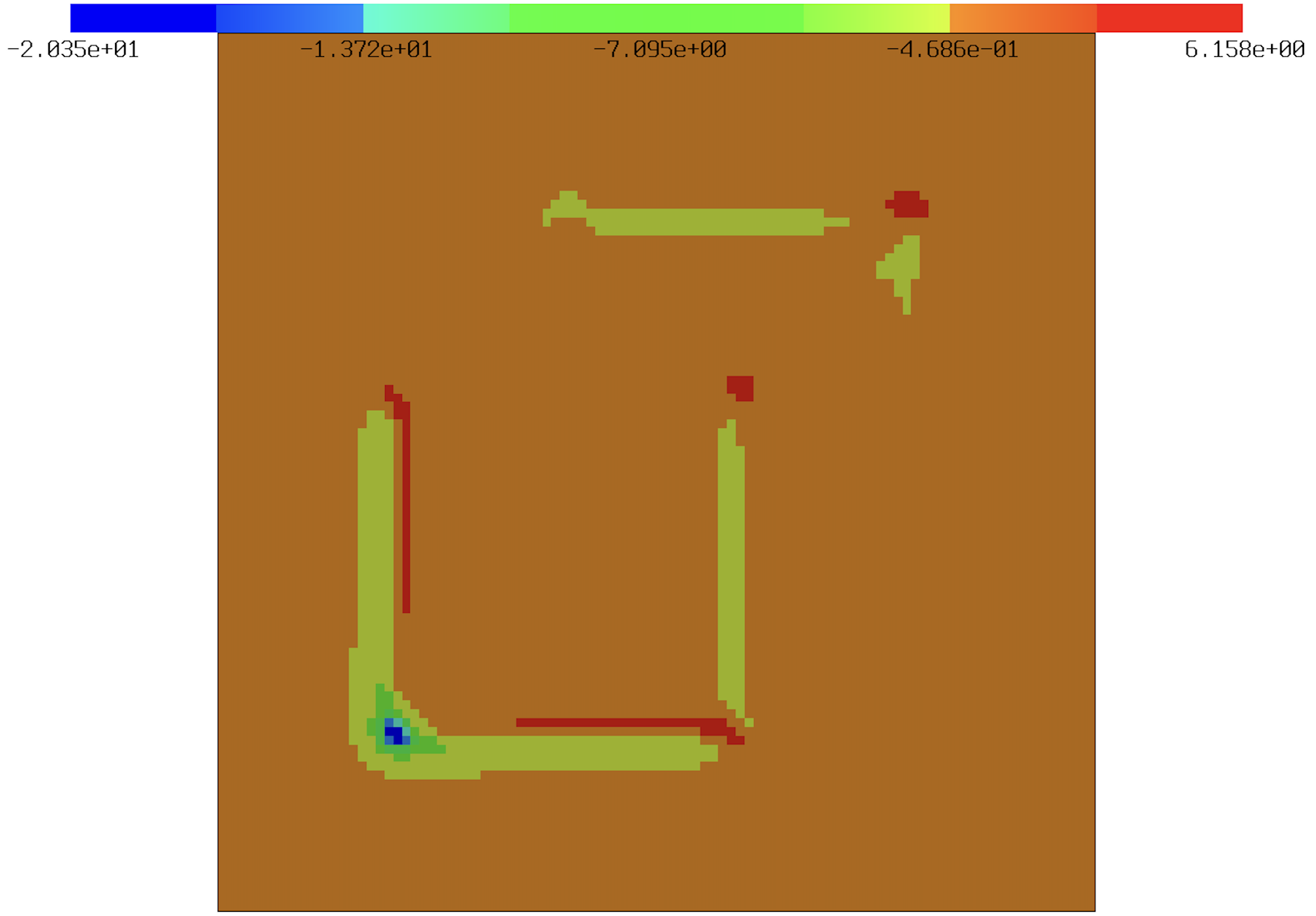}};
    \draw[gray(x11gray), ultra thick, fill] (0.,0.) circle (1.2);
    \draw[line width=3pt, black] (0,0) -- (1,0) (0,0) -- (0,1);
    \end{tikzpicture}
    &
    \begin{tikzpicture}
    \node[anchor=south west,inner sep=0] at (-6.5,-0.5) {\includegraphics[scale=.23]{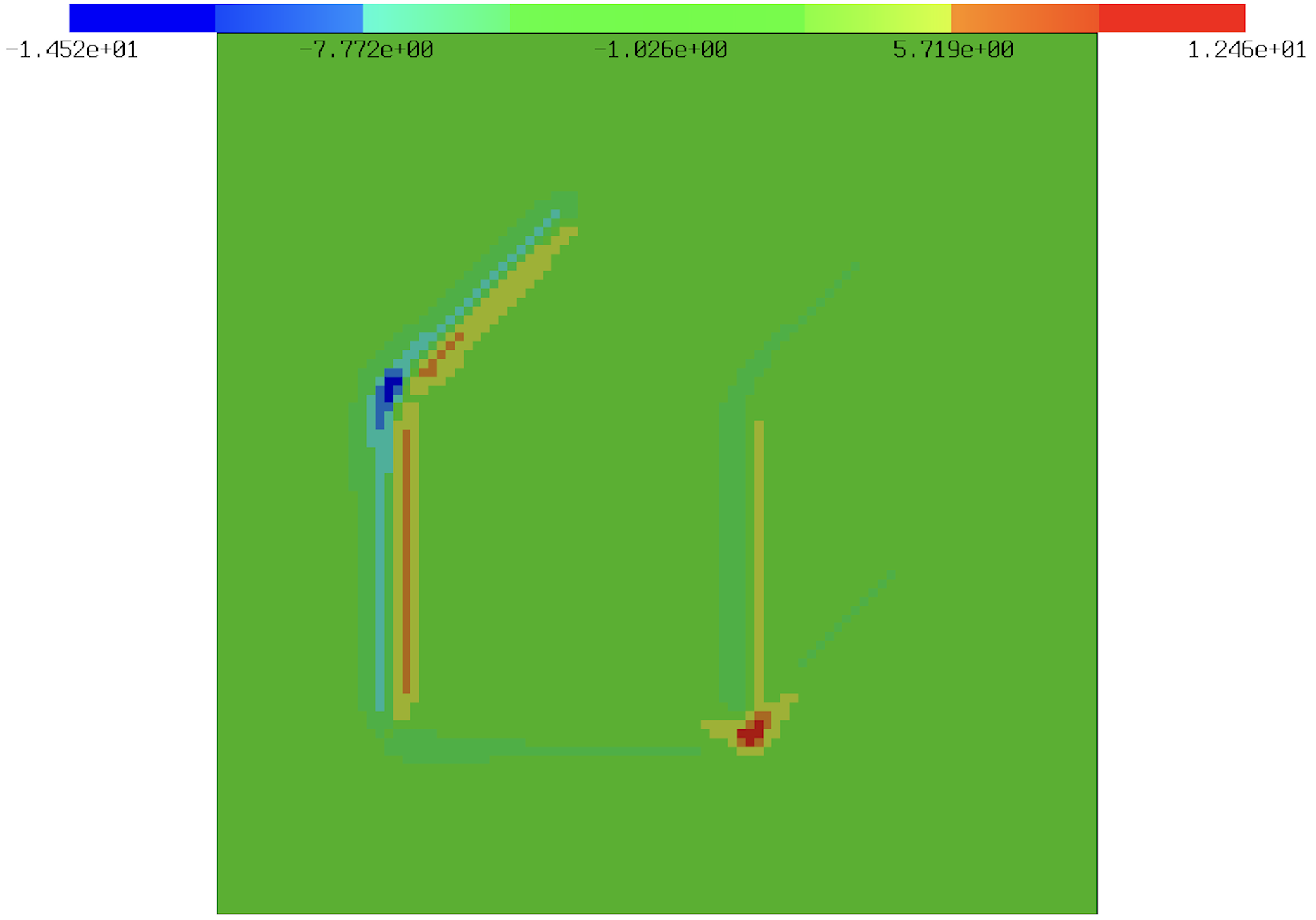}};
    \draw[gray(x11gray), ultra thick, fill] (0.,0.) circle (1.2);
    \draw[line width=3pt, black] (0,0) -- (1,0) (0,0) -- (1/1.41421356237,1/1.41421356237) (0,0) -- (0,-1); 
    \end{tikzpicture} \\
    (a) & (b)
\end{tabular}
\caption{Second-order topological derivative $d^2 \mathcal J(\Omega)(\cdot, \omega)$ for $\underline f= \underline f^{(1)} = (15, 10, 5, 0)^T$ and indicated inclusion shapes $\omega$: (a) deg(P)=$0^\circ$, deg(Q)=$90^\circ$ to identify vertex A. (b) deg(P)=$0^\circ$, deg(Q)=$45^\circ$, deg(R)=$270^\circ$ to identify vertex B.}
\label{fig:Cube2D:vertsAB}
\end{figure}
\begin{figure}
\begin{tabular}{cc}%bbbbbs
    \begin{tikzpicture}
    \node[anchor=south west,inner sep=0] at (-6.5,-0.5) {\includegraphics[scale=.23]{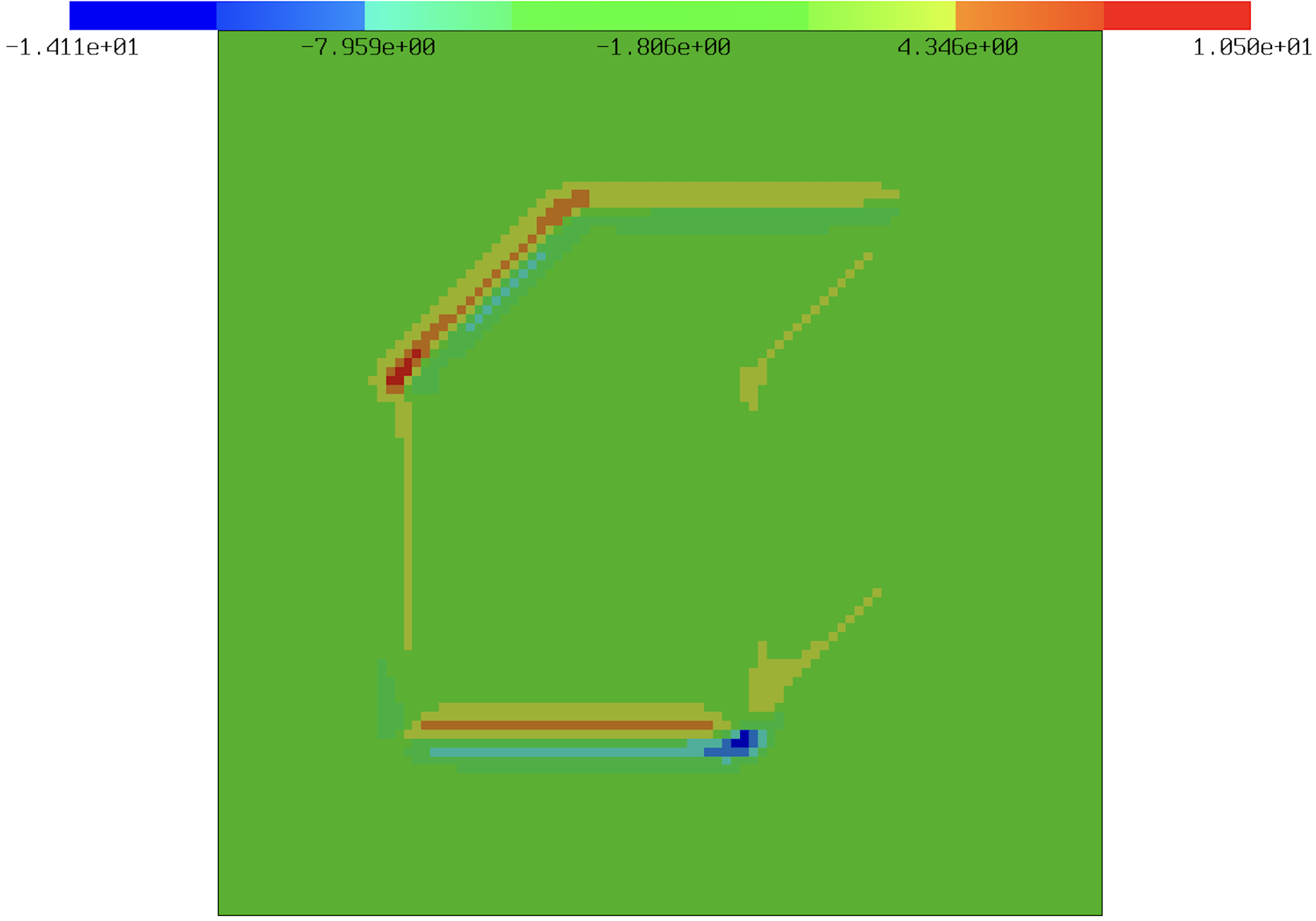}};
    \draw[gray(x11gray), ultra thick, fill] (0.,0.) circle (1.2);
    \draw[line width=3pt, black] (0,0) -- (1/1.41421356237,1/1.41421356237) (0,0) -- (0,1) (0,0) -- (-1,0) ; 
    \end{tikzpicture}
    &
    \begin{tikzpicture}
    \node[anchor=south west,inner sep=0] at (-6.5,-0.5) {\includegraphics[scale=.23]{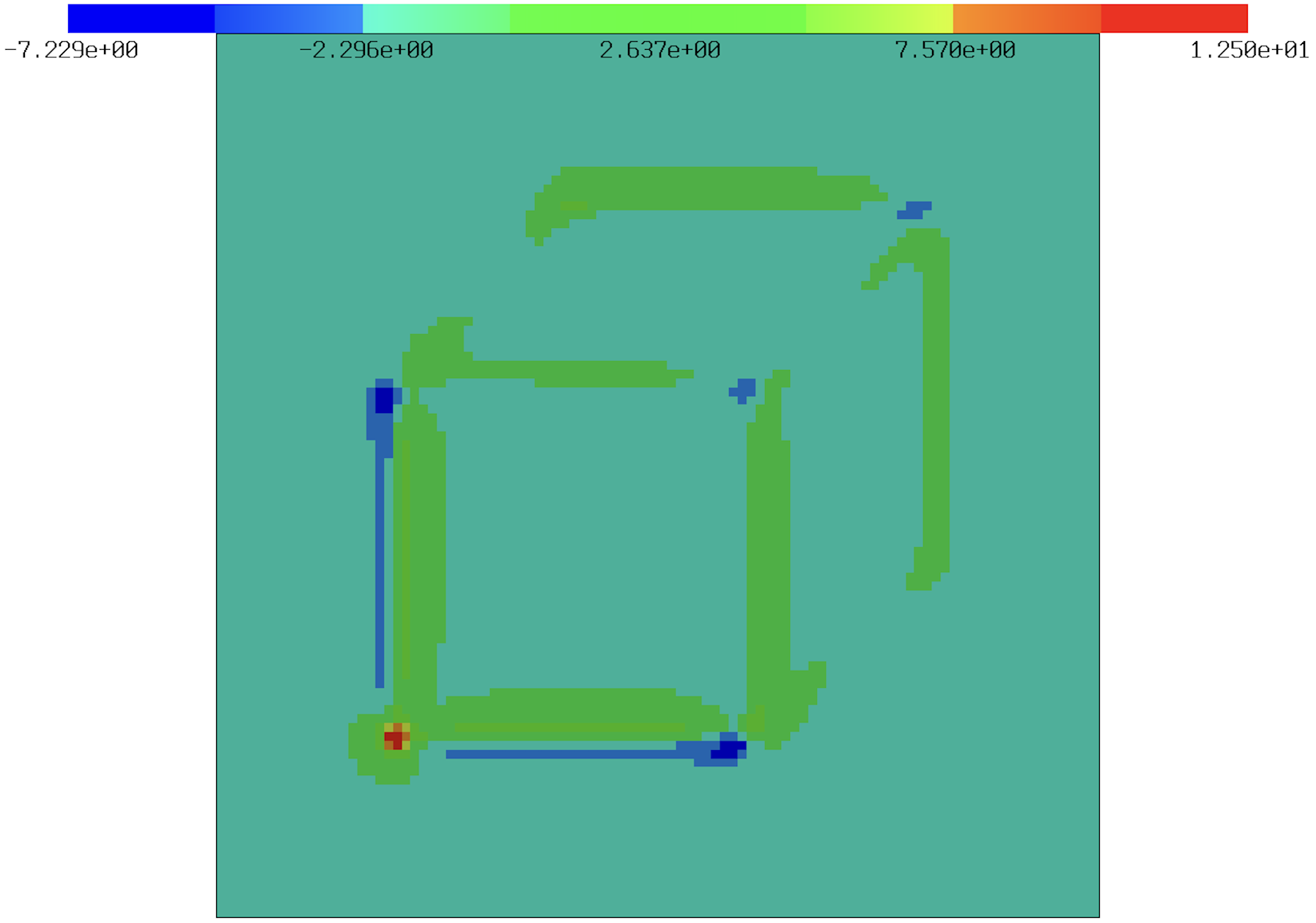}};
    \draw[gray(x11gray), ultra thick, fill] (0.,0.) circle (1.2);
    \draw[line width=3pt, black] (0,0) -- (1/1.41421356237,1/1.41421356237) (0,0) -- (-1,0) (0,0) -- (0,-1); 
    \end{tikzpicture} \\
    (a) & (b)
\end{tabular}
\caption{Second-order topological derivative $d^2 \mathcal J(\Omega)(\cdot, \omega)$ for $\underline f= \underline f^{(1)} = (15, 10, 5, 0)^T$ and indicated inclusion shapes $\omega$: (a) deg(P)=$45^\circ$, deg(Q)=$90^\circ$, deg(R)=$180^\circ$ to identify vertex G. (b) deg(P)=$45^\circ$, deg(Q)=$180^\circ$, deg(R)=$270^\circ$ attempting to identify vertex E.}
\label{fig:Cube2D:vertsGE}
\end{figure}
\begin{figure}
\begin{tabular}{cc}%cccccs
    \begin{tikzpicture}
    \node[anchor=south west,inner sep=0] at (-6.5,-0.5) {\includegraphics[scale=.23]{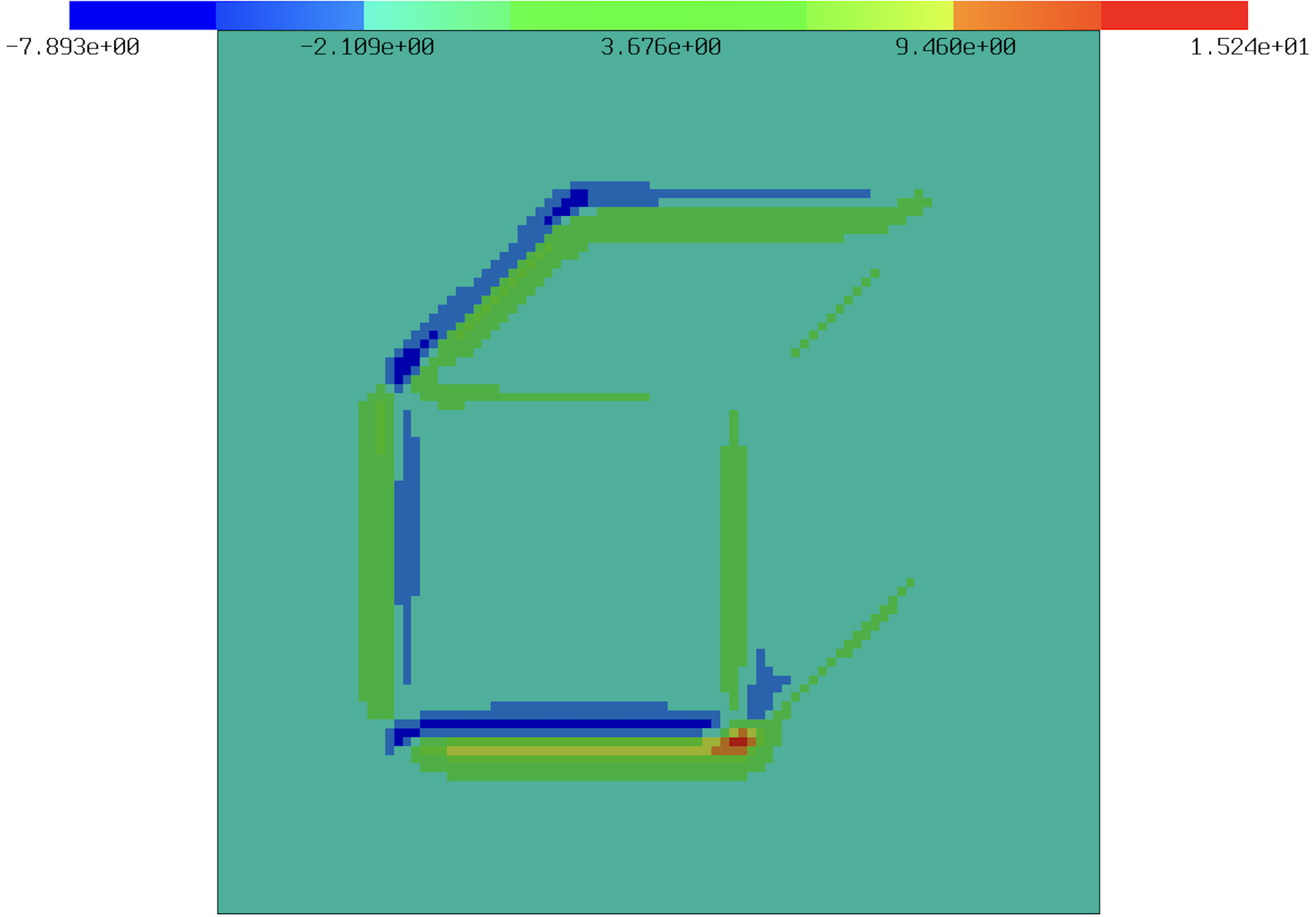}};
    \draw[gray(x11gray), ultra thick, fill] (0.,0.) circle (1.2);
    \draw[line width=3pt, black] (0,0) -- (1,0) (0,0) -- (-1/1.41421356237,-1/1.41421356237); 
    \end{tikzpicture}
    &
    \begin{tikzpicture}
    \node[anchor=south west,inner sep=0] at (-6.5,-0.5) {\includegraphics[scale=.23]{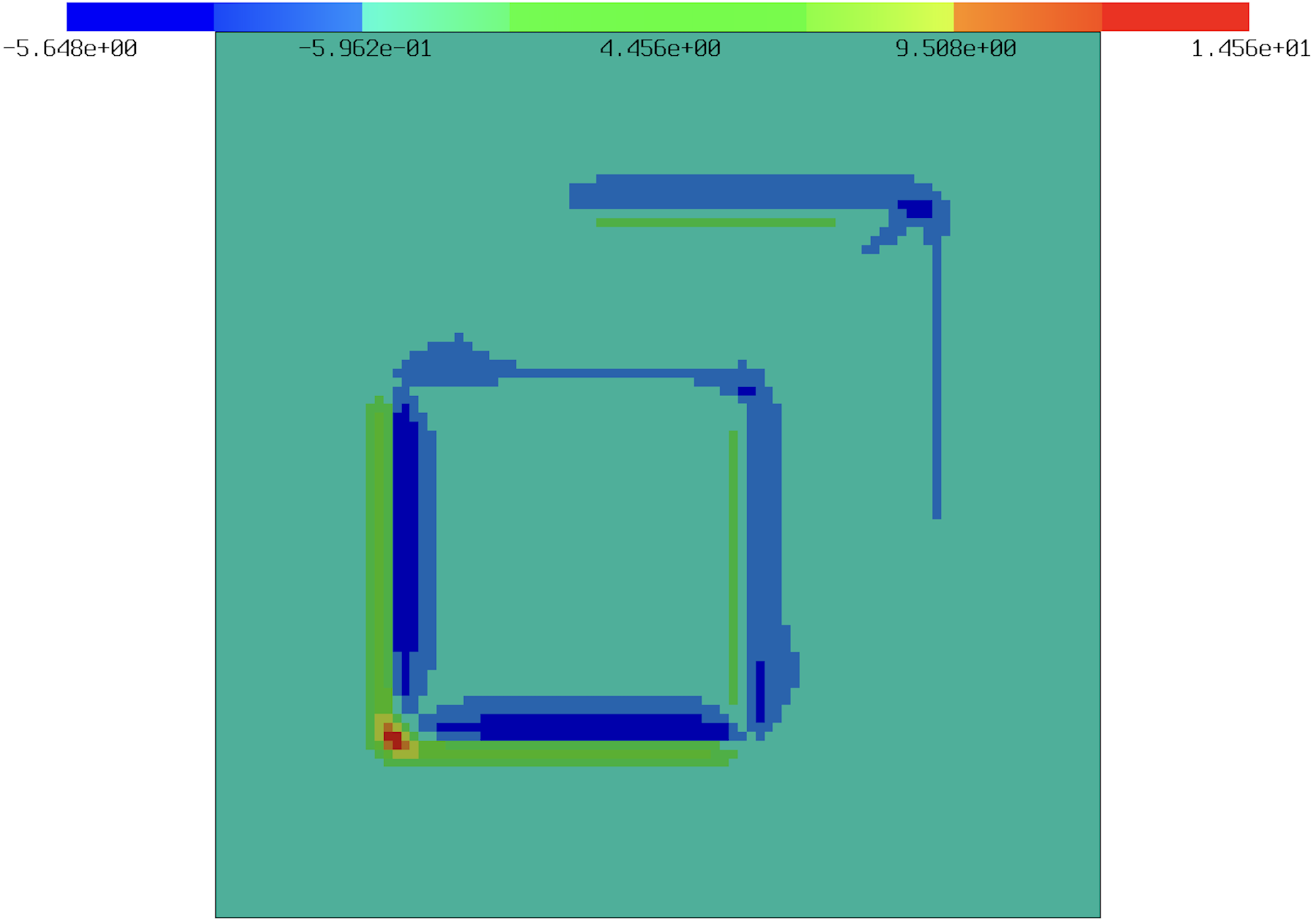}};
    \draw[gray(x11gray), ultra thick, fill] (0.,0.) circle (1.2);
    \draw[line width=3pt, black] (0,0) -- (-1,0) (0,0) -- (-1/1.41421356237,-1/1.41421356237) (0,0) -- (0,-1); 
    \end{tikzpicture} \\
    (a) & (b)
\end{tabular}
\caption{Second-order topological derivative $d^2 \mathcal J(\Omega)(\cdot, \omega)$ for $\underline f= \underline f^{(1)} = (15, 10, 5, 0)^T$ and indicated inclusion shapes $\omega$: (a) deg(P)=$0^\circ$, deg(Q)=$225^\circ$, attempting to identify vertex C. (b) deg(P)=$180^\circ$, deg(Q)=$225^\circ$, deg(R)=$270^\circ$ attempting to identify vertex D.}
\label{fig:Cube2D:vertsCD:f1}
\end{figure}
\fi
Thus, we consider another scenario of regions' color distribution, $\underline f = \underline f^{(2)} =(10, 15, 5, 0)^T$, i.e., their intensity values are $f_1=10$, $f_2=15$, $f_3=5$, $f_4=0$. Figure~\ref{fig:Cube2D:vertsCD:f2}(a) and (b) depict the second-order topological derivative for the inclusion shapes with angles deg(P)=$0^\circ$, deg(Q)=$225^\circ$ and with angles deg(P)=$180^\circ$, deg(Q)=$225^\circ$, deg(R)=$270^\circ$, representing vertex C and vertex D, respectively. The corners' detection is very accurate, where the global minimum of $d^2\mathcal{J}(\Omega)(z,\omega)$ coincides exactly with the vertices' locations. 
\ifpics
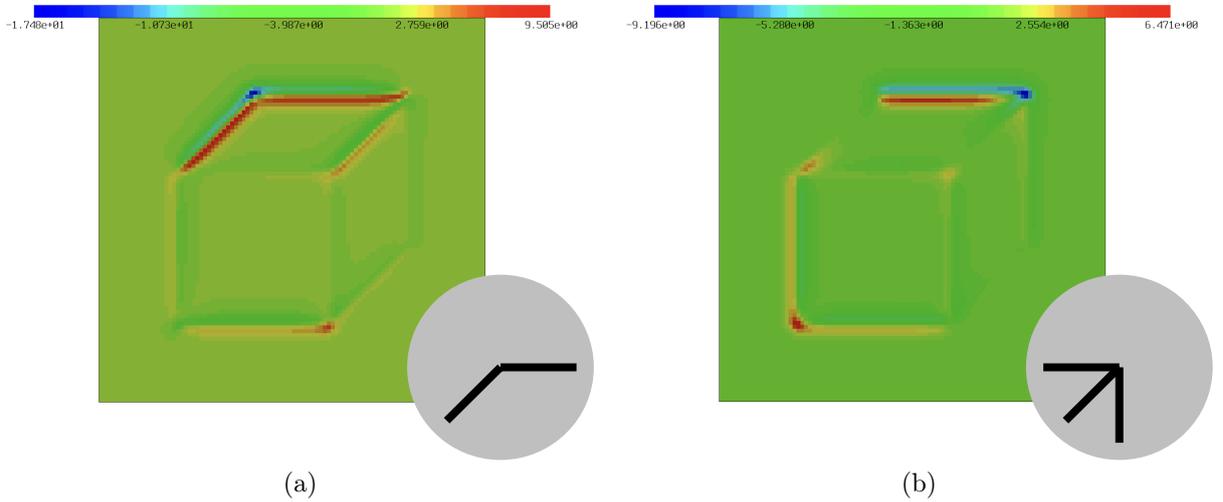
\begin{figure}
\begin{tabular}{cc}%cccccs
    \begin{tikzpicture}
    \node[anchor=south west,inner sep=0] at (-6.5,-0.5) {\includegraphics[scale=.23]{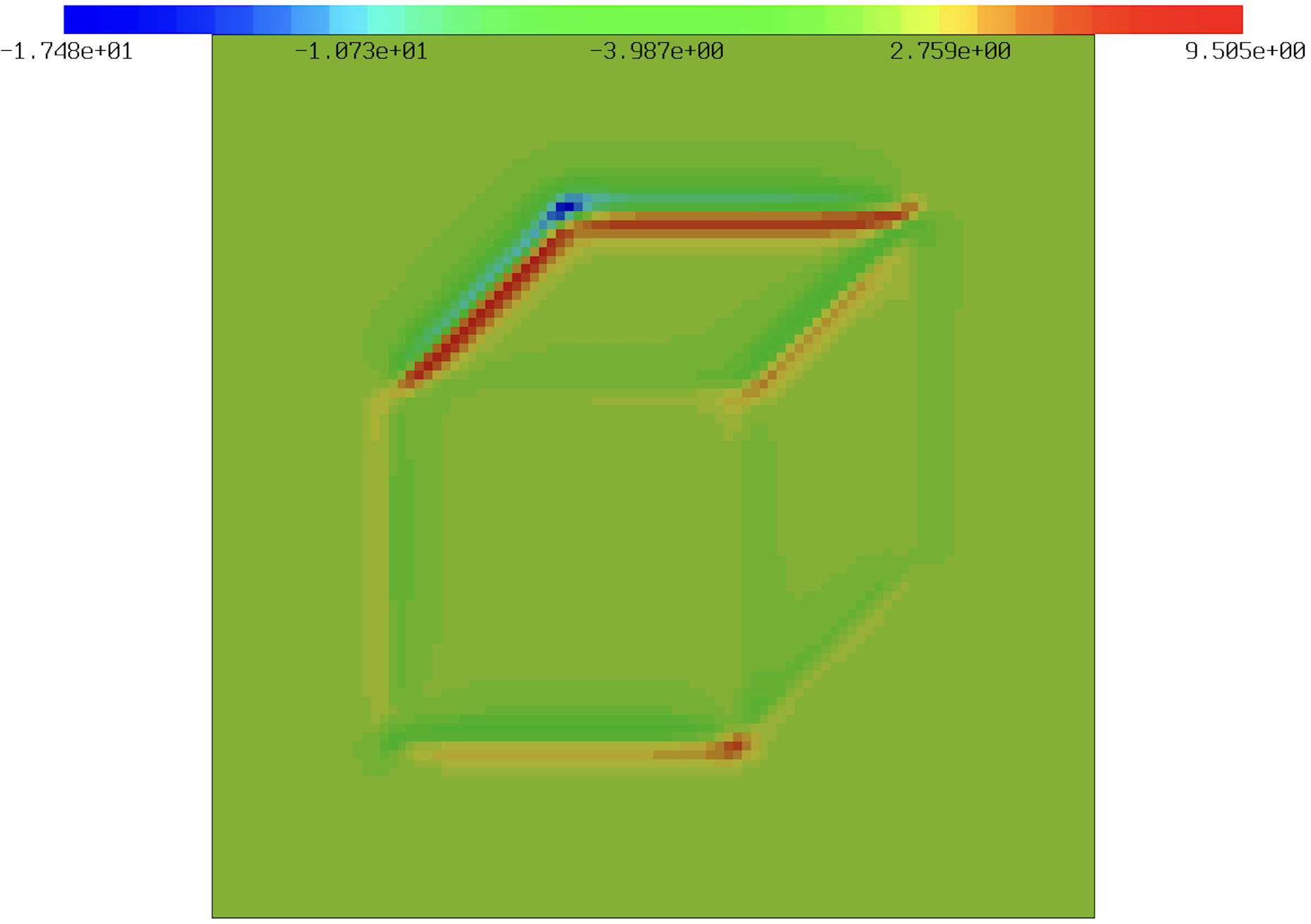}};
    \draw[gray(x11gray), ultra thick, fill] (0.,0.) circle (1.2);
    \draw[line width=3pt, black] (0,0) -- (1,0) (0,0) -- (-1/1.41421356237,-1/1.41421356237); 
    \end{tikzpicture}
    &
    \begin{tikzpicture}
    \node[anchor=south west,inner sep=0] at (-6.5,-0.5) {\includegraphics[scale=.23]{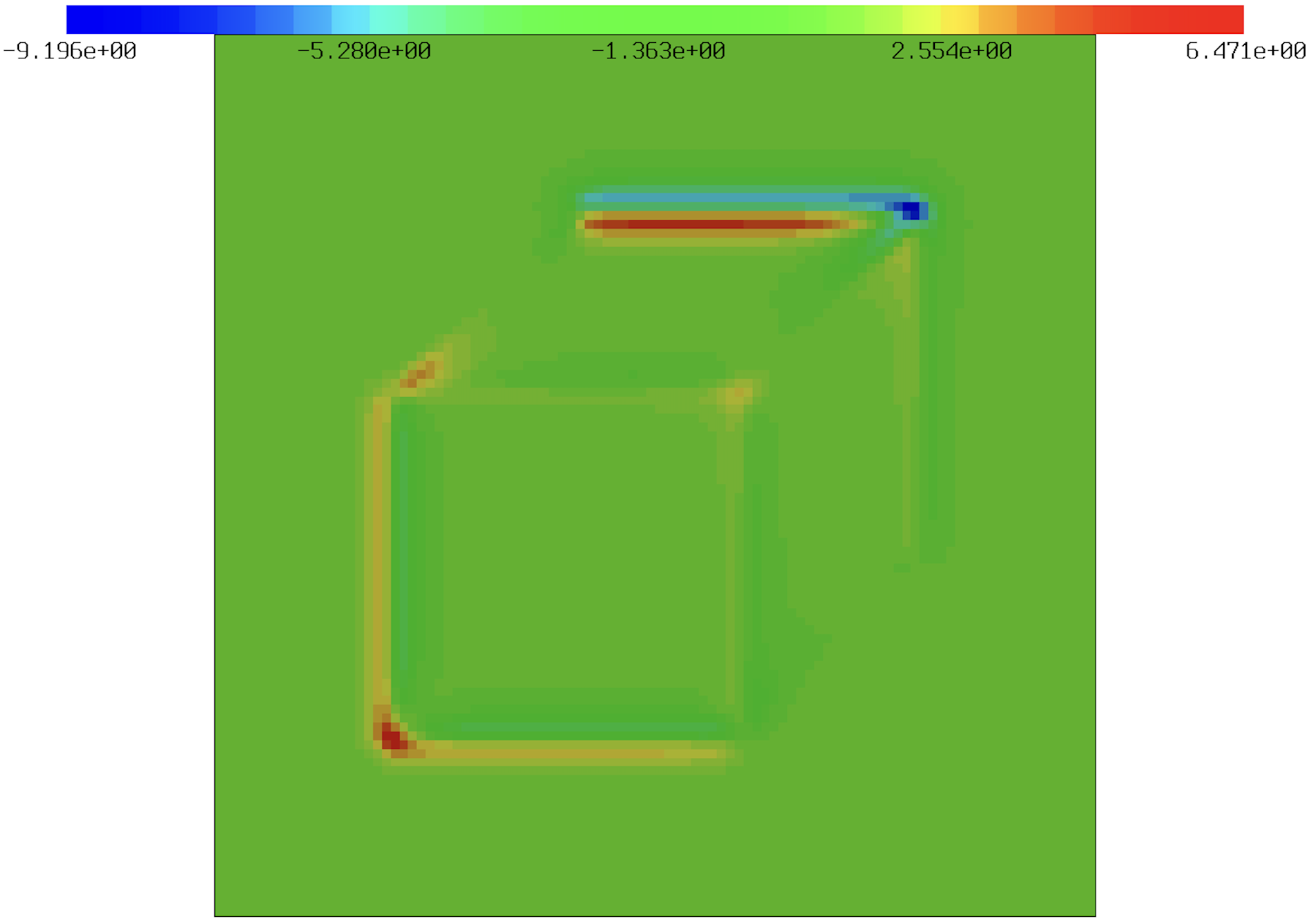}};
    \draw[gray(x11gray), ultra thick, fill] (0.,0.) circle (1.2);
    \draw[line width=3pt, black] (0,0) -- (-1,0) (0,0) -- (-1/1.41421356237,-1/1.41421356237) (0,0) -- (0,-1); 
    \end{tikzpicture} \\
    (a) & (b)
\end{tabular}
\caption{Second-order topological derivative $d^2 \mathcal J(\Omega)(\cdot, \omega)$ for $\underline f= \underline f^{(2)} = (10, 15, 5, 0)^T$ and indicated inclusion shapes $\omega$: (a) deg(P)=$0^\circ$, deg(Q)=$225^\circ$, to identify vertex C. (b) deg(P)=$180^\circ$, deg(Q)=$225^\circ$, deg(R)=$270^\circ$ to identify vertex D.}
\label{fig:Cube2D:vertsCD:f2}
\end{figure}
\fi

Finally, we deduce that the detection algorithm is effective even for different classes of vertices as depicted in Figures~\ref{fig:Cube2D:vertsAB},~\ref{fig:Cube2D:vertsGE} and~\ref{fig:Cube2D:vertsCD:f2}. The approach depends on the contrast of the regions around the vertex to be detected. It highlights the dominating corners or edges, however, it does not detect {\it{fake points}} in the scene.

%%%%%%%%%%%%%%%%%%%%%%%%%
\paragraph{Simultaneous determination of location and inclusion shape.}
Next, we illustrate the procedure described in Section~\ref{sec_detectUnknown}, i.e., we compute the second-order topological derivative $d^2\mathcal{J}(\cdot,\omega)$ for all $\omega$ in a set of inclusion shapes $\Theta$ and sort the inclusion shapes by the minimal values of their corresponding second-order topological derivatives. Here, we choose $\Theta:=\Theta^{(2)} \cup \Theta^{(3)}$ with $m=8$. Thus, the inclusion shapes $\omega^{(i)} \in \Theta$ represent the vertex classes of Fig.~\ref{fig_classification_vertices}(a)-(d) with different lines (two lines for 'L-corner', three lines for 'Fork', 'Arrow', 'T-junction') and with different angles varying in $[0^\circ,360^\circ]$ with precision $45^\circ$. We again consider the scene where $\underline f = \underline f^{(1)}$ as depicted in Fig.~\ref{fig:Cube2D}.
Table~\ref{table:1} presents the first thirteen inclusion shapes of $\Theta$ when ranking them according to the minimal value of the second-order topological derivative over the computational domain, i.e., $\underset{z\in D}{\mbox{min }} d^2\mathcal J(\Omega)(z, \omega^{(i)})$. We can see that the most negative value is attained for the inclusion shape $\omega^{(i)}=\omega[0^\circ, 90^\circ]$ with two lines and angles deg(P)=$0^\circ$, deg(Q)=$90^\circ$. From Fig.~\ref{fig:Cube2D:vertsAB}(a) we see that this minimal value is located exactly at the vertex A. The fact that this vertex is the first in the list is not surprising since, for the scene $\underline f^{(1)}$, vertex A is adjacent to the edges with the highest contrasts, see Fig.~\ref{fig:Cube2D}. After this, we see some only partially correct reconstructions of vertices G and B indicating corners instead of triple-junctions, see lines 2-5, as well as vertex A which has already been detected and could thus be ignored in lines 7-9. Actually, a closer look at lines 2 and 5 shows that the angles suggested by the method are present in the correct inclusion shape at vertex G (and the correct inclusion shape is actually obtained by superposition of the inclusion shapes of lines 2 and 5). The same is the case for vertex B and lines 3 and 4.
Moreover, the edges (AB) and (AG) are dominating edges in the topological derivative map (e.g. see Fig.~\ref{fig:Cube2D:vertsGE}(b) and Fig.~\ref{fig:Cube2D:vertsCD:f1}), as they represent the border between the red and blue materials of highest contrast in Fig.~\ref{fig:Cube2D}, see lines 6, 10, 12 in Table~\ref{table:1}. The full information on the vertices B and G is found in lines 11 and 13 of Table~\ref{table:1}. 

We again observe the well-known issue that areas of high contrast are dominating in the reconstruction which could be mitigated by choosing different scenes $\underline f$ and possibly averaging the obtained second-order topological derivative maps. Nevertheless, we believe that the information obtained from Table~\ref{table:1} is already valuable for detecting vertices without knowledge about their shape. The accuracy of our approach is underlined by the fact that only actual points or edges in the scene are detected and no {\it{fake}} points.
\begin{table}
\begin{center}
\begin{tabular}{|c|c|c|c|c|c|}
\hline 
&Value $d^2\mathcal{J}$ &  \multicolumn{3}{c|}{Inclusion $\omega^{(i)}$}  & Position \\
\hline
& & deg(P) & deg(Q) & deg(R)  & \\
\hline
\rowcolor{gray(x11gray)}
1&-20.348 & $0^\circ$ & $90^\circ$ & & A \\
\hline
2&-18.079  & $90^\circ$ & $180^\circ$ & & G \\
\hline
3&-17.629  & $0^\circ$ & $270^\circ$ & & B \\
%\hline
% -17.258  & $0^\circ$ & & & A, B \\
%\hline
% -17.113  & $90^\circ$ & & & A, G \\
%\hline
% -16.745  & $180^\circ$ & & & G \\
\hline
4&-16.327  & $45^\circ$ & $270^\circ$ & & B \\
%\hline
% -16.193  & $270^\circ$ & & & B \\
\hline
5&-15.105  & $45^\circ$ & $180^\circ$ & & G \\
\hline
6&-14.913  & $45^\circ$ & $315^\circ$ & & (AB) \\
\hline
7&-14.788  & $0^\circ$ & $45^\circ$ & $90^\circ$ & A \\
\hline
8&-14.684  & $90^\circ$ & $315^\circ$ & & A \\
\hline
9&-14.669  & $0^\circ$ & $135^\circ$ & & A \\
\hline
10&-14.545  & $45^\circ$ & $135^\circ$ & & (AG) \\
\hline
\rowcolor{gray(x11gray)}
11&-14.517  & $0^\circ$ & $45^\circ$ & $270^\circ$ & B \\
\hline
12&-14.419  & $225^\circ$ & $315^\circ$ & & (AG) \\
\hline
\rowcolor{gray(x11gray)}
13&-14.111  & $45^\circ$ & $90^\circ$ & $180^\circ$ & G \\
\hline
\end{tabular}
\end{center}
\caption{Minimal values of second-order topological derivative, $\underset{z\in D}{\text{min }} d^2\mathcal{J}(z,\omega)$, sorted in increasing order, their corresponding inclusion shapes $\omega$ (represented by their angles), and their positions for the cube example of Fig.~\ref{fig:Cube2D} and $\underline f= \underline f^{(1)} = (15, 10, 5, 0)^T$.} 
\label{table:1}
\end{table}

%%%%%%%%%%%%%%%%%%%%%%%%%%%%%%%%%
\subsubsection{Overlapping cubes: 'T-junction' detection}
\label{t_junction_detection}
A second example is a 3D object formed by 'Overlapping Cubes' in a 2D image with six colors plus the background, see Fig.~\ref{fig_cube_example_vertices_classification}(b), where every region has a different intensity. This example covers the class 'T-junction' with three different configurations (i.e. different angles). Hereafter, we are mainly interested in detecting the triple junctions indicated by $T_1$, $T_2$, $T_3$, $T_4$ in Fig.~\ref{fig_cube_example_vertices_classification}(b).

Here, we choose the intensities depicted in Fig.~\ref{fig:OverlappingCubes}, $\underline f = \underline f^{(3)} = (15,5,10,30,20,25,0)^T$. Figure~\ref{fig_occluded_T1T2_T3}(a) and (b) show the second-order topological derivative for this setting with respect to the inclusion shapes representing junctions $T_1$ and $T_2$ (deg(P)=$0^\circ$, deg(Q)=$90^\circ$, deg(R)=$180^\circ$) and junction $T_3$ (deg(P)=$0^\circ$, deg(Q)=$180^\circ$, deg(R)=$270^\circ$), respectively. For Fig.~\ref{fig_occluded_T1T2_T3}(a), the globally minimal value coincides exactly with the location of the triple junction $T_1$. The vertex $T_2$ is not detected in this case as it has a weaker intensity jump compared to the dominating vertex $T_1$. Similarly, for Fig.~\ref{fig_occluded_T1T2_T3}(b) the globally minimal value coincides exactly with the location of triple junction $T_3$. It is a dominating point as its intensity jump is strong. Figure~\ref{fig_occluded_T4} shows the second-order topological derivative with respect to an inclusion shape given by the angles 
deg(P)=$45^\circ$, deg(Q)=$225^\circ$, deg(R)=$270^\circ$, attempting to identify triple junction $T_4$. In this case, the vertex is not detected as it has a weak intensity jump compared to the other vertices in the scene.  
\ifpics
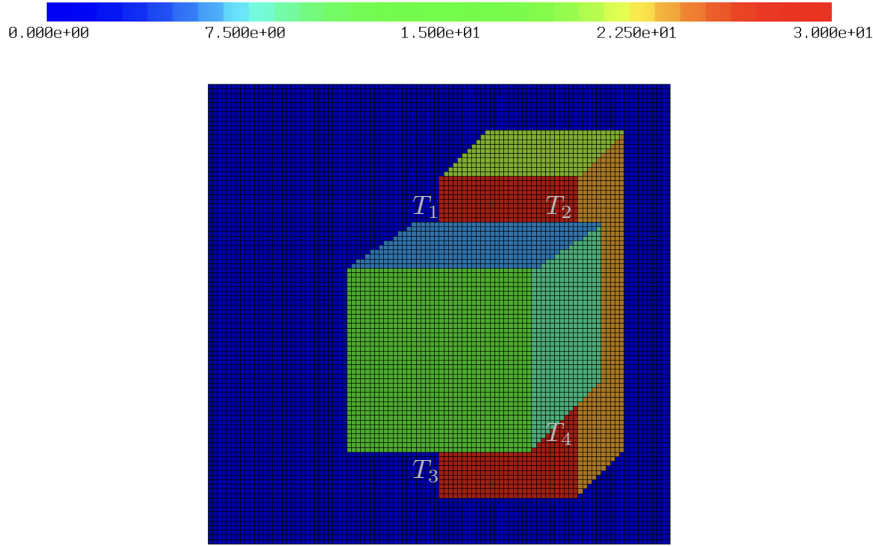
\begin{figure}%[h]
\centering
\begin{tikzpicture}
    \node[anchor=south west,inner sep=0] at (-6.5,-0.5) {\includegraphics[scale=0.35]{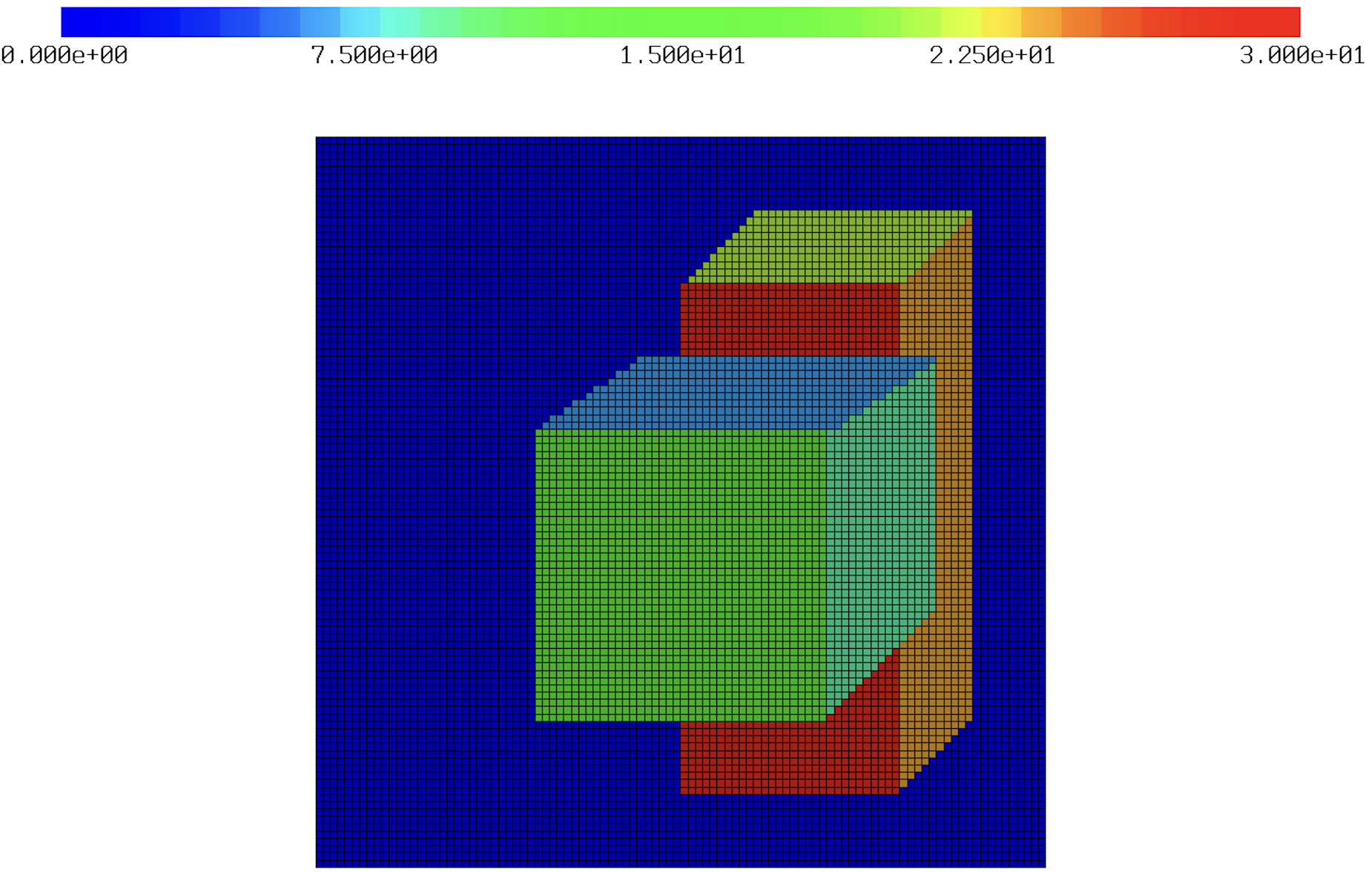}};
    \node[text=gray(x11gray)] at (-1,4) {$T_1$};
    \node[text=gray(x11gray)] at (.75,4) {$T_2$};
    \node[text=gray(x11gray)] at (-1,.5) {$T_3$};
    \node[text=gray(x11gray)] at (.75,1) {$T_4$};
\end{tikzpicture}
\caption{Overlapping Cubes.}
\label{fig:OverlappingCubes}
\end{figure}
\fi

\ifpics
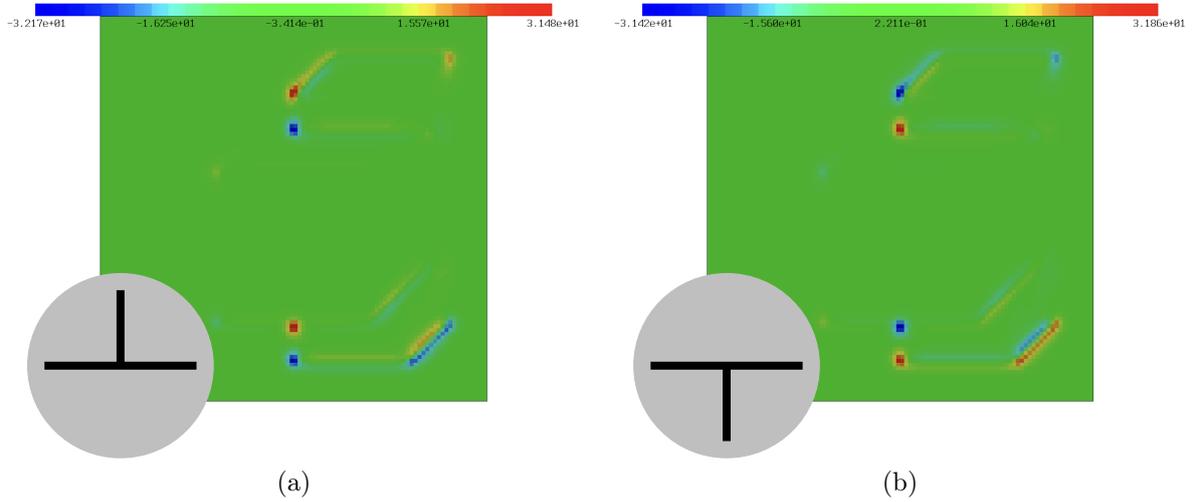
\begin{figure}
\begin{tabular}{cc}%cccccs
    \begin{tikzpicture}
    \node[anchor=south west,inner sep=0] at (-1.5,-0.5) {\includegraphics[scale=.23]{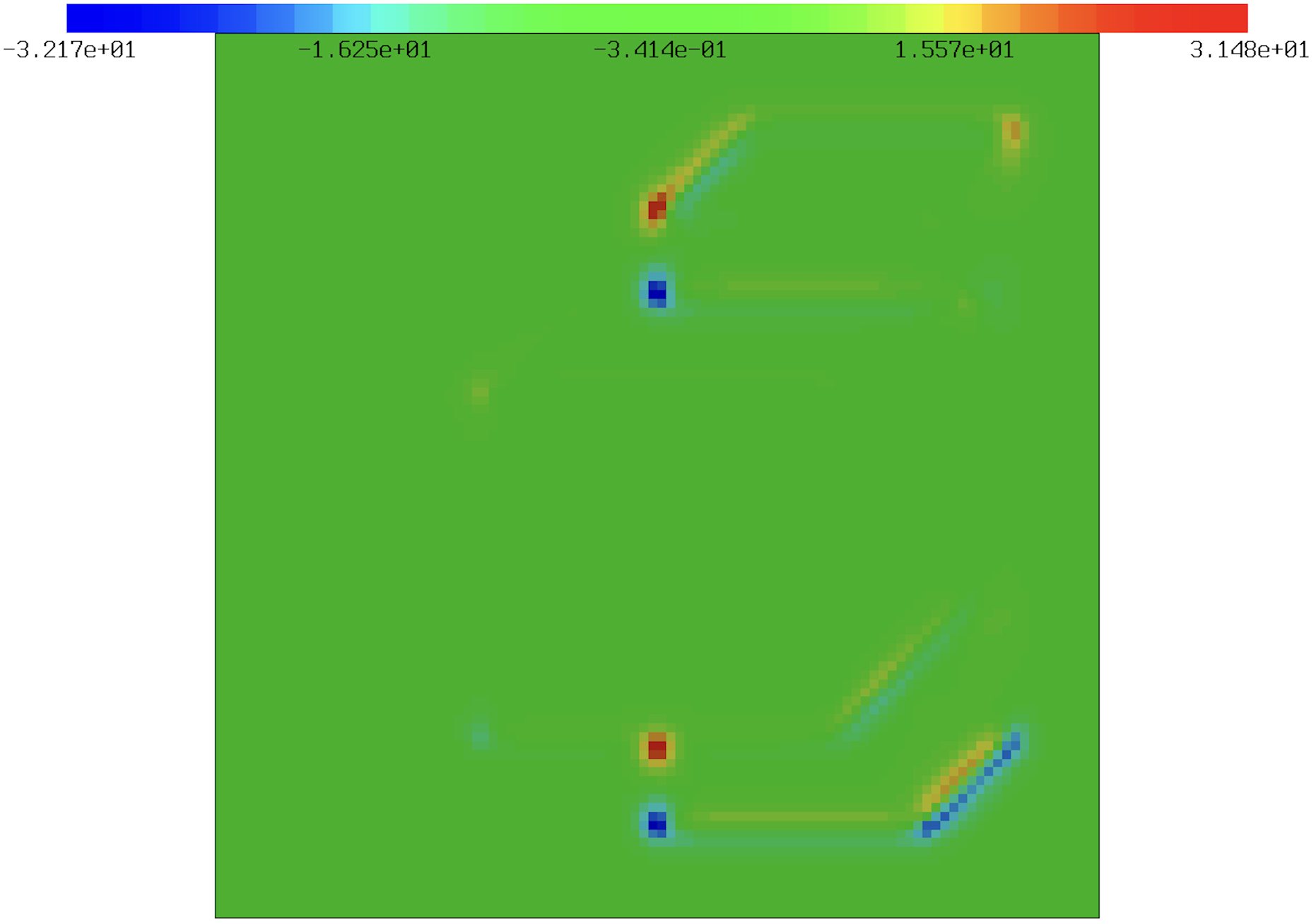}};
    \draw[gray(x11gray), ultra thick, fill] (0.,0.) circle (1.2);
    \draw[line width=3pt, black] (0,0) -- (1,0) (0,0) -- (0,1) (0,0) -- (-1,0); 
    \end{tikzpicture}
    &
    \begin{tikzpicture}
    \node[anchor=south west,inner sep=0] at (-1.5,-0.5) {\includegraphics[scale=.23]{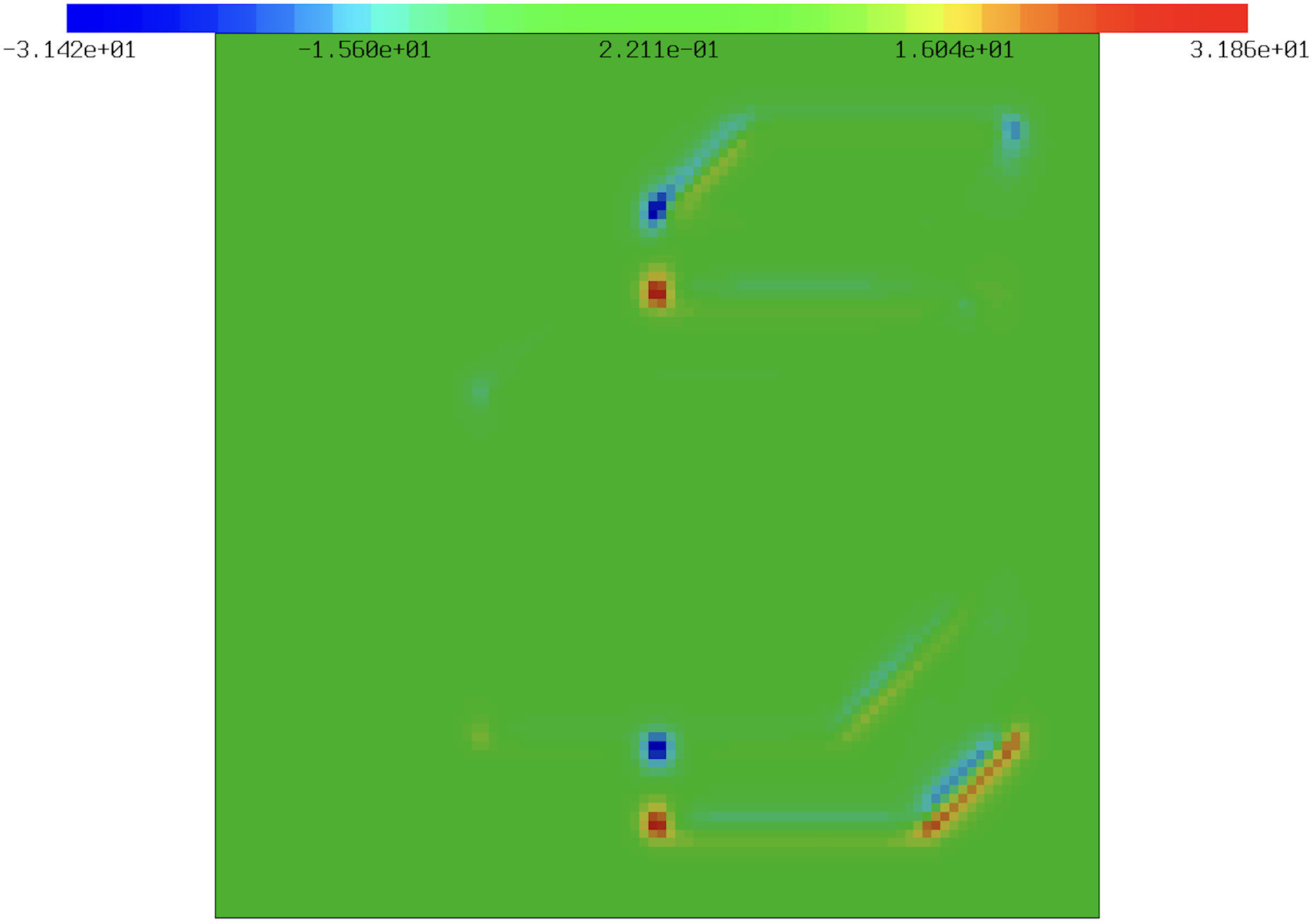}};
    \draw[gray(x11gray), ultra thick, fill] (0.,0.) circle (1.2);
    \draw[line width=3pt, black] (0,0) -- (1,0) (0,0) -- (-1,0) (0,0) -- (0,-1); 
    \end{tikzpicture} \\
    (a) & (b)
\end{tabular}
\caption{Second-order topological derivative $d^2 \mathcal J(\Omega)(\cdot, \omega)$ for $\underline f= \underline f^{(3)} = (15,5,10,30,20,25,0)^T$ and indicated inclusion shapes $\omega$: (a) deg(P)=$0^\circ$, deg(Q)=$90^\circ$, deg(R)=$180^\circ$, attempting to identify triple junctions $T_1$, $T_2$. (b) deg(P)=$0^\circ$, deg(Q)=$180^\circ$, deg(R)=$270^\circ$ to identify triple junction $T_3$.}
\label{fig_occluded_T1T2_T3}
\end{figure}
\fi

\ifpics
\begin{figure}
\centering
\begin{tabular}{c}%cccccs
    \begin{tikzpicture}
    \node[anchor=south west,inner sep=0] at (-1.5,-0.5) {\includegraphics[scale=.23]{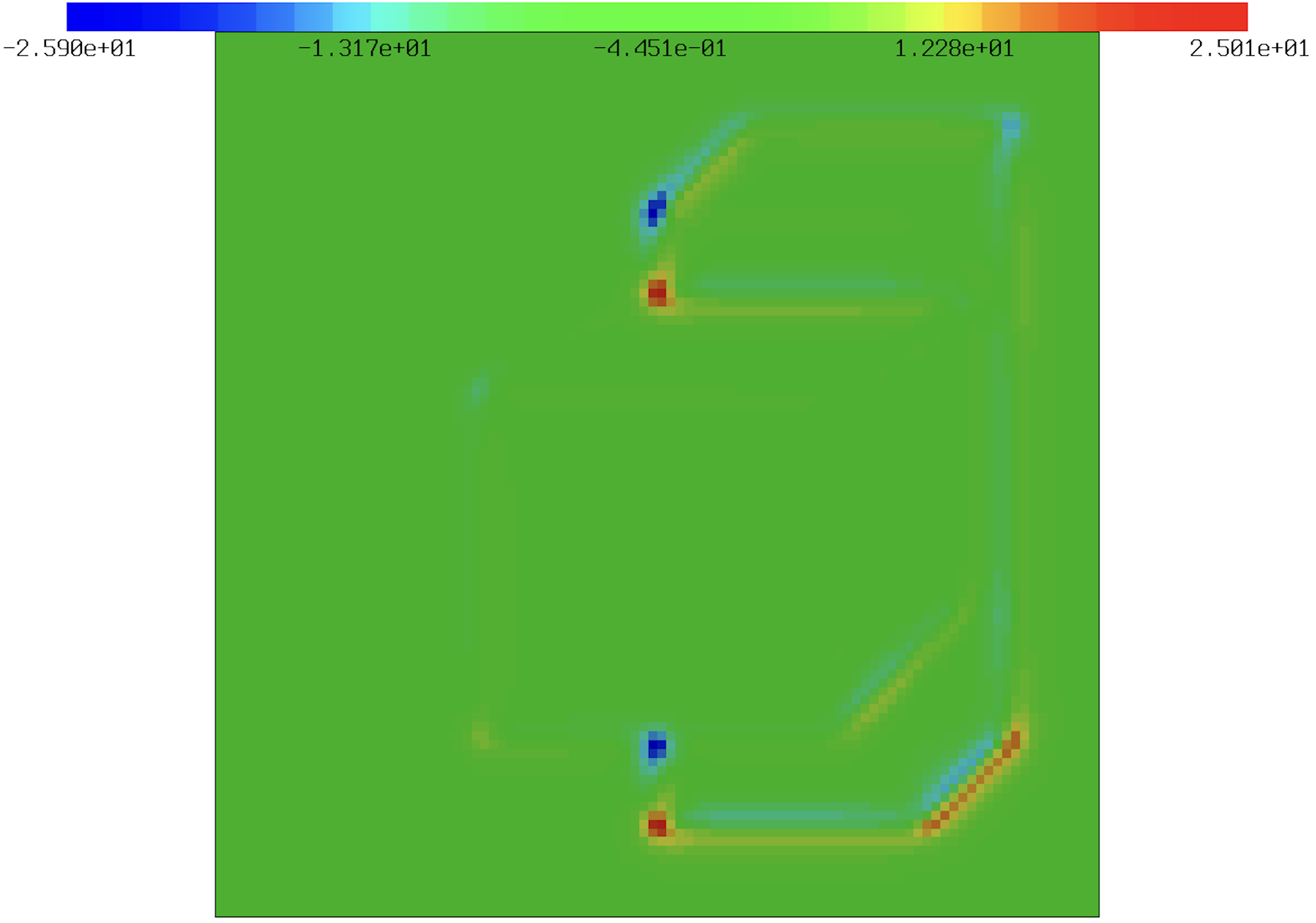}};
    \draw[gray(x11gray), ultra thick, fill] (0.,0.) circle (1.2);
    \draw[line width=3pt, black] (0,0) -- (1/1.41421356237,1/1.41421356237) (0,0) -- (-1/1.41421356237,-1/1.41421356237) (0,0) -- (0,-1); 
    \end{tikzpicture}
\end{tabular}
\caption{Second-order topological derivative $d^2 \mathcal J(\Omega)(\cdot, \omega)$ for $\underline f= \underline f^{(3)} = (15,5,10,30,20,25,0)^T$ and indicated inclusion shape $\omega$ with deg(P)=$45^\circ$, deg(Q)=$225^\circ$, deg(R)=$270^\circ$ attempting to identify triple junction $T_4$.}
\label{fig_occluded_T4}
\end{figure}
\fi

We attempt to also identify triple junctions $T_2$ and $T_4$ by modifying the subregions' intensities and thus the contrasts around the vertices to be identified. In order to identify $T_2$, we choose the vector of intensities $\underline f= \underline f^{(4)} = (10, 5, 15, 30,20,0, 25)^T$ and for identifying $T_4$, we choose $\underline f= \underline f^{(5)} = (10, 15, 20, 30,5,0,25)^T$. Figure~\ref{fig_occluded_T1T2f4_T4f5}(a) shows that, again using deg(P)=$0^\circ$, deg(Q)=$90^\circ$, deg(R)=$180^\circ$, vertex $T_2$ can be identified since the global minimum of the second-order topological derivative is located exactly in its neighboring pixels. In Fig.~\ref{fig_occluded_T1T2f4_T4f5}(b), we again choose deg(P)=$45^\circ$, deg(Q)=$225^\circ$, deg(R)=$270^\circ$ to identify $T_4$. While the location of the globally minimal value is still not at vertex $T_4$, it now is visible as one of the local minima of the second-order topological derivative. 
\ifpics
\begin{figure}
\begin{tabular}{cc}%cccccs
    \begin{tikzpicture}
    \node[anchor=south west,inner sep=0] at (-1.5,-0.5) {\includegraphics[scale=.23]{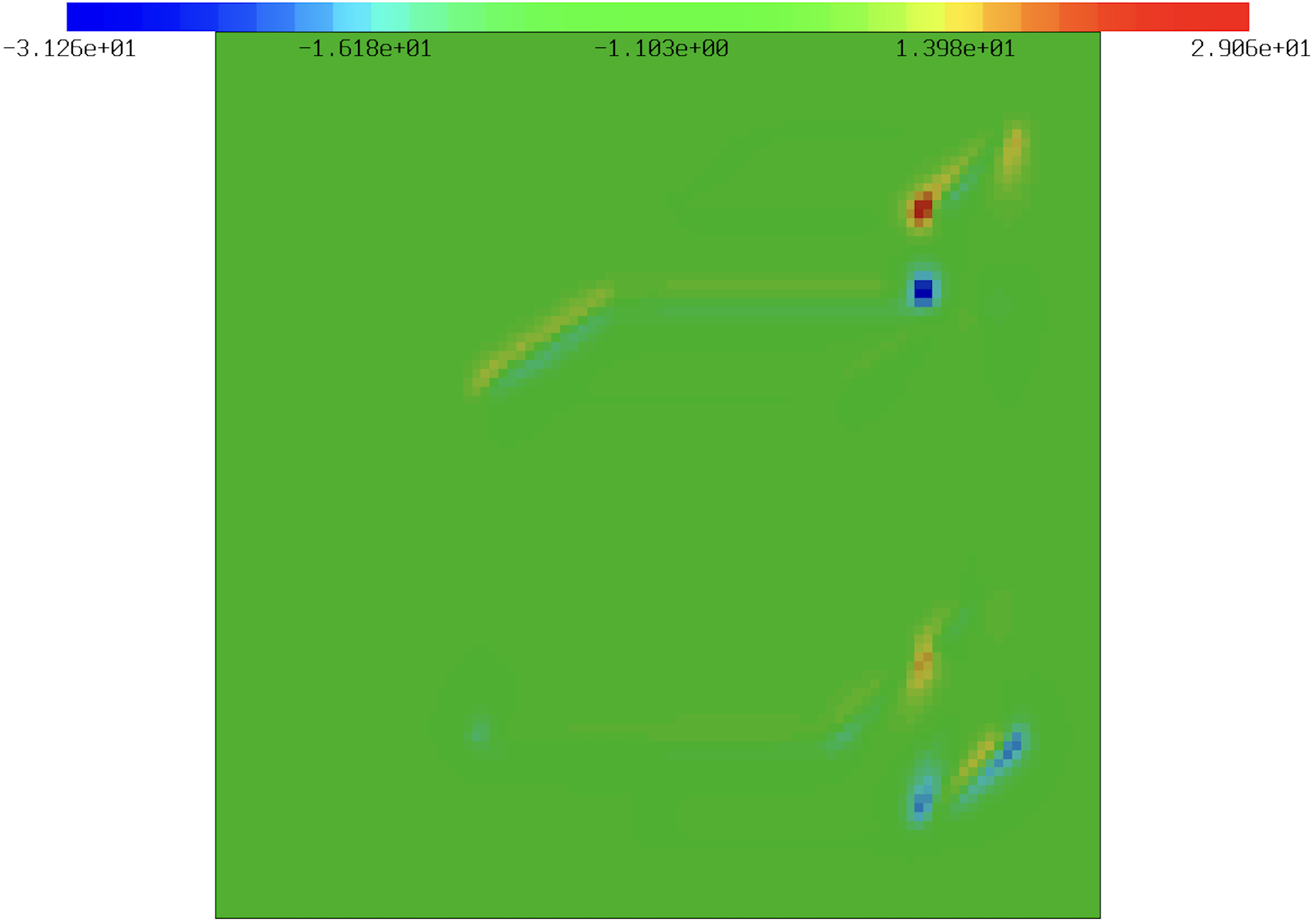}};
    \draw[gray(x11gray), ultra thick, fill] (0.,0.) circle (1.2);
    \draw[line width=3pt, black] (0,0) -- (1,0) (0,0) -- (0,1) (0,0) -- (-1,0); 
    \end{tikzpicture}
    &
    \begin{tikzpicture}
    \node[anchor=south west,inner sep=0] at (-1.5,-0.5) {\includegraphics[scale=.22]{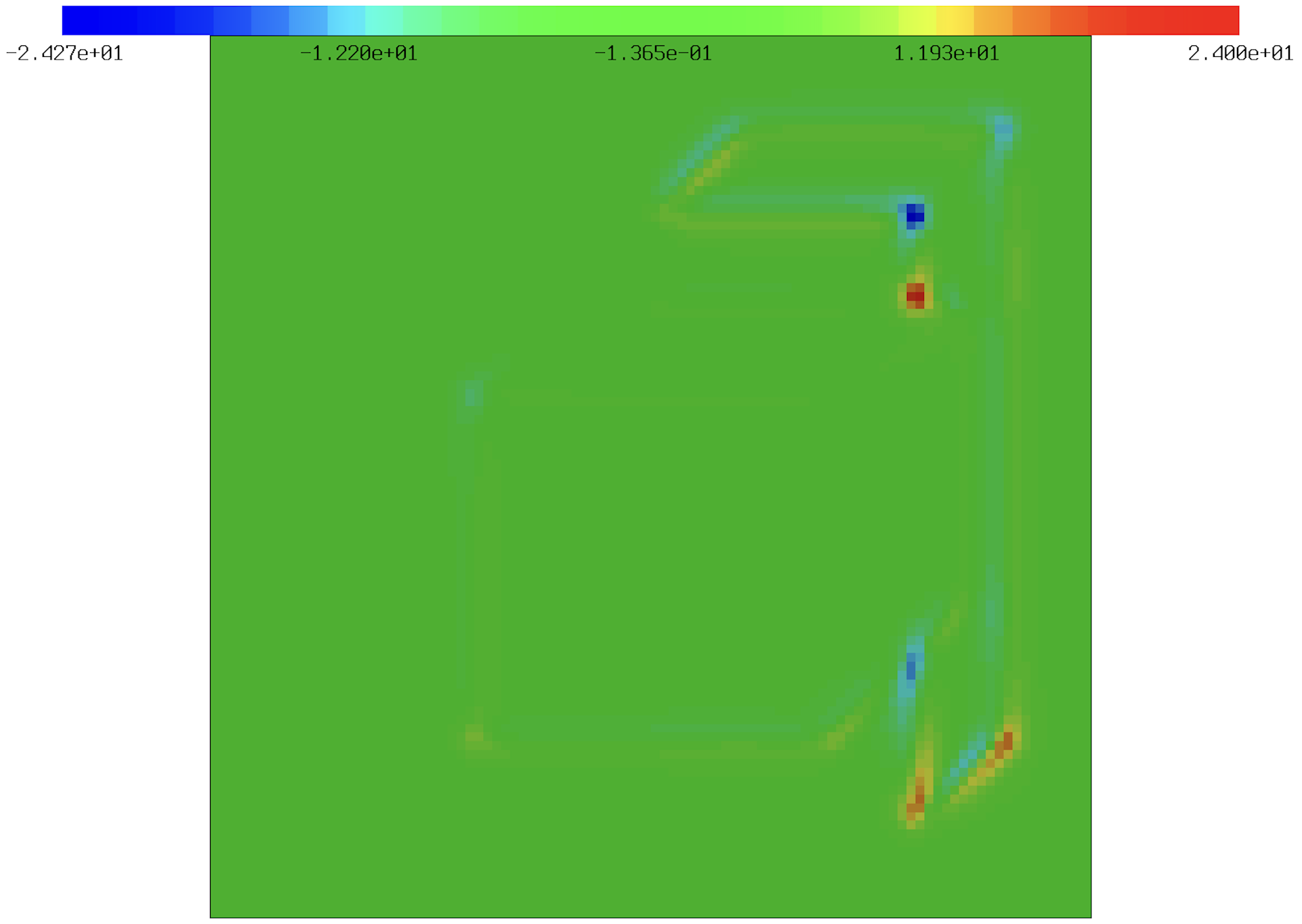}};
    \draw[gray(x11gray), ultra thick, fill] (0.,0.) circle (1.2);
    \draw[line width=3pt, black] (0,0) -- (1/1.41421356237,1/1.41421356237) (0,0) -- (-1/1.41421356237,-1/1.41421356237) (0,0) -- (0,-1); 
    \end{tikzpicture} \\
    (a) & (b)
\end{tabular}
\caption{Second-order topological derivative $d^2 \mathcal J(\Omega)(\cdot, \omega)$ and indicated inclusion shapes $\omega$: (a) $\underline f= \underline f^{(4)} = (10, 5, 15, 30,20,0, 25)^T$, deg(P)=$0^\circ$, deg(Q)=$90^\circ$, deg(R)=$180^\circ$, attempting to identify triple junctions $T_1$, $T_2$. (b) $\underline f= \underline f^{(5)} = (10, 20, 15, 30,5,0,25)^T$, deg(P)=$45^\circ$, deg(Q)=$225^\circ$, deg(R)=$270^\circ$ to identify triple junction $T_4$.}
\label{fig_occluded_T1T2f4_T4f5}
\end{figure}
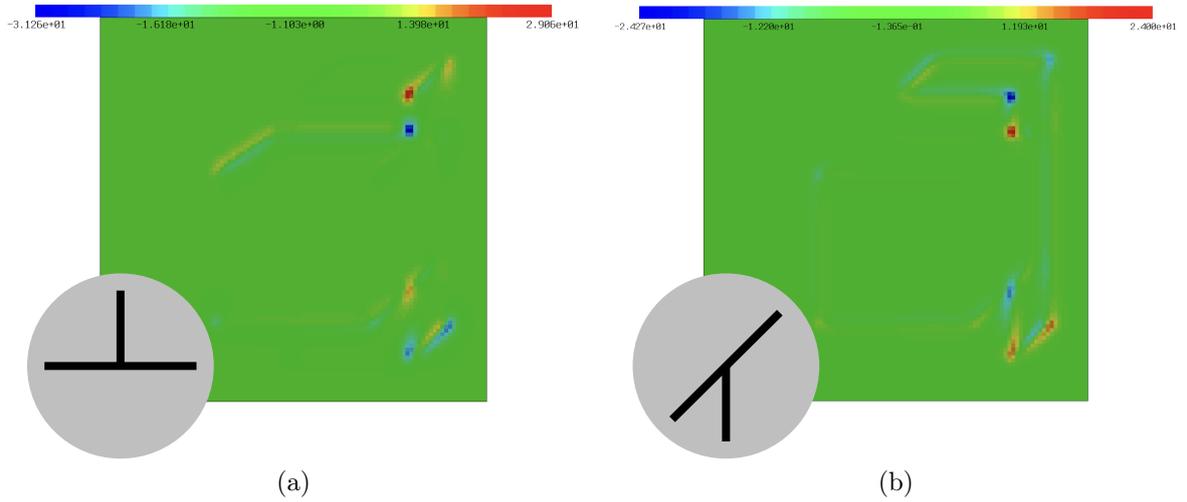
\fi 

%%%%%%%%%%%%%%%%%%%%%%%%%%%%%%%%%
\section{Conclusion}
\label{conclusion}
In this work, we introduced a vertex detection method based on topological asymptotic analysis. We derived the second-order topological derivative of a Mumford-Shah-type functional with respect to arbitrary inclusion shapes, specifically focusing on vertices. We also discussed its numerical computation for certain polygonal inclusion shapes corresponding to various vertex classes. The effectiveness of the proposed one-shot detection technique was demonstrated through two examples, where accurate vertex localization was confirmed across different classes and angles. These examples illustrate well-known concepts in optical vision by applying topological derivatives. 

%%%%%%%%%%%%%%%%%%%%%%%%%%%%%%%%
\appendix
\subsection*{Acknowledgements}
%
% ADAPT (remove what doesn't apply)
%
% NIPSUM
This research was funded in whole, or in part, by the Austrian Science
Fund (FWF) P 34981. For the purpose of open access, the author has applied
a CC BY public copyright licence to any Author Accepted Manuscript version arising
from this submission.
% SFB (for example Otmar's 07 Project)
Moreover, O.S. is supported by the Austrian Science Fund (FWF),
with SFB F68, project F6807-N36 (Tomography with Uncertainties).
% SFB (for example Peters's 04 Project)
This work is partially supported by the joint DFG/FWF Collaborative Research Centre CREATOR (FWF Project DOI 10.55776/F90) at TU Darmstadt, TU Graz and JKU/RICAM Linz. P.G. is partially supported by the State of Upper Austria.
% ?? is supported by the Austrian Science Fund (FWF),
% with SFB F68, project F6804-N36 (Quantitative Coupled Physics Imaging).
% CD-LAB MaMSi
The financial support by the Austrian Federal Ministry for Digital and Economic Affairs, the National Foundation for Research, Technology and Development and the Christian Doppler
Research Association is gratefully acknowledged.
%%%%%%%%%%%%%%%%%%%%%%%%%%%%%%
% \section*{References}
%%%%%%%%%%%%%%%%%%%%%%%%%%%%%%
% \printbibliography
\input{reference.bbl}

%%%%%%%%%%%%%%%%%%%%%%%%%%%%%%%%
\end{document}

%% file: header.tex
\pdftrailerid{}
\usepackage[utf8]{inputenc}
\usepackage[T1]{fontenc}
\usepackage{amsmath,amssymb,dsfont}
\numberwithin{equation}{section}
\usepackage{microtype}
\usepackage{graphicx,tikz,pgfplots}
\graphicspath{{images/}}
\pgfplotsset{compat=newest}
\usepackage[hyperref,amsmath,thmmarks]{ntheorem}
\usepackage{aliascnt}
\usepackage[a4paper,centering,bindingoffset=0cm,marginpar=2cm,margin=2.5cm]{geometry}
\usepackage[pagestyles]{titlesec}
\usepackage[font=footnotesize,format=plain,labelfont=sc,textfont=sl,width=0.75\textwidth,labelsep=period]{caption}

\usepackage{xcolor,colortbl}
\definecolor{ao(english)}{rgb}{0.0, 0.5, 0.0}
\usepackage[pdftex,colorlinks=true,linkcolor=blue,citecolor=ao(english),urlcolor=blue,bookmarks=true,bookmarksnumbered=true]{hyperref}
\hypersetup
{
    bookmarksnumbered,
    breaklinks,
    pdfborderstyle=,
    colorlinks,
    citecolor=[rgb]{0,.5,0},
    linkcolor=[rgb]{0,0,.5},
    urlcolor=[rgb]{0,0,.5},
    pdfauthor={P. Gangl, B. Mejri and O. Scherzer},%ADAPT
    pdfsubject={Subject},%ADAPT
    pdftitle={Vertex characterization via Second-Order Topological Derivatives},%ADAPT
    pdfkeywords={Vertex characterization, second-order method, topological derivative}%ADAPT
}

\newpagestyle{headers}
{
	\headrule
	\sethead[\footnotesize\thepage][\footnotesize\sc P. Gangl, B. Mejri and O. Scherzer][]{}{\footnotesize\sc Vertex characterization via Second-Order Topological Derivatives}{\footnotesize\thepage}
	\setfoot{}{}{}
}
\newtheorem{lemma}{Lemma}[section]

\newaliascnt{proposition}{lemma}

\aliascntresetthe{proposition}

\newaliascnt{corollary}{lemma}

\aliascntresetthe{corollary}

\newaliascnt{theorem}{lemma}
\newtheorem{theorem}[theorem]{Theorem}
\aliascntresetthe{theorem}

\theorembodyfont{\normalfont}
\newaliascnt{definition}{lemma}
\newtheorem{definition}[definition]{Definition}
\aliascntresetthe{definition}

\newaliascnt{assumption}{lemma}

\aliascntresetthe{assumption}

\theoremstyle{nonumberplain}
\theoremseparator{:}
\theoremheaderfont{\normalfont\itshape}

\newtheorem{remark}{Remark}

\theoremsymbol{\ensuremath{\square}}
\newtheorem{proof}{Proof}

\newcommand{\R}{\mathds{R}}

\let\RE\Re
\let\Re=\undefined
\DeclareMathOperator{\Re}{\RE e}
\let\IM\Im
\let\Im=\undefined
\DeclareMathOperator{\Im}{\IM m}

%% file: reference.bbl
\newcommand{\etalchar}[1]{$^{#1}$}
\providecommand{\bysame}{\leavevmode\hbox to3em{\hrulefill}\thinspace}
\providecommand{\MR}{\relax\ifhmode\unskip\space\fi MR }
% \MRhref is called by the amsart/book/proc definition of \MR.
\providecommand{\MRhref}[2]{%
  \href{http://www.ams.org/mathscinet-getitem?mr=#1}{#2}
}
\providecommand{\href}[2]{#2}